\documentclass[reqno]{amsart}
\usepackage{xr}
\usepackage{hyperref}
\usepackage{graphicx}
\usepackage[usenames,dvipsnames,svgnames,table]{xcolor}
\usepackage{epstopdf}
\usepackage{booktabs}
\usepackage{amssymb,amsmath}
\usepackage{amsthm}
\usepackage{enumerate}
\newtheorem{theorem}{Theorem}[section]
\newtheorem{proposition}[theorem]{Proposition}
\newtheorem{lemma}[theorem]{Lemma}

\newtheorem{remark}[theorem]{Remark}
\theoremstyle{definition}
\newtheorem{example}[theorem]{Example}
\newtheorem{definition}[theorem]{Definition}
\numberwithin{equation}{section}

\newcommand{\C}{{\mathbb C}}
\newcommand{\R}{{\mathbb R}}
\newcommand{\Z}{{\mathbb Z}}
\newcommand{\N}{{\mathbb N}}
\newcommand{\Q}{{\mathbb Q}}

\newcommand{\CB}{{\mathcal B}}

\newcommand{\CJ}{{\mathcal J}}
\def\BjoernerLovasz(#1,#2){\Pi_{N,n}}
\newcommand{\CL}{{\mathcal L}}
\newcommand{\CM}{{\mathcal M}}
\newcommand{\CP}{{\mathcal P}}

\newcommand{\Barvi}{\mathrm{Barvinok}}
\newcommand{\ConeByCone}{\mathrm{cone\text{-}by\text{-}cone}}
\newcommand{\Barvinok}{{E^{k, \Barvi}}}

\newcommand{\retroS}{{M}}
\newcommand{\polypp}{{\mathcal Q}}

\newcommand{\CR}{{\mathcal R}}

\newcommand{\CV}{{\mathcal V}}
\newcommand{\la}{{\langle}}
\newcommand{\ra}{{\rangle}}
\newcommand{\e}{{\mathrm{e}}}

\newcommand{\lin}{\operatorname{lin}}

\newcommand{\vol}{\operatorname{vol}}
\def\binomial(#1,#2){\binom{#1}{#2}}
\def\mult(#1,#2){\binom{#1}{#2}}

\renewcommand{\c}{{\mathfrak{c}}}
\renewcommand{\d}{{\mathfrak{d}}}
\newcommand{\f}{{\mathfrak{f}}}
\renewcommand{\t}{{\mathfrak{t}}}

\newcommand{\p}{{\mathfrak{p}}}
\newcommand{\pp}{{\mathfrak{p}^{\mathrm{partition}}}}

\renewcommand{\u}{{\mathfrak{u}}}
\newcommand{\lattice}{\Lambda}

\newcommand{\Jposet}[2]{\mathcal J^{#1}_{\geq #2}}
\newcommand{\Moebius}{\textup{M\"obius}}

\makeatletter
\let\OurMathBbAux=\mathbb
\DeclareRobustCommand\OurMathBb{\OurMathBbAux}
\let\mathbb=\OurMathBb
\let\bfseries=\undefined
\DeclareRobustCommand\bfseries
{\not@math@alphabet\bfseries\mathbf
  \boldmath\fontseries\bfdefault\selectfont\let\OurMathBbAux=\mathbf}
\def\@thm#1#2#3{%
  \ifhmode\unskip\unskip\par\fi
  \normalfont
  \trivlist
  \let\thmheadnl\relax
  \let\thm@swap\@gobble
  \thm@notefont{\fontseries\mddefault\upshape\unboldmath}
  \thm@headpunct{.}
  \thm@headsep 5\p@ plus\p@ minus\p@\relax
  \thm@space@setup
  #1
  \@topsep \thm@preskip               
  \@topsepadd \thm@postskip           
  \def\@tempa{#2}\ifx\@empty\@tempa
    \def\@tempa{\@oparg{\@begintheorem{#3}{}}[]}%
  \else
    \refstepcounter{#2}%
    \def\@tempa{\@oparg{\@begintheorem{#3}{\csname the#2\endcsname}}[]}%
  \fi
  \@tempa
}
\makeatother

\newcommand{\tmagenta}[1]{}
\newcommand{\tgreen}[1]{}
\newcommand{\tred}[1]{}
\newcommand{\tblue}[1]{}

\title{Three Ehrhart quasi-polynomials}
\author{V. Baldoni}
\address{Velleda Baldoni: Dipartimento di Matematica, Universit\`a degli studi di  Roma ``Tor Vergata'',
Via della ricerca scientifica 1, I-00133, Italy}
\email{baldoni@mat.uniroma2.it}
\author{N. Berline}
\address{Nicole Berline:  \'Ecole Polytechnique, Centre de Math\'ematiques Laurent Schwartz, 91128 Palaiseau Cedex, France}
\email{Nicole.Berline@math.cnrs.fr}
\author{J. A. De Loera}
\address{Jes\'us A. De Loera:  Department of
 Mathematics, University of California,
 Davis, One Shields Avenue, Davis, CA, 95616, USA}
\email{deloera@math.ucdavis.edu}
\author{M. K\"oppe}
\address{Matthias~K\"oppe:  Department of
 Mathematics, University of California,
 Davis, One Shields Avenue, Davis, CA, 95616, USA}
\email{mkoeppe@math.ucdavis.edu}
\author{M. Vergne}
\address{Mich\`ele Vergne: Universit\'e Paris 7 Diderot, Institut Math\'ematique de
Jussieu, Sophie Germain, case 75205, Paris Cedex 13} \email{michele.vergne@imj-prg.fr}

\begin{document}
\maketitle{}

\begin{abstract}
Let  $\p(b)\subset \R^d$ be a semi-rational parametric polytope, where $b=(b_j)\in \R^N$ is a real multi-parameter.
We study intermediate sums of polynomial functions $h(x)$ on $\p(b)$,
\begin{equation*}\label{eq:abstract}
 S^L (\p(b),h)=\sum_{y}\int_{\p(b)\cap (y+L)} h(x)\,\mathrm dx,
\end{equation*}
where we integrate over  the intersections of $\p(b)$  with  the subspaces
parallel to a fixed rational subspace $L$ through all lattice points, and sum
the integrals.  The purely discrete sum is of course a particular case
($L=0$), so $S^0(\p(b), 1)$ counts the integer points in the parametric
polytopes.

The chambers are  the open conical subsets of $\R^N$ such that the shape of $\p(b)$ does not change when $b$ runs over a chamber. We first prove that on every chamber of $\R^N$, $ S^L (\p(b),h)$ is given by a
quasi-polynomial function of $b\in \R^N$. A key point of our paper 
is an analysis of the interplay between two notions of degree on quasi-polynomials: the usual polynomial degree and
a filtration, called the local degree.

Then, for a fixed $k\leq d$, we consider a particular linear combination of
such intermediate weighted sums,  which was introduced by Barvinok in order to
compute efficiently the $k+1$ highest coefficients of the Ehrhart
quasi-polynomial which gives the number of points of a dilated rational
polytope. Thus, for each chamber, we obtain  a  quasi-polynomial function of
$b$, which we call \emph{Barvinok's patched quasi-polyno\-mial} (at codimension level
$k$).

Finally, for each chamber, we introduce a new quasi-polyno\-mial function of
$b$, the \emph{cone-by-cone patched quasi-polyno\-mial} (at codimension level $k$),
defined in a refined way by linear combinations of intermediate generating
functions for the cones at vertices of $\p(b)$.

We prove that both patched quasi-poly\-nomials agree with
the discrete weighted sum $b\mapsto   S^{\{0\}}(\p(b),h)$ in the
terms corresponding to the $k+1$ highest polynomial degrees.
\end{abstract}

\begin{figure}
\begin{center}
 \includegraphics[width=4cm]{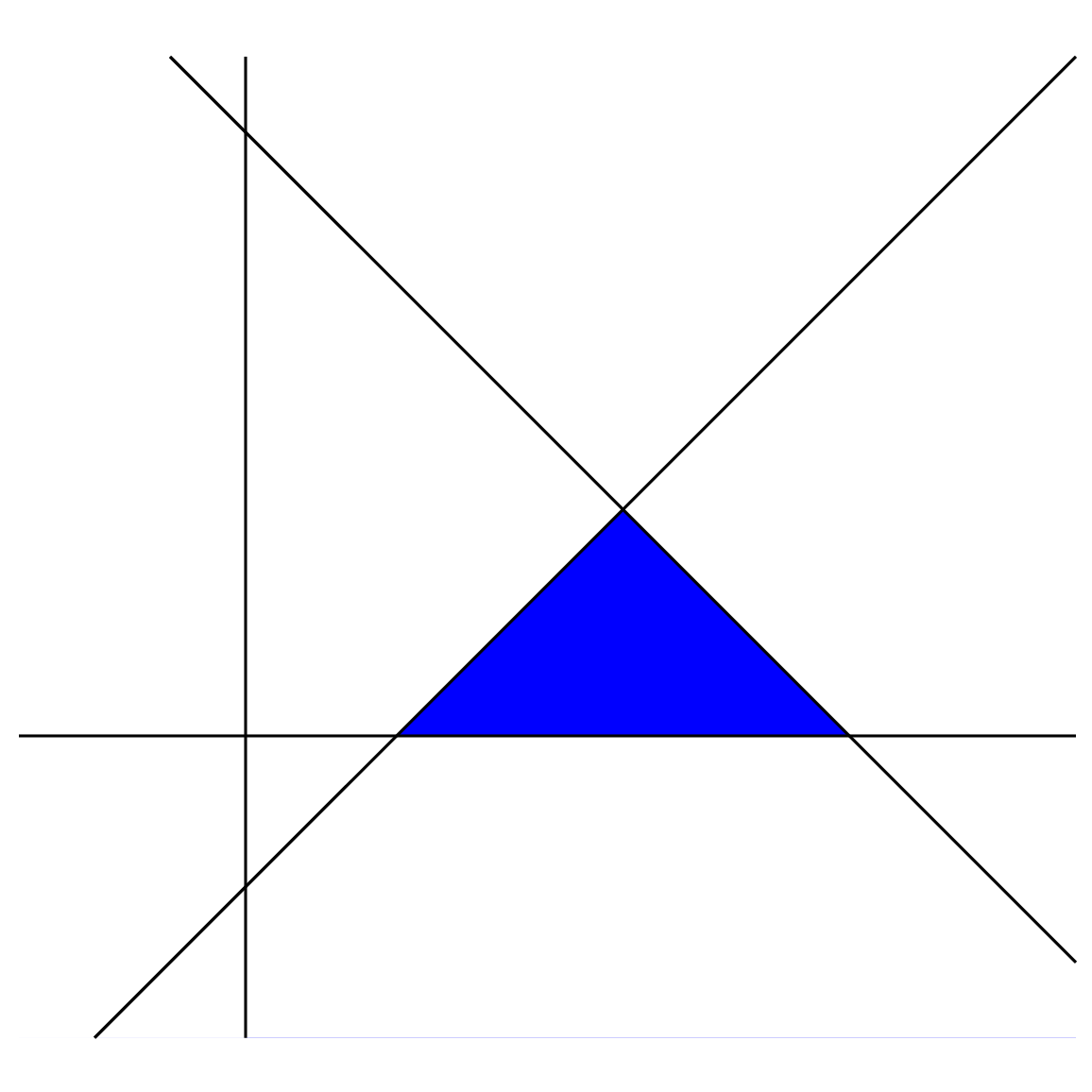}\
 \ \includegraphics[width=4cm]{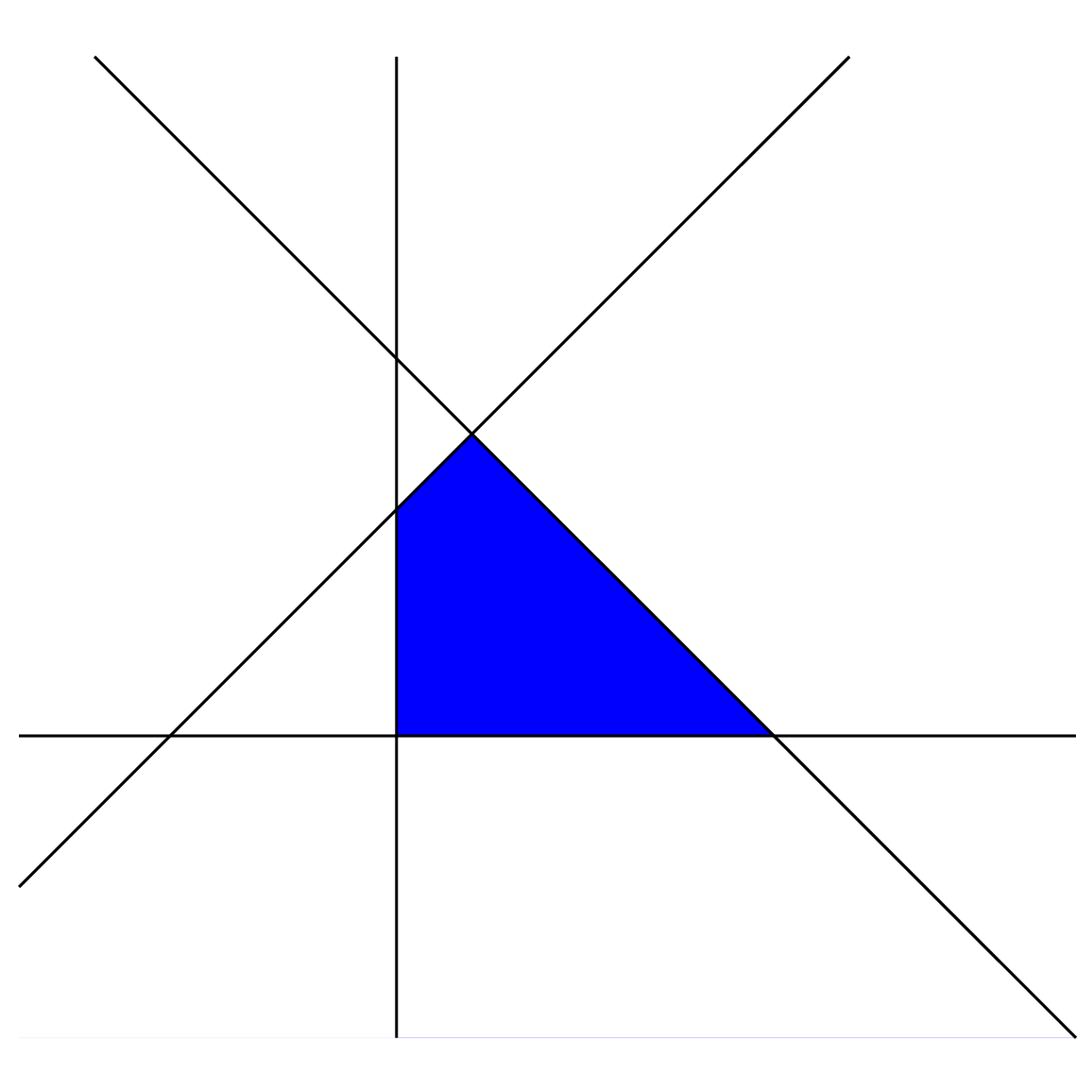}\ \ \includegraphics[width=4cm]{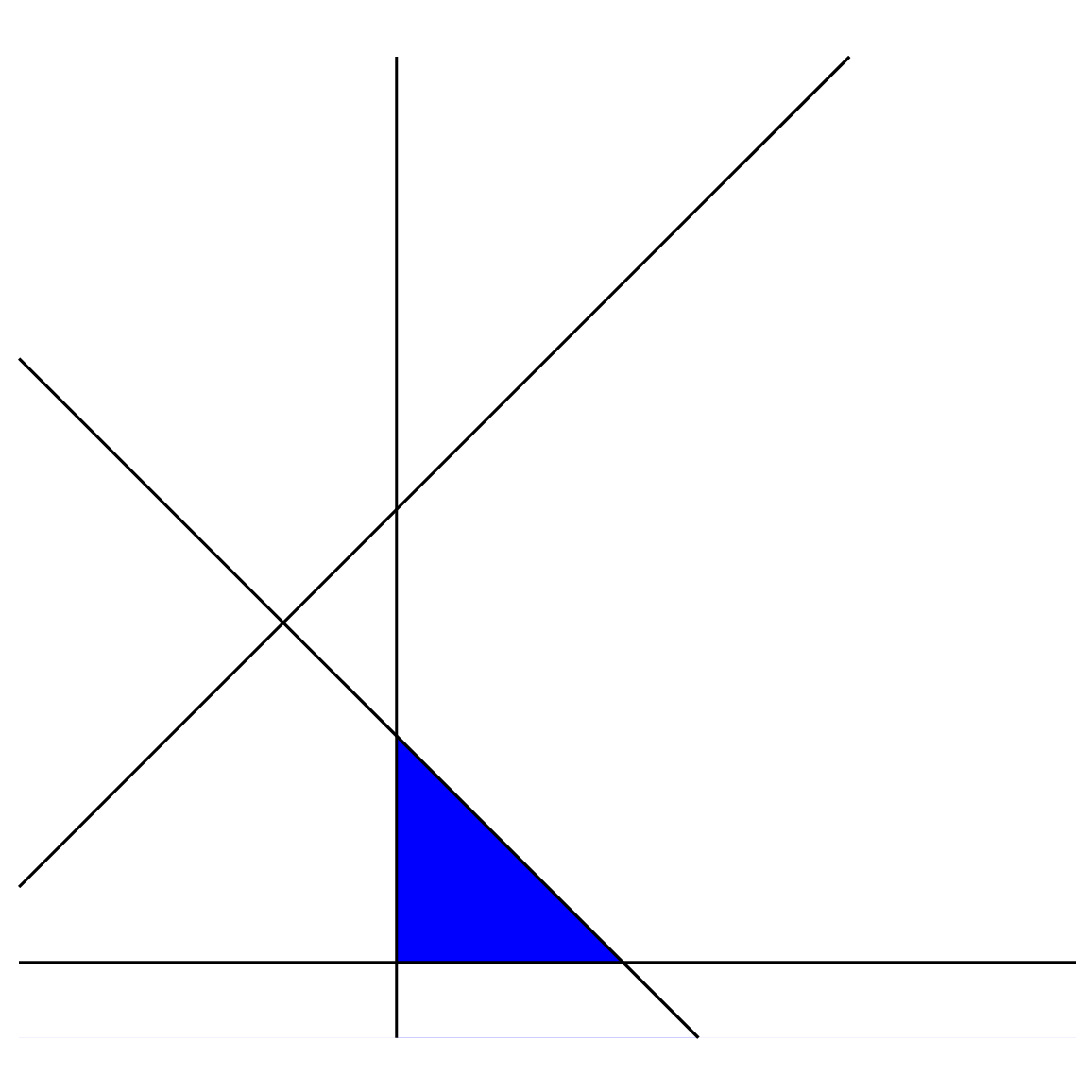}
 \caption{The parametric polytope $\p(b)$ from Example
   \ref{ex:parametric-dim2}, for $b$ in various
   chambers}\label{fig:parametric-dim2again}
 \end{center}
\end{figure}

\clearpage
{
\tableofcontents
}
\newpage

\tred{My questions appear in red.}
\tmagenta{To-do items for myself (no input needed) are magenta.}
\tgreen{Resolved questions, or mere change notes, appear in green.}
\tblue{Answers to my questions appear in blue.}
\tgreen{Many old comments have been moved to file \texttt{3-Ehrhart-polynomials-old-comments}.}

\section{Introduction}
In this article, a \emph{parametric semi-rational polytope} $\p(b)\subset \R^d$ is defined by inequalities:
\begin{equation}\label{def:pb}
\p(b)=\bigl\{\,x\in \R^d: \langle \alpha_j, x\rangle \leq b_j,\; j=1,\ldots, N\,\bigr\}
\end{equation}
where $\alpha_1,\alpha_2,\ldots,\alpha_N$ are {\em fixed} linear forms with integer
coefficients (the case of rational $\alpha_j$ can be treated
by rescaling $\alpha_j$ and $b$)
and the parameter $b=(b_1,b_2,\ldots,b_N)$ varies in $\R^N$.
The shape of the polytope $\p(b)$ varies when the parameter $b$ varies (see Figure \ref{fig:parametric-dim2again}).
\emph{Chambers} $\tau\subset \R^N$ are open convex polyhedral  cones such that the shape of $\p(b)$ does not change
when $b$ runs over $\tau$ (see Definition \ref{def:chamber}).
We consider weighted integrals and sums, where the weight is a polynomial
function $h(x)$ of degree~$m$ on  $\R^d$.
 $$
 I(\p(b),h)=\int_{\p(b)} h(x) \,\mathrm dx, \;\;\;\;     S(\p(b),h)=\sum_{x\in \p(b)\cap \Z^d} h(x).
 $$
When the weight is the constant $1$,  then $I(\p(b),1)$ is  the volume of $\p(b)$, while
 $S(\p(b),1)$ is  the number of integral points in $\p(b)$.

As introduced by Barvinok, we also study intermediate sums associated to a  rational subspace $L$:
\begin{equation}
 S^L(\p(b),h)=\sum_{y} \int_{\p(b)\cap (y+L)} h(x)\,\mathrm dx.
\end{equation}
Here, we integrate over  the intersections of $\p(b)$  with  the subspaces parallel to a fixed rational subspace $L$ through all lattice points, and sum the integrals.

The unweighted case ($h=1$), the study of the counting function~$S(\p(b),1)$
is, of course, very important in algebraic combinatorics.  Polytopes depending
on multiple parameters have appeared, for example, in the celebrated
Knutson--Tao honeycomb model \cite{MR1671451}.  Also the classical vector
partition functions \cite{Brion1997residue} appear as a special case.
However, a large part of the literature has focused on the case of
one-parameter families of dilations of a single polytope (see our discussion
on Ehrhart theory in section~\ref{s:intro-single-parameter} below), with few exceptions
\cite{koeppe-verdoolaege:parametric,beck:multidimensional-reciprocity,HenkLinke}.
Indeed \cite{HenkLinke} was part of our motivation to consider the case of a
real multi-parameter and not just one-parameter dilations as in our previous
articles on the subject.  Our interest in the general problem $S(\p(b),h)$ is
motivated in part by the important applications in compiler optimization and
automatic code parallelization, in which multiple parameters arise naturally
(see \cite{Clauss1998parametric,Verdoolaege2005PhD,Verdoolaege2007parametric}
and the references within).
For a broader context of analytic combinatorics, we refer to
\cite{Pemantle:2013:ACS:2505450}. 

The relations between the two functions $I(\p(b),h)$ and $S(\p(b),h)$  of the
parameter vector $b$
have been the central theme of several works. In this article,
we (hope to) add a contribution to these questions.

We introduce the new notion  of \emph{local degree}, which we believe is important.
A function $b\mapsto f(b)$ of the real multi-parameter $b$ is of local degree
(at most) $\ell$  if it can be expressed as a linear combination
of products of  a number less or equal to $\ell$ of  step-linear forms of $b$
and linear forms of $b$ (see Definition~\ref{def:step-poly-V} below and
Figure~\ref{fig:irrational-rectangle}, left). If the number of linear forms is less than or
equal to~$q$, we say that
$f$ is of \emph{polynomial degree} (at most) $q$.


\bigbreak

 The present article is the culmination of a study based on \cite{SLII2014,so-called-paper-2}. These two articles were devoted to the properties of intermediate generating functions  only for  polyhedral \emph{cones}. Here, using the Brianchon--Gram set-theoretic decomposition of a polytope as a signed sum of its supporting cones, we study the function $b\mapsto S^L(\p(b),h)$.

We show first that, on each chamber, the function $b\mapsto S^L(\p(b),h)$ is
of local degree at most $d+m$.  In particular its term of polynomial degree $0$ is expressed as a linear combination of at most $d+m$ step-linear functions of $b$.

Then we study the terms of highest polynomial degree of $S(\p(b),h)$
on each chamber.

Given a fixed integer $k\leq d$
 we construct two quasi-polynomials,
 \emph{Barvinok's patched quasi-polyno\-mial} (at  level
$k$) and the \emph{cone-by-cone patched quasi-polyno\-mial}
(at level $k$).
The two   constructions  use linear combinations of intermediate sums associated to rational subspaces $L$ of codimension less or equal to $k$.
 The first one is due to Barvinok  \cite{barvinok-2006-ehrhart-quasipolynomial}.
  We give a more streamlined proof of Barvinok's Theorem 1.3 in  \cite{barvinok-2006-ehrhart-quasipolynomial}  and a more explicit formula  for it when
$\p(b)$ is a simplex.
 The cone-by-cone patched quasi-polyno\-mial is a new construction.
We prove that both patched quasi-poly\-nomials agree with
the discrete weighted sum $b\mapsto   S(\p(b),h)$ in the
terms corresponding to the $k+1$ highest polynomial degrees $d+m, d+m-1,\ldots, d+m-k$.

\smallbreak

We now give more details on the content of this article.

\subsection{Weighted Ehrhart quasi-polynomials and intermediate sums}\label{s:intro-single-parameter}

When a rational parameter vector~$b$ is fixed, then the polytope $\p=\p(b)$ is a rational polytope. If we dilate it by a non-negative number~$t$, the function $t\mapsto S(t\p,h)$ is a
quasi-polynomial function of $t$, i.e., it takes the form
$$ S(t\p,h) = E(t) = \sum_{j=0}^{d+m} E_j(t) t^j, $$
where the coefficients $E_j(t)$ are periodic functions of~$t$, rather than constants.
It is called the \emph{weighted Ehrhart quasi-polynomial} of $\p$.
In traditional Ehrhart theory, only non-negative \emph{integer} dilation
factors~$t$ are considered, and so a coefficient function with period~$q\in\Z_{>0}$ can
be given as a list of $q$~values, one for each residue class modulo~$q$.
However, the approach to computing Ehrhart quasi-polynomials via
generating functions of parametric polyhedra
\cite{Verdoolaege2007parametric,Verdoolaege2005PhD,koeppe-verdoolaege:parametric},
which we follow in the present paper, leads to a natural, shorter representation of the coefficient functions as closed-form formulas (so-called
\emph{step-polynomials}) of the dilation parameter~$t$, using the
``fractional part'' function.  These closed-form formulas are naturally valid for
arbitrary non-negative \emph{real} dilation
parameters~$t$, as well as any real (not just rational) parameter $b$. 
This fact was implicit in the computational works following this
method \cite{Verdoolaege2007parametric,Verdoolaege2005PhD}, and was made
explicit in \cite{koeppe-verdoolaege:parametric}.
The resulting \emph{real Ehrhart theory} has recently caught the interest of other
authors \cite{linke:rational-ehrhart,HenkLinke}; see also~\cite{so-called-paper-2}.
\tgreen{Updated previous paragraph according to Michele's comment 2014-08-30.} 

The highest ``expected'' degree term of the weighted Ehrhart quasi-polyno\-mial is
$I(\p,h) t^{d+m}$, if $h(x)$ is homogeneous of degree $m$; of course, this
term may vanish, as the example $\p =[-1,1]$, $h(x) = x$ illustrates.
\tgreen{(Thanks Michele for pointing this out. I have rephrased using
  ``expected'' degree -- I think this makes the next sentence more
  transparent. --Matthias)}
For a study of the coefficients  of degree $d+m$, $d+m-1$, \dots, $d+m-k$ of the quasi-polynomial $S(t\p,h)$,
a key tool introduced  by Barvinok
(in \cite{barvinok-2006-ehrhart-quasipolynomial}, 
for the unweighted case $h=1$) is the \emph{intermediate weighted sum} $S^L(\p,h)$,
where $L$ is a rational subspace of~$V=\R^d$:
\begin{equation}\label{eq:abstract-first}
 S^L (\p,h)=\sum_{y}\int_{\p\cap (y+L)} h(x)\,\mathrm dx,
\end{equation}
where the summation variable $y$ runs over the projected lattice in $V/L$.  The polytope $\p$ is sliced by subspaces parallel to $L$ through lattice points
and  the integrals of $h$ over the slices are added (see Figure~\ref{fig:parametric-dim2againagain}).
When $L=V$, $S^L(\p,h)$ is just the integral $I(\p,h)$, while for $L=\{0\}$, we recover the discrete sum $S(\p,h)$.
In the present study, we generalize Barvinok's ideas in several ways, building
on our previous work in \cite{so-called-paper-1,so-called-paper-2,SLII2014}.
\begin{figure}
\begin{center}
  \includegraphics[width=6cm]{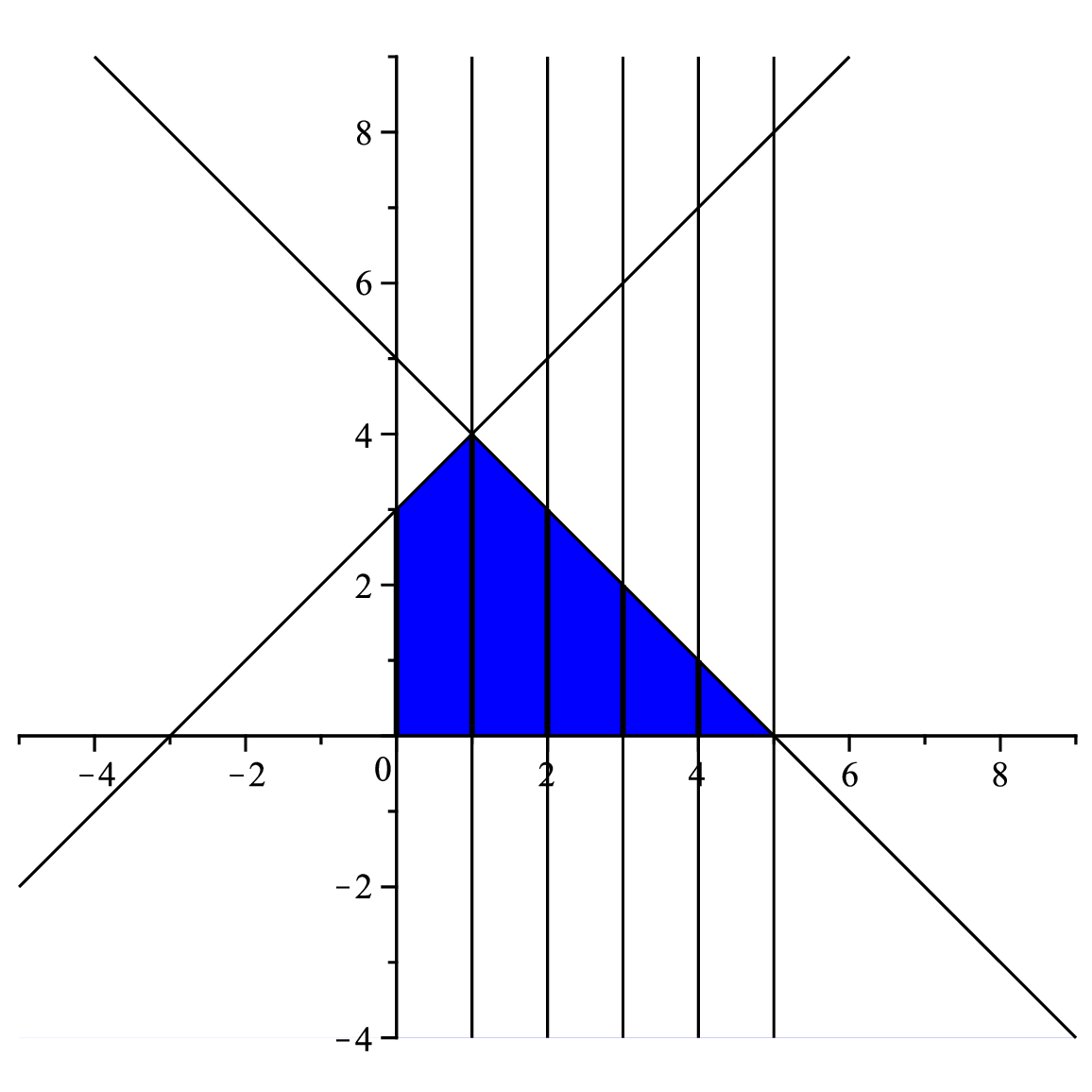}
  \caption{Intermediate sum over a polytope $\p$ (blue).  We sum the integrals
    over the slices of~$p$ parallel to~$L$ going through lattice points
    (vertical lines).
}\label{fig:parametric-dim2againagain}
 \end{center}
\end{figure}
\begin{figure}
\begin{center}
 \includegraphics[width=6cm]{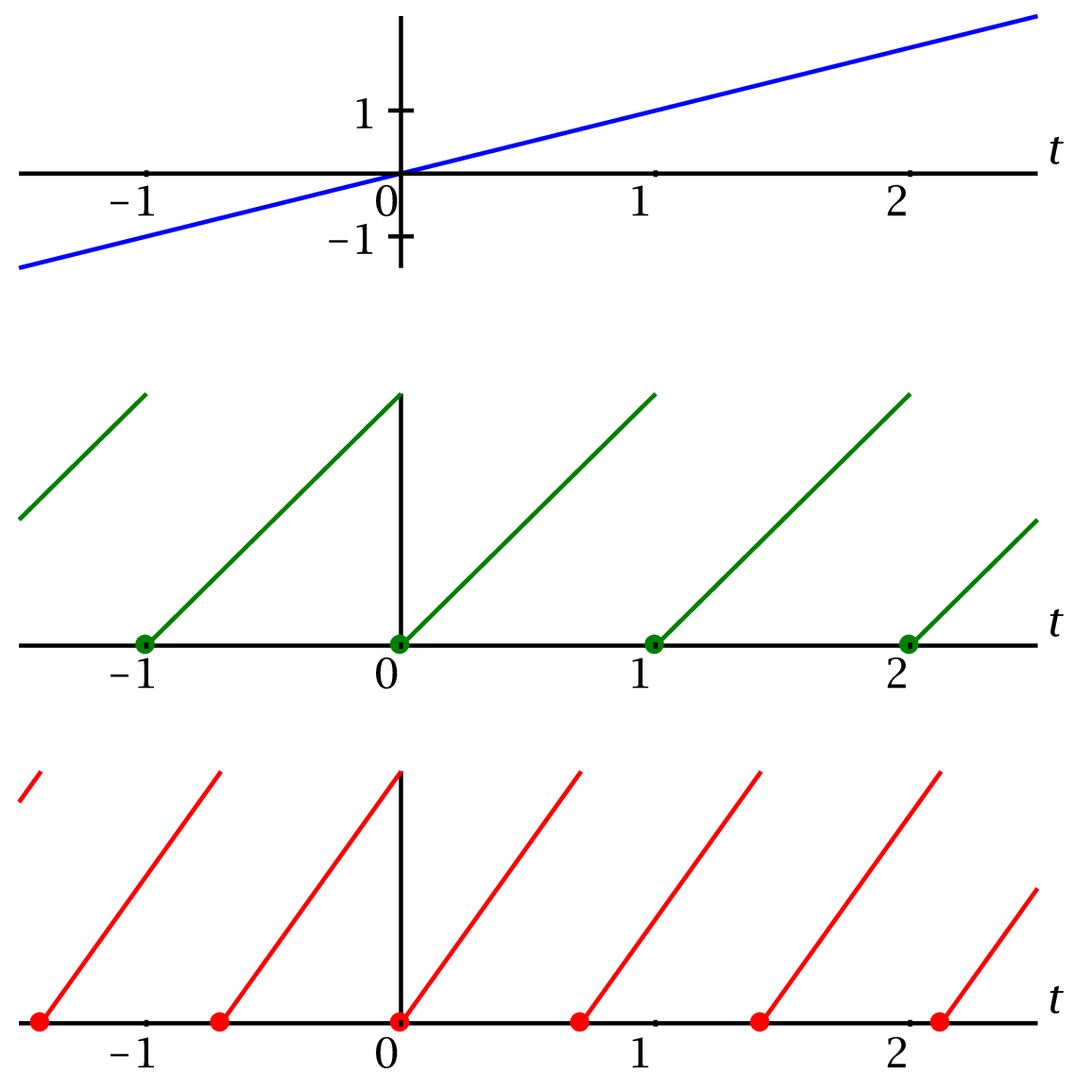}\quad
 \includegraphics[width=6cm]{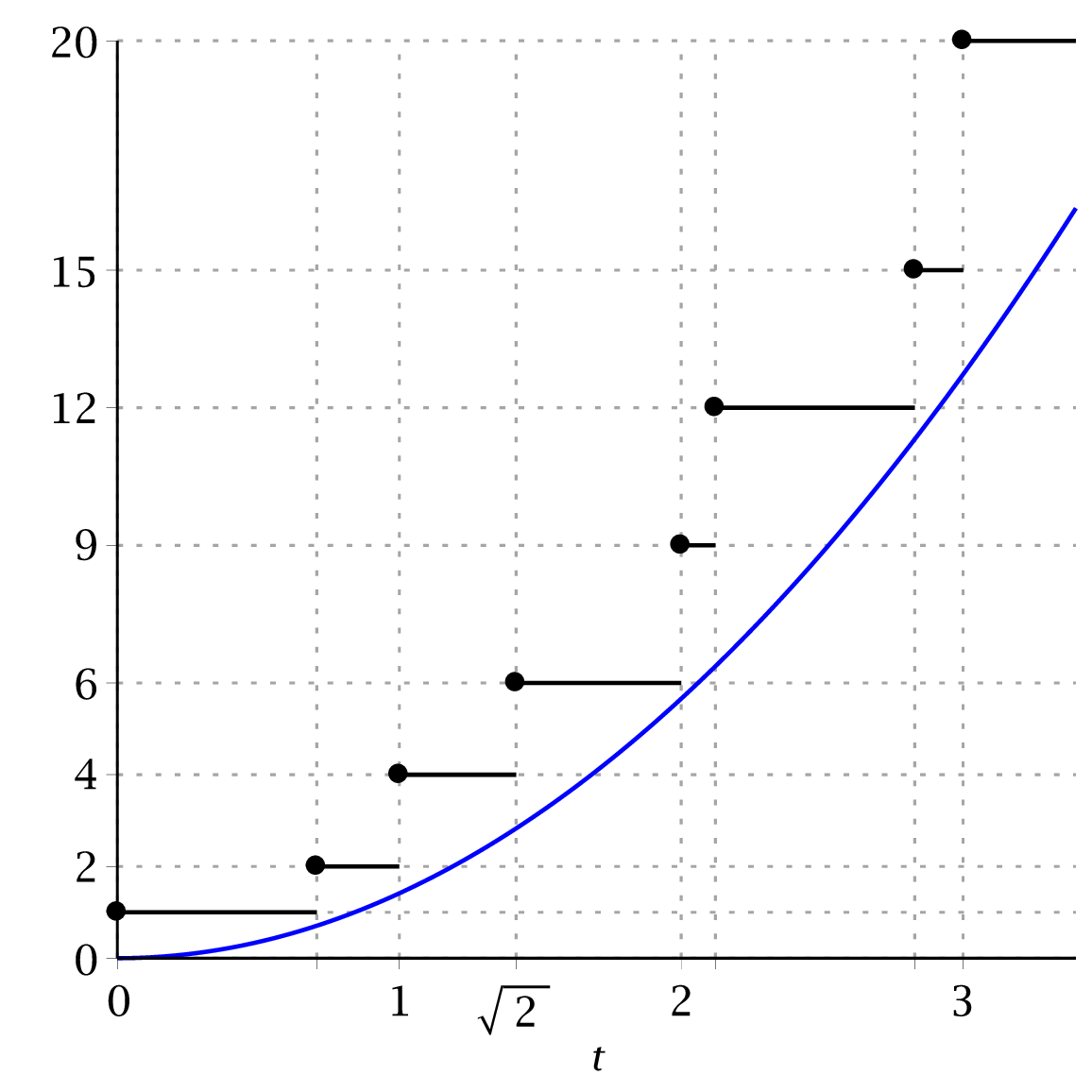} \\
\caption{Example \ref{ex:irrational-rectangle}.
  \emph{Left.} The linear function $t$ (\emph{top, blue}), the rational
  step-linear function $\{t\}$ (\emph{center, green}), and the irrational step-linear function
  $\{\sqrt2 t\}$ (\emph{bottom, red}). 
  \emph{Right.} The number of lattice points (\emph{black}) of  the  rectangle $[0,t]\times[0,t\sqrt{2}]$ is a function in
  the algebra generated by these three functions; the area
  (\emph{blue}) is a polynomial function.}\label{fig:irrational-rectangle}
 \end{center}
\end{figure}

\subsection{Real multi-parameter quasi-polynomials and their degrees}
To describe our contributions, let us first define our notion of (real,
multi-parameter) \emph{quasi-polynomials on $\R^N$} and notions of degree, 
which are crucial for our paper. 

First we define \emph{(rational) step-polynomials}.
For $t\in \R$, we denote by $\{t\}\in [0,1[$ the fractional part of $t$.
Thus $t\mapsto \{t\}$ is a function on $\R/\Z$.
Let $\eta=(\eta_1,\eta_2,\ldots, \eta_N)\in \Q^N$, which we consider as a linear form on $\R^N$.
We say that the function $b\mapsto \{\langle \eta,b\rangle \}$
is a \emph{(rational) step-polynomial  function}  of \emph{(step) degree} (at most) one
(or a \emph{(rational) step-linear function}).
If all $\eta_i$ have the same denominator $q$, this is a  function of $b\in \R^N/q\Z^N$.
We define $\polypp(\R^N)$ to be the algebra of functions on $\R^N$ generated by the functions  $b\mapsto \{\langle \eta,b\rangle \}$.
 An element of $\polypp(\R^N)$  is called a \emph{(rational) step-polynomial} on $\R^N$.
 The space $\polypp(\R^N)$ has an obvious filtration, where
$\polypp_{[\leq k]}(\R^N)$ is the linear span  of   $k$ or fewer products of functions  $b\mapsto \{\langle \eta,b\rangle \}$.
The elements of $\polypp_{[\leq k]}(\R^N)$ are said to be (rational) step-polynomials of
\emph{(step) degree (at most)~$k$}
.

Next, we define
$\polypp\CP(\R^N)$ to be the algebra of functions on $\R^N $ generated by (rational) step-polynomials and ordinary polynomial functions of $b$.
Elements of $\polypp\CP(\R^N)$ are called \emph{quasi-polynomials} on
$\R^N$ and take the form
\begin{equation}
  E(b) = \sum_{\substack{j=(j_1,\dots,j_N) \in \Z_{\geq0}^N\\
      |j| = j_1 + \dots + j_N \leq d + m}}  E_j(b)\, b^j,
  \label{eq:quasi-polynomial-by-monomials}
\end{equation}
using multi-index notation for the monomials $b^j = b^{j_1} \cdots
b^{j_N}$. Here the $E_j(b)$ are step-polynomials. 

This definition of quasi-polynomials on $\R^N$ is a natural generalization of
the notion of quasi-polynomial function on the lattice $\Z^N$, which is more
familiar in Ehrhart theory and the theory of vector partition functions.
Describing quasi-polynomials in this form, using step-polynomials as its
coefficient functions, has been implicit in the computational works using the
method of parametric generating functions
\cite{Verdoolaege2007parametric,Verdoolaege2005PhD}.  The extension to real
(rather than integer or rational) multi-parameters~$b$ appeared in
\cite{koeppe-verdoolaege:parametric}.

The algebra $\polypp\CP(\R^N)$ inherits a grading from the degree of
polynomials, which we call the \emph{polynomial degree}.  This is the notion
of degree that has been used throughout the literature on Ehrhart theory.

Crucial to our study will be the interplay of the polynomial degree with
another notion of degree, first introduced in our paper~\cite{SLII2014}.
The algebra $\polypp\CP(\R^N)$ also has a filtration, which we call the
\emph{local degree}. It combines
the polynomial degree and the filtration according to step degrees on
step-polynomials.  For instance, $b \mapsto b_1b_2^2 \{b_1+b_3\}$ has
polynomial degree $3$, step degree $1$, and local degree $4$.
This terminology of local degree comes from the fact that on each local region
$n<b_1+b_3<n+1$, $n\in\Z$, this function coincides with a polynomial function of $b$ of
degree $4$. \tgreen{Added last sentence from Michele's comment 2014-08-30.}

\subsection{First contribution: Intermediate real multi-parameter
  Ehr\-hart quasi-polynomials and their degree structure} 
\tgreen{I've changed it to `degree structure'
  (meaning polynomial and local degree)
  from `bidegree structure' (which now means exclusively means degree in s and
  degree in $\xi$).} 

Let $L$ be a rational subspace of~$V$.
We show that
$b\mapsto S^L (\p(b),h)$ is given by  a quasi-polynomial formula when the real
multi-parameter~$b$ varies in a chamber.  
This generalizes results in the literature in various ways.  
\begin{enumerate}[i.]
\item It extends from the case of discrete sums ($L=\{0\}$) as it appears
  in Ehrhart theory and the theory of vector partition functions
  \cite{Brion1997residue} to the general case of intermediate sums.
\item It generalizes these works also to the real multi-parameter case.
\item It analyzes the degree structure, i.e., the interplay of local degree
  and polynomial degree.  This is crucial for our second
  contribution, relating the terms of highest polynomial degree in~$b$ of
  $S(\p(b),h)$ to those of certain linear combinations of intermediate sums.
\end{enumerate}
Our theorem is the following (see Theorem~\ref{th:ehrhart-chamber} for a more
detailed statement).
\begin{theorem}
  \label{th:ehrhart-chamber-intro-summary}
  Assume that the weight $h(x)$ is  homogeneous  of degree $m$.
  When the real multi-parameter  $b$ varies in the closure of a chamber,
  the function
  $b\mapsto S^L(\p(b),h)$ is given by a quasi-polynomial function  of $b$ of
  local degree equal to $d+m$.
\end{theorem}

In particular, the terms of highest polynomial degree in $b$ of the function
$S^L(\p(b),h)$ form the homogeneous  polynomial of degree $d+m$ given by the
integral  $I(\p(b),h)$, while the term of polynomial degree $0$ (the
``constant term") is a step-polynomial of (step) degree at most $d+m$.

\begin{example}
  The simplest example is $V=\R$ with $\la\alpha_1,x\ra=x$ and
  $\la\alpha_2,x\ra=-x$. Thus $ \p(b)=\{\,x \in \R : x\leq b_1, -x\leq b_2\,\}$.
  If $b_1+b_2\geq 0$, the polytope $ \p(b)$ is the interval $[-b_2,b_1]$.  If
  $L=V$, then $ S^V(\p(b),1)=b_1+b_2$, while for $L=\{0\}$,
  $S^{\{0\}}(\p(b),1)= b_1+ b_2 -\{b_1\}-\{b_2\}+1$.  These two functions have
  local degree $1$ with respect to $(b_1,b_2)$.
\end{example}

The family $\p(b)$ is the family of polytopes obtained from a fixed simple  rational polytope
$\p$ by moving  each facet  parallel to itself in all possible ways.
We can consider smaller  families of polytopes with parallel faces.
For example, as in classical Ehrhart theory, we can dilate $\p$ to obtain $t\p$ for $t\in \R_{\geq 0}$, or more generally we can consider \emph{Minkowski linear systems}
$t_1\p_1+t_2\p_2+\cdots+t_q\p_q$.
By specializing our quasi-polynomial formulas, we obtain formulas for $S^L(t\p,h)$ and
$S^L(t_1\p_1+t_2\p_2+\cdots+t_q\p_q,h)$.

\begin{figure}
\begin{center}
  \includegraphics[width=6cm]{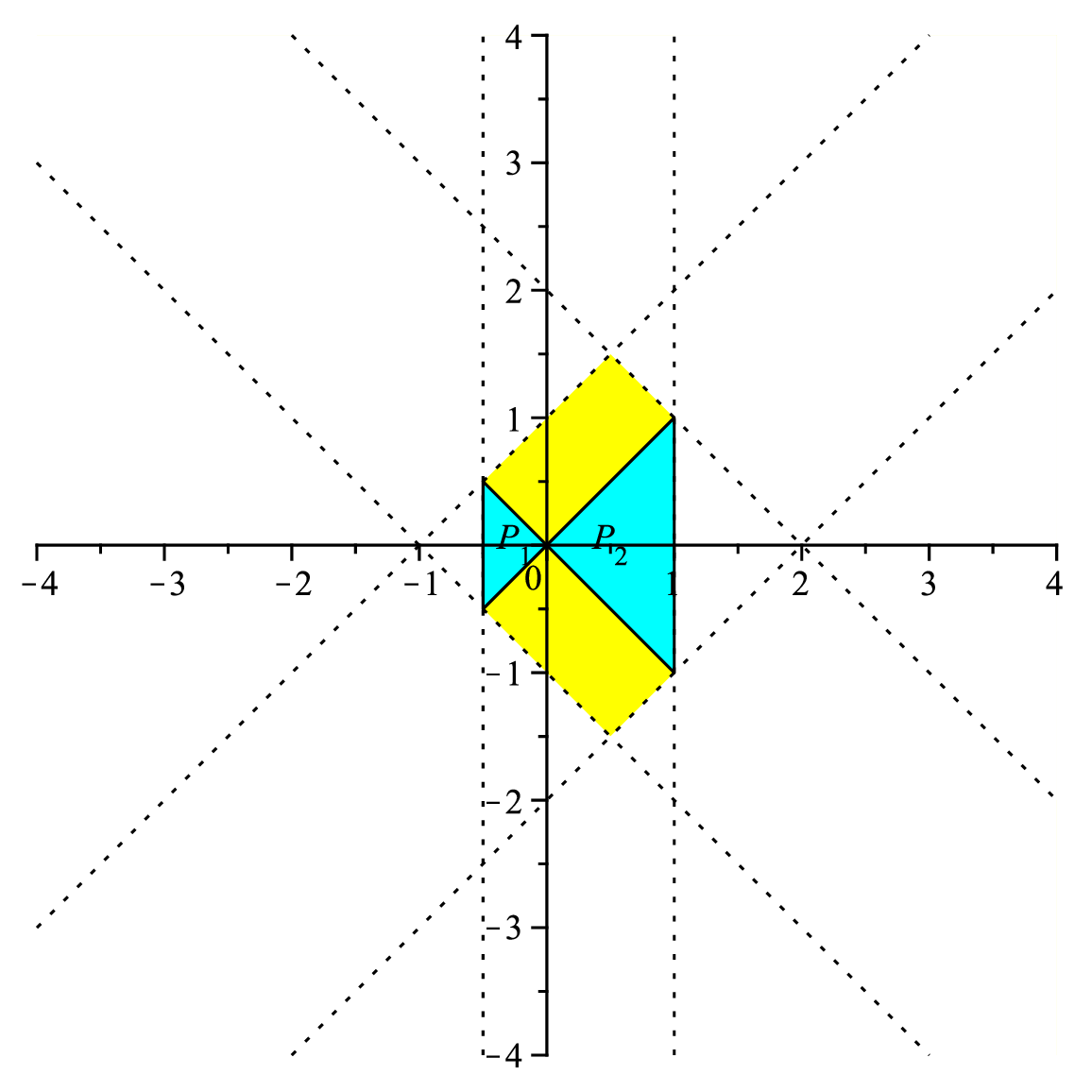} \includegraphics[width=6cm]{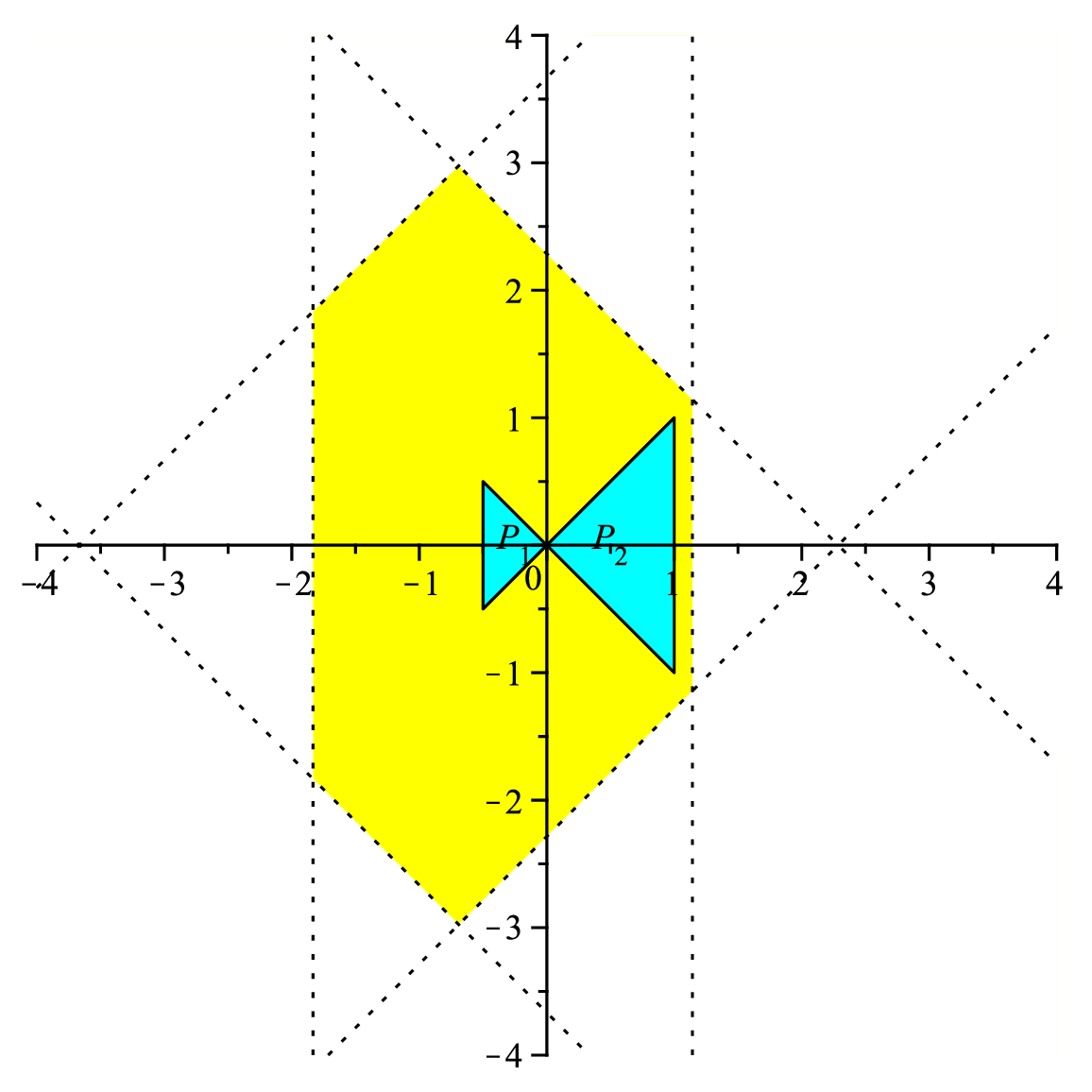}\\
 \caption{Minkowski linear system, Example \ref{ex:minkowski}.
   \emph{Left}, $\p_1+\p_2$.  \emph{Right}, $\frac{11}{3}\p_1+\frac{8}{7}\p_2$. }\label{fig:minkowski}
 \end{center}
\end{figure}

\begin{example}\label{ex:minkowski}
  \let\frac=\tfrac
We consider the Minkowski linear system generated by two triangles
$\p_1=((0,0),(-\tfrac12,\tfrac12), (-\tfrac12,-\tfrac12))$  and  $\p_2= ((0,0), (1,-1),\allowbreak (1,1))$;
see Figure~\ref{fig:minkowski}. Then $t_1\p_1+\t_2\p_2$ is the hexagon $\p(\alpha, t_1b^1+ t_2b^2)$, with
\begin{align*}
\alpha &=(x_1,x_1+x_2,-x_1+x_2,-x_1,-x_1-x_2,x_1-x_2 ),\\
 b^1 &= (0,0,1,\frac12,1,0), \\
 b^2 &= (1,2,0,0,0,2).
\end{align*}
Its vertices are $[t_2,t_2],[-\frac12t_1+t_2,\frac12t_1+t_2],[-\frac12t_1,\frac12t_1],[-\frac12t_1,-\frac12t_1],[-\frac12t_1+t_2,-\frac12t_1-t_2],[t_2,-t_2]$.

\begin{enumerate}[\rm(i)]
\item
The volume of $t_1\p_1+\t_2\p_2$ is
$$E_{[2]}(t_1,t_2)=\frac{1}{4}t_1^2+2t_1t_2+t_2^2.$$
\item The number $S(t_1,t_2)$ of lattice points in $t_1\p_1+\t_2\p_2$ is
$$S(t_1,t_2)=E_{[2]}(t_1,t_2)+E_{[1]}(t_1,t_2)+E_{[0]}(t_1,t_2)$$
where $E_{[1]}(t_1,t_2)$ is of polynomial degree $1$ and $E_{[0]}(t_1,t_2)$ is of
polynomial degree $0$, which are given by
\begin{align*}
E_{[1]}(t_1,t_2)&=(1-\{\frac{t_1}{2}\}-\{2t_2\})t_1+ (2-2\{t_1\}-2\{t_2\})t_2\\
\intertext{and}
E_{[0]}(t_1,t_2)&=1
-\{t_2\}^2-\{2t_2\}^2+2\{t_1\}\{\frac{t_1}{2}\}+2\{\frac{t_1}{2}+t_2\}\{t_1\} \\
&\qquad -\{\frac{t_1}{2}\}^2  -2\{\frac{t_1}{2}+t_2\}^2-
\{t_1\}-\{2t_2\}-\{t_1\}^2\\
&\qquad +2\{2t_2\}\{\frac{t_1}{2}+t_2\}+2\{t_2\}\{2t_2\}.
\end{align*}
\item
Similarly, if $L$ is the vertical line, the sum of the lengths of vertical segments is given by the quasi-polynomial function
$$S^L(t_1,t_2)=E_{[2]}(t_1,t_2)+E_{[1]}^L(t_1,t_2)+E_{[0]}^L(t_1,t_2)$$
with
\begin{align*}
E_{[1]}^L(t_1,t_2)&=(\frac{1}{2}-\{\frac{t_1}{2}\})t_1+(1-2\{t_2\})t_2,\\
E_{[0]}^L(t_1,t_2)&=
-\{-\frac{t_1}{2}+t_2\}-\{\frac{t_1}{2}-t_2\}+\{t_2\}+\{\frac{t_1}{2}-t_2\}^2
\\
&\qquad +\{\frac{t_1}{2}\}-\{\frac{t_1}{2}\}^2-\{t_2\}^2
+\{-\frac{t_1}{2}+t_2\}^2.
\end{align*}
\end{enumerate}
\end{example}

In contrast to the typical settings in the literature, 
we can allow the polytopes $\p_i$ to be merely \emph{semi-rational}, i.e., 
the facets of $\p_i$ are parallel to rational hyperplanes, 
whereas the vertices are allowed to be arbitrary real points in~$V$.

\begin{example}\label{ex:irrational-rectangle}
Let $\p$ be the rectangle $0\leq x\leq \sqrt{2}$,
$0\leq y\leq 1$ (see Figure~\ref{fig:irrational-rectangle}),
a semi-rational polytope.\footnote{The reader is invited to follow the
  examples using our Maple programs, available at
  \url{https://www.math.ucdavis.edu/~latte/software/packages/maple/} and 
  as part of \emph{LattE integrale},
  version 1.7.2 and later. 
}
For $t\in \R_{\geq 0}$,
the number of lattice points in $t\p$ is
$$
S(t\p,1) = \bigl(t-\{t\}+1\bigr)\bigl(\sqrt{2}\,t-\{\sqrt{2}\,t\}+1\bigr) 
= E_2(t)\, t^2 + E_1(t)\, t + E_0(t)
$$
with coefficient functions 
\begin{align*}
  E_2(t) &= \sqrt{2},\\
  E_1(t) &= -\sqrt{2}\,\{t\} - \{\sqrt{2}\,t\} + \sqrt2 + 1, \\
  E_0(t) &= \bigl(1-\{t\}\bigr) \bigl(1-\{\sqrt{2}\,t\}\bigr).
\end{align*}
Since the formulas of these functions involve both the rational step-linear function $\{t\}$ and
the irrational step-linear function $\{\sqrt{2}\,t\}$, 
the coefficient functions are not periodic in~$t$, but merely bounded
functions of~$t$.  Functions of this type generalize quasi-polynomials and are
called \emph{semi-quasi-polynomials} (the precise definition appears in
section~\ref{sub:quasipoly}). 
In the example, the function $S(t\p,1)$ is constant on the intersections
$\mathopen]m,m+1\mathclose[ \cap \mathopen]\frac{n}{\sqrt{2}},\frac{n+1}{\sqrt{2}}\mathclose[$  for $m, n$ positive integers.
\end{example}

More examples of  (semi-)quasi-polynomial functions  $S^L(t\p,h)$ will be given later.

\subsection{Second contribution: Two families of approximating multi-parameter quasi-polynomials}
We study the terms of highest polynomial degree of $S(\p(b),h)$.
We consider a \emph{patched weighted sum}, i.e., a particular linear combination of
intermediate weighted sums, for a  finite family $\CL_k^{\Barvi}$ of subspaces~$L$
 which was introduced by Barvinok in
 \cite{barvinok-2006-ehrhart-quasipolynomial} (see section~\ref{s:Lk-Barvinok}
 below for a definition)
 .  We thus obtain also a   function of $b$ which is given by a quasi-polynomial on each chamber.
\begin{equation}\label{eq:abstract-barvinok}
S^{\CL_k^{\Barvi}}(\p(b),h)= \sum_{L\in \CL_k^{\Barvi}} \rho(L) S^L (\p(b),h)
\end{equation}
(where the constants $\rho(L)$ are defined  in section \ref{sub:patch}).

Furthermore, we introduce a new quasi-polynomial function of $b$,  defined in a refined way by
linear combinations of intermediate generating functions for the cones at vertices of $\p(b)$.
We denote it $S^{k,\ConeByCone}\allowbreak(\p(b),h)$.  This function is canonically defined due to
the surprising analyticity of the cone-by-cone patched generating function
(Proposition~\ref{prop:cone-by-cone-is-analytic}). 

Our next main result is Theorem \ref{th:cone-by-cone-parametric-Ehrhart}.
We show that on each chamber the three quasi-polynomials, $b\mapsto S(\p(b),h)$, $b\mapsto
S^{\CL_k^{\Barvi}}(\p(b),h)$ and 
$b\mapsto S^{k,\ConeByCone}\allowbreak(\p(b),h)$, have the same terms corresponding to the
$k+1$ highest polynomial degrees.  This result  generalizes  Barvinok \cite{barvinok-2006-ehrhart-quasipolynomial} in several ways. Besides the introduction of the new $S^{k,\ConeByCone}(\p(b),h)$, we allow any polynomial weight $h(x)$, while Barvinok considered only $h(x)=1$,
 and we write formulas in terms of quasi-polynomial functions of the real-valued multi-parameter $b$, while Barvinok considered a single polytope dilated by a positive integer.

  $S^{k,\ConeByCone}(\p(b),h)$
and  $S^{\CL_k^{\Barvi}}(\p(b),h)$ involve subspaces $L$ of
 codimension  $\leq k$. For such an $L$,
  the computation of $S^L(\p(b),h)$
  involves discrete sums over  lattice points of semi-rational cones of dimension
 $\leq k$. For this reason, the  quasi-polynomials $S^{\CL_k^{\Barvi}}(\p(b),h)$ and
 $S^{k,\ConeByCone}(\p(b),h)$ are easier to compute than the original Ehrhart  quasi-polynomial $S(\p(b),h)$. Moreover,  $S^{k,\ConeByCone}(\p(b),h)$ is easier to compute than $S^{\CL_k^{\Barvi}}(\p(b),h)$.

Finally, when $\p(b)$ is a simplex, we give an explicit formula
for the coefficients $\rho(L)$ of the particular linear combination  \eqref{eq:abstract-barvinok} of  intermediate weighted sums used in Barvinok's approximation
$S^{\CL_k^{\Barvi}}(\p(b),h)$ (Proposition~\ref{prop:patching-simplex}),
using an explicit formula for  a M\"obius function that was obtained by
A.~Bj\"orner and L.~Lov\'asz in a different context \cite{MR1243770}.

\medbreak
The precise statements of our main results will lead to explicit algorithms for
computing both quasi-polynomials $S^{\CL_k^{\Barvi}}(\p(b),h)$ and
$S^{k,\ConeByCone}\allowbreak(\p(b),h)$ and thus the terms corresponding to the highest
$k+1$ polynomial degrees of $S(\p(b),h)$.
In various interesting settings,
for fixed~$k$ and fixed chamber,
%
%
%
%
as a corollary of our results, one can develop polynomial time algorithms to compute
$S^{\CL_k^{\Barvi}}(\p(b),h)$ and $S^{k,\ConeByCone}(\p(b),h)$; and thus
two types of polynomial time algorithms to compute the terms
corresponding to the highest expected
$k+1$ polynomial degrees of $S(\p(b),h)$ for a parametric  simplex.
Such polynomial time algorithms were the initial motivation
for the study of these intermediate sums in our papers.
However, in the present paper, in contrast to our previous papers~\cite{baldoni-berline-deloera-koeppe-vergne:integration,so-called-paper-1,so-called-paper-2},
we suppress detailed statements of such algorithms and their complexity.

\tgreen{I have removed (commented out) the complexity discussion per Michele's
  comment 2014-08-30.} 
\tgreen{I have now also removed the subsection title ``third contribution''
  entirely and gotten rid of the complexity theorems in the last section. --Matthias}
\medbreak

We end our article by some explicit computations (obtained via a simple Maple program)
of the quasi-polynomials $S^{k,\ConeByCone}(t\p,1)$
and  $S^{\CL_k^{\Barvi}}(t\p,1)$
for a dilated rational simplex $\p$ in  dimension $d\leq 4$, for $k\leq d$.
For $k=0$ they both give the volume $t^d \vol(\p)$ and for $k=d$ they both give the number of lattice points of $t\p$.
For $1\leq k\leq d-1$,
$S^{\CL_k^{\Barvi}}(t\p,1)$ and $S^{k,\ConeByCone}(t\p,1)$ have the same $k+1$ highest degree coefficients, but we see that
they are actually different quasi-polynomials.
It is not clear to us which one is the ``best.''


\subsection{Techniques of this paper}

Let us give the main ideas of our proofs.

We use the Brianchon--Gram decomposition of a
 polytope $\p(b)$ as a signed sum of its supporting cones
 and the corresponding Brion formula,
 reducing the study of intermediate weighted sums $S^L(\p(b),h)$ to that of intermediate generating functions $S^L(s+\c)(\xi)$
 over tangent cones  at vertices, which are defined in a similar way, by replacing  $h(x)$ with  an exponential function  $x\mapsto \e^{\la\xi,x\ra}$.
For $b$ in  a chamber $\tau$, the vertices of $\p(b)$ can be indexed, $s_1(b),\dots, s_r(b)$,  in such a way that
 for each index $j$,  the cone of feasible directions at vertex $s_j(b) $ does not depend on $b$.
 Furthermore, the vertex $s_j(b)$ depends linearly on the parameter $b$.
  This fact is the basis of all the  constructions and results of the present paper.
  We use our previous results on semi-rational affine polyhedral cones in \cite{SLII2014} and \cite{so-called-paper-2},
  where we studied  in detail the intermediate generating functions of a shifted cone $S^L(s+\c)(\xi)$
  (where $\c$ is a fixed polyhedral rational   cone)   as a function of  $s$ and~$\xi$.
  Theorem \ref{th:cone-by-cone-parametric-Ehrhart} uses the results of
  \cite{SLII2014} on the bidegree structure, i.e., the interplay between the
  local degree with respect to $s$ and the homogeneous degree with respect to~$\xi$. 
  Also,  as in  Barvinok's Fourier inversion method in \cite{barvinok-2006-ehrhart-quasipolynomial},  the  Poisson summation formula for   $S^L(s+\c)(\xi)$  obtained in \cite{SLII2014} is crucial to the proof of Theorem~\ref{th:cone-by-cone-parametric-Ehrhart}.

\section{Intermediate weighted Ehrhart quasi-polynomials for parametric polytopes}

\subsection{Notations}

In this paper, $V$ is a  vector space over $\R$ of dimension~$d$. The running element of $V$ is denoted by $x$.
As usual, the dual vector space is denoted by~$V^*$. By $V^*_\C$ we denote the
complexified dual space.  The running element of $V^*$ or $V^*_\C$ is denoted
by $\xi$.

The vector space $V$ is endowed with a lattice  $\lattice$ that spans~$V$ (one says that $V$
is \emph{rational}).  We denote by $\lattice^*$ the dual lattice in $V^*$.
We denote by $V_\Q = \lattice\otimes \Q$ and $V^*_\Q=\lattice^*\otimes \Q$ the
sets of rational elements of $V$ and $V^*$, respectively. 
\tgreen{Added proper definition in previous sentence according to Michele's
  comment 2014-08-31.}
A subspace $L$ of $V$ is called
\emph{rational} if $L\cap \lattice $ is a lattice in $L$. If $L$ is a
rational subspace, the image of $\lattice$ in $V/L$ is a lattice in
$V/L$, so that $V/L$ is a rational vector space. The image of
$\lattice$ in $V/L$ is called the \emph{projected lattice}. It is denoted
by $\lattice_{V/L}$.  A rational space~$V$, with lattice~$\lattice$,
has a canonical Lebesgue measure $\mathrm dx=\mathrm dm_\Lambda(x)$, for which
$V/\lattice$ has measure~$1$.
We denote by $L^{\perp}\subset V^*$ the space of linear forms $\xi\in V^*$
which vanish on~$L$.

An  affine subspace of $V$ is called \emph{semi-rational} if it can be written
as  $s+L$ where  $L$ is a rational subspace and $s$ is any point of $V$.

The polyhedra of this study are subsets of $V$.  A polyhedron is the
intersection of a finite number of closed halfspaces.  A \emph{polytope} is a
compact polyhedron, not necessarily full-dimensional.

The faces of a polyhedron $\p$ can be of dimension $0$ (vertices), $1$ (edges), $\ldots$, $\dim(\p)-1$ (facets), and $\dim(\p)$ (the polyhedron $\p$ itself).
A \emph{wall} of a polyhedron $\p$ is a hyperplane $H$ such that $\p$ is on one side
of $H$ and $\dim(H\cap \p)=\dim (\p)-1$. Then $H\cap \p$ is called a
\emph{facet} of $\p$.
A  polytope $\p$ of dimension $d$ is called \emph{simple} if  each vertex $s$ belongs to exactly $d$ facets.

If  $\f\subset V$ is a polyhedron,
the subspace $\lin(\f)$ is defined as the linear subspace of $V$ generated by $p-q$  for $p,q\in \f$.
A polyhedron $\p$ is called \emph{semi-rational} if $\lin(\f)$ is rational for all facets $\f$ of $\p$.

In this article, a \emph{cone} is a convex polyhedral rational  cone (with vertex $0$) and
an \emph{affine  cone} is the shifted set $s+\c$ of a  rational cone $\c$ for some $s\in V$.
A cone~$\c$ is called \emph{pointed} if it does not contain a line.
A  cone $\c$ is called \emph{simplicial} if it is  generated by linearly independent
elements of $V $. A simplicial cone~$\c$ is called \emph{unimodular} if it
is generated by independent lattice vectors $v_1,\dots, v_k$ such
that $\{v_1,\dots, v_k\}$ is part of a basis of
$\lattice$. An affine cone~$s+\c$ is called \emph{pointed} (\emph{simplicial},
\emph{unimodular}, respectively) if the associated cone~$\c$ is.

The set of vertices of a polytope~$\p$ is denoted by $\CV(\p)$.
For each vertex $s$, the cone of feasible directions at $s$  is denoted by
$\c_s$.

When we speak of the \emph{parametric polytope} $\p(b)$,
the parameter space is  $\R^N$. Its running element is denoted by $b=(b_1,b_2,\ldots,b_N)$.
We  denote by $e_j$ the canonical basis of $\R^N$, and write also $b=\sum_{j=1}^N b_j e_j$.

The indicator function of a subset $E$ is  denoted by $[E]$.

For  $t\in \R$, we denote by $\{t\}\in [0,1[$ the fractional part of $t$. Then $t-\{t\}$ is an integer.

\subsection{Intermediate weighted sums and generating functions on polyhedra}\label{subsection:polyhedra}
Let  $\p\subset V$ be a semi-rational polytope, $L$ a rational subspace of $V$
and let $h$ be a polynomial function on  $V$. We are interested in the properties of the \emph{intermediate weighted sum} on $\p$,
\begin{equation}\label{eq:abstract-first}
 S^L(\p,h)=\sum_{y\in \lattice_{V/L}} \int_{\p\cap (y+L)} h(x)\,\mathrm dx.
\end{equation}
Here, $\lattice_{V/L}$ is the  projected  lattice
and $\mathrm dx$ is the Lebesgue measure on $y+L$ defined by the intersection lattice $L\cap\lattice$.
Thus we integrate over  the intersections of $\p$  with  subspaces parallel to a fixed rational subspace $L$ through all lattice points, and sum the integrals.

Although  $S^L(\p,h)$ depends on the lattice~$\Lambda$, we do not indicate this dependence  in the notation.

$S^L(\p,h)$ interpolates between the integral
$$
I(\p,h)= \int_\p h(x) \,\mathrm dx
$$
of $h$ on $\p$,
which corresponds to $L=V$,  and the discrete weighted sum
$$
S(\p,h)= \sum_{x\in \p\cap \lattice}h(x),
$$
which corresponds to $L=\{0\}$.

Consider, instead of the polynomial function $h(x)$, the exponential function $\e^{\la
\xi,x\ra}$.
Thus we define the following holomorphic function of $\xi$:
\begin{equation}\label{eq:SL}
S^L(\p)(\xi)= \sum_{y\in \lattice_{V/L}} \int_{\p\cap (y+L)} \e^{\la
\xi,x\ra}\,\mathrm dx.
\end{equation}
Following the method initiated by Barvinok,
the study of the generating function $S^L(\p)(\xi)$  reduces to the
computation of the similar functions  $S^L(\u)(\xi)$, where $\u$ are affine
cones, see \cite{bar,baldoni-berline-deloera-koeppe-vergne:integration,so-called-paper-1}, etc.
However, for an arbitrary non-compact polyhedron~$\p$, the above definition  \eqref{eq:SL} of  $S^L(\p)(\xi)$
 makes sense only as a generalized function of~$i\xi$.
To avoid the use of distributions, and
 stay in an algebraic context,
we will define  $S^L(\p)(\xi)$, for any semi-rational polyhedron~$\p$ as a meromorphic function of~$\xi$, satisfying a valuation property.

Let us recall the notations of~\cite[Definitions 2.1--2.3]{SLII2014}. 
\begin{definition}\label{def:Mell}~
  \begin{enumerate}[\rm(a)]
  \item We denote by $\CM_{\ell}(V^*)$ the ring of meromorphic functions
    around $0\in V^*_\C$ which can be written as a quotient
    $\frac{\phi(\xi)}{\prod_{j=1}^{N}  \la\xi,w_j\ra}$, where $\phi(\xi)$ is holomorphic
    near $0$ and $w_j$ are non-zero elements of $V$ in finite number.
    (The subscript $\ell$ is mnemonic for the linear forms
      that appear in the denominator.)
  \item
    We denote by ${\mathcal R}_{[\geq m]}(V^*)$ the space of rational
    functions which can be written as $\frac{P(\xi)}{\prod_{j=1}^{N}
      \la\xi,w_j\ra}$, where $P$ is a homogeneous polynomial of degree greater
    or equal to $m+N$.  These rational functions are said to be
    \emph{homogeneous of degree at least~$m$}.
  \item We denote by ${\mathcal R}_{[ m]}(V^*)$ the space of rational
    functions which can be written as $\frac{P(\xi)}{\prod_{j=1}^{N}
      \la\xi,w_j\ra}$, where $P$ is homogeneous of degree $m+N$.
    These rational functions are said to be \emph{homogeneous of degree~$m$}.
  \end{enumerate}
\end{definition}
\begin{definition}
  For $\phi\in \CM_{\ell}(V^*)$, the \emph{homogeneous component} $
  \phi_{[m]}$ of degree~$m$ of~$\phi$ is defined by considering $\phi(\tau
  \xi)$ as a meromorphic function of one variable~$\tau\in\C$, with Laurent
  series expansion
 $$
 \phi(\tau \xi) = \sum_{m\geq m_0} \tau^m \phi_{[m]}(\xi).
 $$
\end{definition}
Thus $\phi_{[m]}\in {\mathcal R}_{[m]}(V^*)$.

\begin{definition}
  An $\CM_{\ell}(V^*)$-valued \emph{valuation} on the set of semi-rational
  polyhedra $ \p\subseteq V$ is a map $F$ from this set to the vector space
  $\CM_{\ell}(V^*)$ such that whenever the indicator functions $[\p_i]$ of a
  family of polyhedra $\p_i$ satisfy a linear relation $\sum_i r_i [\p_i]=0$,
  then the elements $F(\p_i)$ satisfy the same relation
$$
\sum_i r_i F(\p_i)=0.
$$
\end{definition}

We recall the following result from \cite[Proposition~19]{so-called-paper-1} (see
also \cite[Proposition~3]{so-called-paper-2} and \cite[Proposition~2.4]{SLII2014}).

\begin{proposition}
  Let $L\subseteq V$ be a rational subspace. There exists a unique valuation
  which associates a meromorphic function $S^L(\p)(\xi)$ belonging to
  $\CM_{\ell}(V^*)$ to every semi-rational polyhedron $\p\subseteq V$, so that
  the following properties hold:

  \begin{enumerate}[\rm(i)]
  \item  If $\p$ contains a line, then $S^L(\p)=0$.

  \item
    $$S^L(\p)(\xi)= \sum_{y\in \lattice_{V/L}} \int_{\p\cap (y+L)} \e^{\la
      \xi,x\ra}\,\mathrm dx,$$ for every $\xi\in V^*$ such that the above sum
converges.
\end{enumerate}
\end{proposition}
We remark that property~(i) above reflects the fact that the distribution
$t\mapsto \sum_{n\in \Z} \e^{in t}$ is supported on  $2\pi \Z$. 

Following \cite[section~3.4]{so-called-paper-1},
\cite[section~2.2]{so-called-paper-2} and \cite[Definition~2.5]{SLII2014}, the
function $S^L(\p)(\xi)$ will be called {\em the intermediate generating
  function} of the polyhedron $\p$.

\bigskip

If $\p$ is a polytope, then clearly the homogeneous component of degree $r$ of $S^L(\p)(\xi)$ is given by
$$S^L(\p)_{[r]}(\xi)=\sum_{y\in \lattice_{V/L}} \int_{\p\cap (y+L)} \frac{\la
\xi,x\ra^r}{r!}\,\mathrm dx.$$
Thus  the intermediate generating function  $S^L(\p)(\xi)$ of the polytope $\p$
allows us to compute $S^L(\p,h)$ for any polynomial $h$ by decomposing $h$ in
sums of powers of linear forms \cite[section~3.4]{baldoni-berline-deloera-koeppe-vergne:integration}.
\bigskip

The intermediate generating function  $S^L(\p)(\xi)$
interpolates between the integral (continuous generating function)
$$
I(\p)(\xi)= \int_\p \e^{\la \xi,x\ra} \,\mathrm dx
$$
which corresponds to $L=V$,  and the discrete sum (discrete generating function)
$$
S(\p)(\xi)= \sum_{x\in \p\cap \lattice}\e^{\la \xi,x\ra}
$$
which corresponds to $L=\{0\}$.

\medskip

 We introduced in \cite[Definition~2.18]{SLII2014} 
 the algebra  $\polypp\CP^{\Psi}(V)$
 of quasi-polynomial functions on $V$ defined as follows.
 (These definitions are analogues of the quasi-polynomial functions on $\R^N$
 defined in the introduction.)
 \begin{definition}
   Let $\Psi$ be a finite subset of elements of $V^*_\Q$.
   Define the algebra
   $\polypp^{\Psi}(V)$ of functions of $s\in V$ generated by the functions
   $s\mapsto \{\la\gamma,s\ra\}$ with $\gamma \in \Psi$.  A function $f$ will be
   called a \emph{(rational) step-polynomial function} on $V$, if $f$ belongs to
   $\polypp^{\Psi}(V)$ for some $\Psi$.
 \end{definition}
Such a function $f$ is bounded and periodic modulo
$q\Lambda$ for some integer $q$.

The algebra  $\polypp^{\Psi}(V)$ again has a natural filtration,
 where    $\polypp^{\Psi}_{[\leq k]}(V)$ is  the subspace spanned  by products of at most $k$ functions
$\{\la\gamma,s\ra\}$. The algebra
$\polypp\CP^{\Psi}(V)$ is the algebra generated by
$\polypp^{\Psi}(V)$ and ordinary  polynomial functions.
A function $f$ will be called a \emph{quasi-polynomial function} on $V$, if $f$ belongs to
$\polypp\CP^{\Psi}(V)$ for some $\Psi$.

\medskip
Let $\c$ be a fixed cone. The translated cone $s+\c$ is a semi-rational polyhedron.
Let $S^L(s+\c)_{[m]}(\xi)$ be the degree-$m$ homogeneous component of the functions $S^L(s+\c)(\xi)$.
The lowest homogeneous component of $S^L(s+\c)$ has  degree $-d$ and is equal
to $I(\c)(\xi)$, as proved in \cite[Theorem \ref{sl2:prop:homogeneous-ML}~(iii)]{SLII2014}.

 Remark that for any  $v\in \lattice+L $,
  $$
  S^L(v+s+\c)(\xi)= \e^{\langle \xi,v\rangle} S^L(s+\c)(\xi).
  $$

Thus the function
\begin{equation}\label{def:ML}
M^L(s,\c)(\xi)=\e^{-\langle \xi,s\rangle} S^L(s+\c)(\xi)
\end{equation}
is a function of $s\in V/(\lattice+L)$.

Let $M^L(s,\c)_{[m]}$ be the degree-$m$ homogeneous component of the function $M^L(s,\c)$.
 The function $s\mapsto M^L(s,\c)_{[m]}$
 is a periodic function of the variable $s\in V$ with values in the space $\CR_{[m]}(V^*)$.
 Results of \cite{SLII2014} on the nature of this function will be used in section \ref{sub:quasipoly}.

Note the obvious, but fundamental formula:

\begin{equation}\label{eq:MversusS}
 S^L(s+\c)_{[m]}(\xi)=\sum_{r=0}^{m+d} M^L(s,\c)_{[m-r]}(\xi) \frac{\la \xi,s\ra^r}{r!}.
\end{equation}

\begin{remark}
  Equation~\eqref{eq:MversusS} expresses the function
$s\mapsto S^L(s+\c)_{[m]}(\xi)$ as a
quasi-polynomial function
 of the variable $s\in V$ with values in the space $\CR_{[m]}(V^*)$.
Indeed  $s\mapsto M^L(s,\c)_{[m-r]}(\xi)$ is a periodic function on~$V$
given by a step-polynomial formula, while $s\mapsto \frac{\la \xi,s\ra^r}{r!}$ is a polynomial function of $s$.
\end{remark}

 The lowest degree homogeneous component of $M^L(s,\c)$ is the same as that of $ S^L(s+\c)$, hence of degree $-d$ and  equal to $I(\c)(\xi)$.
\begin{equation}\label{eq:lowest}
M^L(s,\c)_{[-d]}=I(\c)(\xi).
\end{equation}

When $L=\{0\}$, we write $S(s+\c)$ instead of $S^{\{0\}}(s+\c)$
and $M(s,\c)$ instead of $M^{\{0\}}(s,\c)$.
\subsection{Parametric polytopes}
\subsubsection{Chambers}
First, we recall some well known notions about parametric polytopes (cf.\ for instance \cite{MR0326574}).
Let $\alpha=(\alpha_1,\alpha_2,\ldots,\alpha_N)$ be a list of  $N$ elements of $\Lambda^*$
such  that the cone generated by  $\alpha$ is the whole space $V^*$.
 For
 $b=(b_1,b_2,\ldots,b_N)\in \R^N$, let
 $$
\p(\alpha,b)=\{\,x\in V : \la \alpha_j,x\ra\leq b_j,\, j=1,\ldots,N\,\}.
$$
Then $\p(\alpha,b)$ is a semi-rational polytope.
We will often denote it simply by $\p(b)$.

It is clear that if we dilate the parameter $b$ in $tb$ with $t\geq 0$, we obtain the dilated polytope $t\p( b).$
\begin{equation}\label{eq:tb}
t\p(b)=\p(tb).
\end{equation}

We denote also  by $\alpha$ the  map  $V\to \R^N$  given by $$\alpha(x)=(\la \alpha_1,x\ra, \la \alpha_2,x\ra,\ldots, \la \alpha_N,x\ra).$$

If $v_0\in V$, the shifted polytope  $\p(b)+v_0$ is given by
\begin{equation}\label{eq:b-modmuV}
\p(b)+v_0=\p(b+\alpha(v_0)).
\end{equation}
Hence, for many aspects of parametric polytopes, the relevant space of parameters is not $\R^N$ itself,
 but the quotient space $\R^N/\alpha(V)$, (see Remark \ref{rem:parametric-versus-partition}). For instance, $\p(b)$ is not empty if and only if
 $b$ lies in  the  closed cone generated by the standard basis $e_i$  of $\R^N $ and the subspace $\alpha(V)$.
Indeed, for $x\in \p(b)$, $b=\sum_j (b_j-\la \alpha_j,x\ra) e_j+\alpha(x)$.

\begin{example}\label{ex:parametric-dim1}
    The simplest example is $V=\R$ with $\alpha_1=x, \alpha_2=-x$. Thus
    $    \p(b)=\{x \in \R, x\leq b_1, -x\leq b_2\}$.
 If $b_1+b_2<0$, $ \p(b)$ is empty, otherwise it is the interval $[-b_2,b_1]$.
\end{example}
\begin{example}[Figure \ref{fig:parametric-dim2}]\label{ex:parametric-dim2}
Let $V=\R^2$ and $\alpha =(-x_1,-x_2,x_1+x_2, -x_1+x_2)$.
Then $\p(b)$ is defined by the inequalities
$$ -x_1\leq  b_1,\quad -x_2\leq  b_2,\quad x_1+ x_2\leq b_3,\quad -x_1+x_2\leq  b_4.
$$
Figure \ref{fig:parametric-dim2} displays  this polygon for three values of $b$, one for each of  the three  combinatorial types which occur in this example. (See below Example \ref{ex:parametric-dim2-chambers}).
\end{example}
\begin{figure}
\begin{center}
 \includegraphics[width=4cm]{mirage1.eps}  \includegraphics[width=4cm]{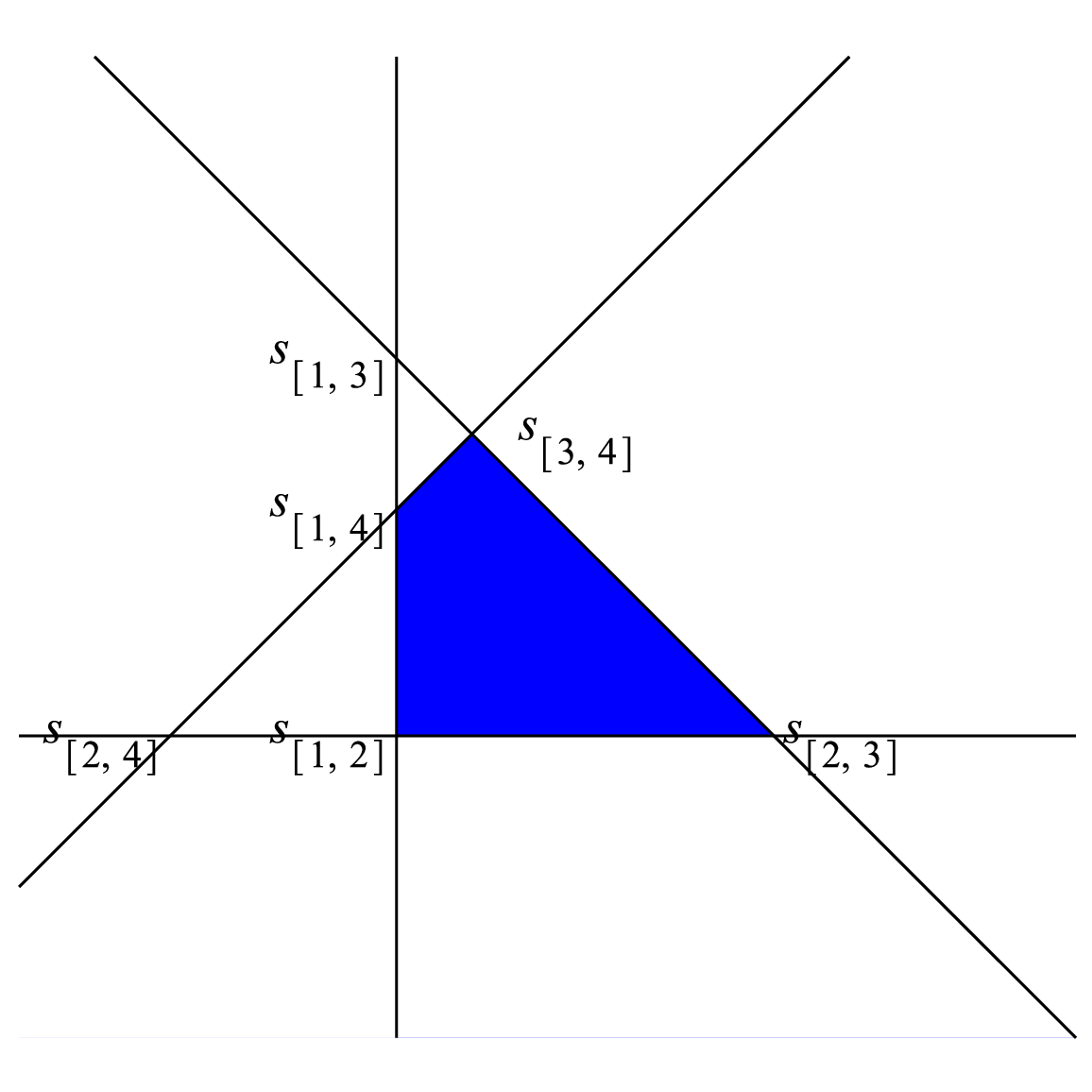} \includegraphics[width=4cm]{mirage3.eps}
 \caption{Example \ref{ex:parametric-dim2}, $\p(b)$ for $b=(2,0,0,6), (0,0,5,3), (0,3,3,0)$. }\label{fig:parametric-dim2}
 \end{center}
\end{figure}
When $b$ varies, the combinatorial type of $\p(b)$ changes.
The \emph{$\alpha$-chambers} $\tau\subset \R^N$, defined below, are  conical open subsets
such that  the  combinatorial type of $\p(b)$ does not change when $b$ runs over $ \tau$. First, we want to index the vertices of $\p(b)$. For this purpose, we introduce the following subsets of $\{1,\dots,N\}$.
\begin{definition}\label{def:mubasis}
 $\CB$ is the set of subsets  $B\subseteq \{1,\dots,N\}$
   such that $(\alpha_j, j\in B) $ is a basis of $V^*$.

For $B\in \CB$,  the linear map
$ b\mapsto s_B(b): \R^N \to V$ is defined by
$$
\langle\alpha_j,s_B(b)\rangle=b_j \mbox{ for } j\in B.
$$
For $B\in \CB$,  the cone $\c_B\subset V$ is defined by
$$
\c_B=\{\,x\in V : \langle\alpha_j,x\rangle\leq 0,  \mbox{ for } j\in B \,\}.
$$
\end{definition}
Thus, the point $s_B(b)$ is the intersection of the $d$ linearly independent hyperplanes
$\langle\alpha_j,s_B(b)\rangle=b_j$ for $j\in B$. The points $s_B(b), B\in \CB,$ are classically called the \emph{vertices} of the arrangement of hyperplanes $\langle\alpha_j,x\rangle =b_j$.

\begin{example}[Continuation of Example \ref{ex:parametric-dim2}, Figure \ref{fig:parametric-dim2}]\label{ex:parametric-dim2-vertices}
Here,  $\CB$ consists of all pairs $[i,j]$ with  $1\leq i<j\leq 4$. The six  points $s_B(b)$ are
   the intersections of the corresponding two lines:
\newline   $ s_{[1,2]}(b)=(-{ b_1},-{ b_2}), s_{{[1,3]}}(b)=(-{ b_1},b_1+ b_3),
 s_{{[1,4]}}(b)=(-{ b_1},{ -b_1}+{ b_4})$,
\newline $s_{[2,3]}(b)=(b_2+b_3,- b_2), s_{[2,4]}(b)=( -b_2- b_4,-{b_2})$,
\newline  $ s_{{[3,4]}}(b)=(\frac{1}{2}(b_3-b_4) ,\frac{1}{2}(b_3+b_4))$.
\newline%
The six cones $\c_B$ (up to a shift) appear as the tangent cones of
the vertices of the three example polytopes~$\p(b)$ in Figure~\ref{fig:parametric-dim2}.
\end{example}

The actual vertices of the polytope $\p(b)$ are exactly those points $s_B(b)$ which satisfy the remaining inequalities $\langle\alpha_k,s_B(b)\rangle\leq b_k$ for $k\notin B$.
Thus, for $B\in \CB$, we consider the  following cone in $ \R^N$.
\begin{equation}\label{eq:cones-for-chambers}
 \d_B=\{\, b\in \R^N :  b_k-\langle \alpha_k,s_B(b)\rangle\geq 0, \mbox{  for } k\notin B \,\}.
\end{equation}

\begin{lemma}\label{lemma:cones-for-chambers}
    Let $e_j, 1\leq j\leq N$ be the standard basis of $\R^N$, and let $\phi_j$ be the
    projection of $e_j$ on $\R^N/\alpha(V)$ for $1\leq j\leq N$.

    \begin{enumerate}[\rm(i)]
    \item The cone $\d_B$ coincides with $\alpha(V)+\sum_{k\notin
        B}\R_{\geq 0}e_k.$

    \item The projection of $\d_B$ on $\R^N/\alpha(V)$ is the
      simplicial cone with generators $\phi_k$, $k\notin B$.
    \end{enumerate}
\end{lemma}

\begin{proof}
Part (i) follows from the formula  $$b-\alpha(s_B(b))=\sum_{k\notin B} \bigl(b_k-\langle \alpha_k,s_B(b)\rangle\bigr)e_k.$$

For (ii), note that if $B$ is a subset of $d$ elements of $\{1,\dots,N\}$,
the linear forms $(\alpha_j,j\in B)$ are linearly independent if and only if the vectors $(\phi_k, k\notin B)$ are.
\end{proof}

\begin{definition}\label{def:chamber}
 An $\alpha$-chamber  $\tau\subset \R^N$ is a connected component of the complement of
 the union of the boundaries of the  cones $\d_B$
  \eqref{eq:cones-for-chambers}
 for all $B\in \CB$.
\end{definition}
Thus, for every $j\in \{1,\dots,N\}$ and $B\in \CB$,  the linear form $ b_j- \langle \alpha_j,s_B(b)\rangle$
keeps a constant sign on every chamber.

There is a unique chamber $\tau_e$ such that $\p(b)$ is empty for $b\in \tau_e$,
it is the complement of the cone  $\alpha(V)+\sum_{k=1}^N\R_{\geq 0}e_k$.
The other chambers are  contained in  $\alpha(V)+\sum_{k=1}^N\R_{\geq 0}e_k$, they are called \emph{admissible}.

Chambers  are open conical sets which contain the vector subspace $\alpha(V)$.
Thus the relevant sets are the projected chambers in  $\R^N/\alpha(V)$. They are easier to visualize.
\begin{lemma}\label{lemma:projected-chambers}
The projected chambers in $\R^N/\alpha(V)$ are the connected components
    of the complement of the union of the boundaries of the cones generated by
    $N-d$ linearly independent vectors among the $\phi_j$'s.
\end{lemma}
\begin{figure}
\begin{center}
 \includegraphics[width=4cm]{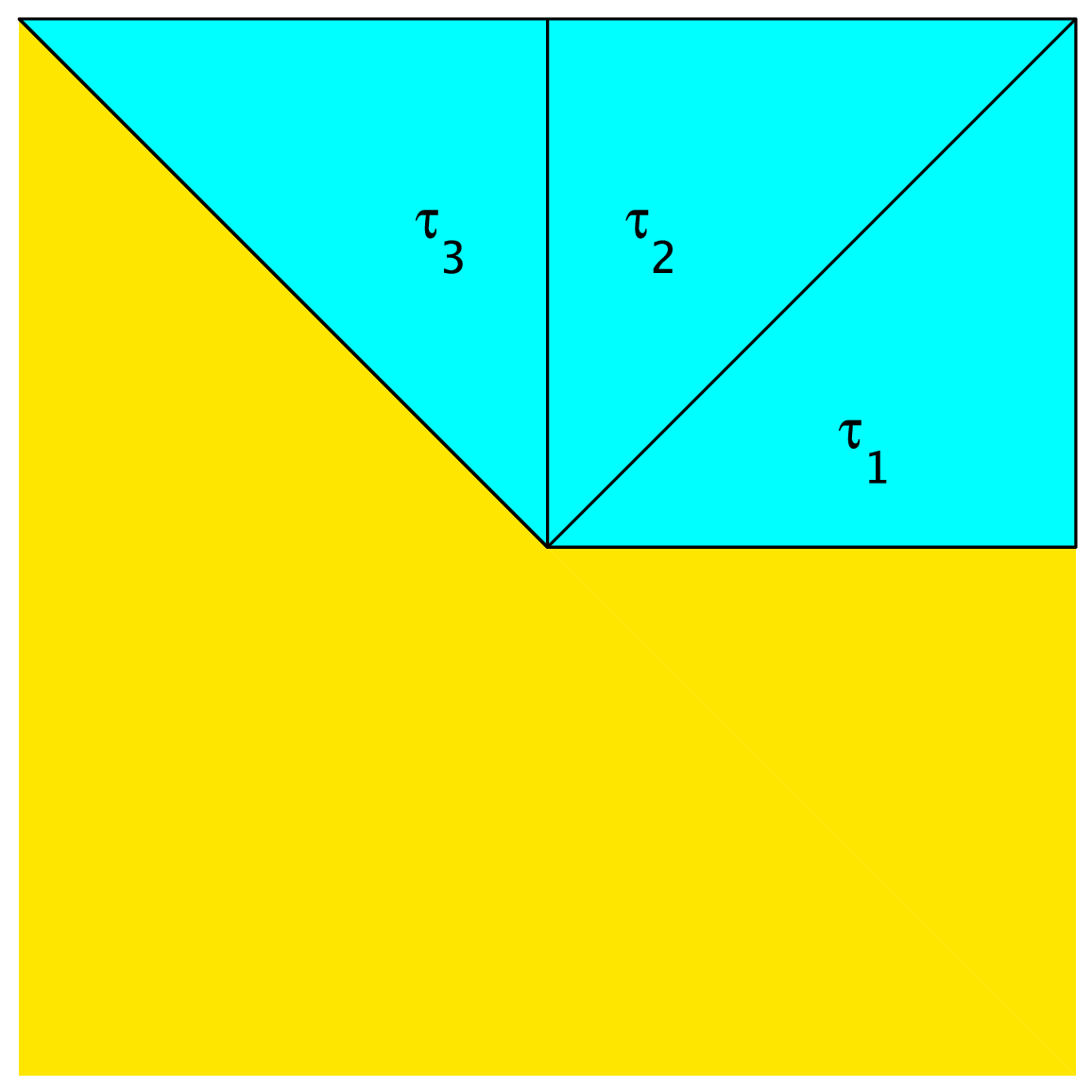}
 \caption{Four projected chambers in Example~\ref{ex:parametric-dim2-chambers}}\label{fig:parametric-dim2-chambers}
 \end{center}
\end{figure}
We have the following well-known  description of the walls, vertices and cones at vertices of $\p(b)$ when $b$ varies in  an $\alpha$-chamber.
\begin{proposition}\label{prop:vertices-chamber}
 Let $\tau$ be an $\alpha$-chamber. Let   $\CB_\tau$ be  the set of $B\in \CB$ such that $\tau$ is contained in the cone $\d_B$.
 Then, for $b\in\tau$, the following holds.

 \begin{enumerate}[\rm(i)]
 \item $\p(b)$ is simple.

 \item The hyperplane $\la \alpha_j,x\ra=b_j$ is a wall of $\p(b)$, i.e., its
   intersection with $\p(b)$ is a facet, if
   and only if $j$ belongs to some $B$ in $\CB_\tau$.

 \item The vertices of $\p(b)$ are the points $s_B(b)$ where $B$
   runs over $\CB_\tau$.

 \item For $B\in \CB_\tau$, the cone of feasible directions at
   vertex $s_B(b)$ of $\p(b)$ is the cone $\c_B=\{\,x\in V:
   \langle\alpha_j,x\rangle\leq 0 \mbox{ for } j\in B\,\}$. In particular, it
   depends only on $B$, not on $b$.
 \end{enumerate}
\end{proposition}

The projection $\tau/\alpha(V)$ of the chamber $\tau$ on $\R^N/\alpha(V)$ is the intersection
of all the simplicial cones,
generated by  subsets of  the $\phi_k$, and containing it.
Thus an equation $\alpha_j$ is redundant for $\p(b)$
if and only if the vector $\phi_j$ is an edge of {\bf all} the simplicial cones  $\d_B/\alpha(V)$, containing $\tau/\alpha(V)$.

\begin{example}[Continuation of  Example \ref{ex:parametric-dim2} and  Figure \ref{fig:parametric-dim2}]  \label{ex:parametric-dim2-chambers}
 Here,
$\Phi=(\phi_1,\phi_2, \phi_3=\frac{\phi_1+\phi_2}{2}, \phi_4=\frac{-\phi_1+\phi_2}{2})$.
There are four chambers,  whose projections are represented in Figure \ref{fig:parametric-dim2-chambers}. The projections of the three admissible chambers $\tau_i$, $i=1,2,3$ are colored in blue. The fourth chamber (for which $\p(b)$ is empty) is colored in yellow.

   $\tau_1$ is the cone $\{\, b : \sum_{j=1}^4 b_j\phi_j\in \R_{>0}\phi_1+ \R_{>0}\phi_3\,\}$.

   $\tau_2$ is the cone $\{\, b : \sum_{j=1}^4 b_j\phi_j\in \R_{>0}\phi_2+ \R_{>0}\phi_3\,\}$.

   $\tau_3$ is the cone $\{\, b : \sum_{j=1}^4 b_j\phi_j\in \R_{>0}\phi_2+ \R_{>0}\phi_4\,\}$.

For instance, $\tau_2$ is defined by the inequalities $-b_1+b_2+b_4>0, 2b_1+b_3-b_4>0$.

As we see on Figure \ref{fig:parametric-dim2},
for $b\in \tau_1$, $\p(b)$ is a triangle, 
for $b\in \tau_2$, it is a quadrilateral,
for $b\in \tau_3$, it is again a triangle.
%
%
%
We have $\CB_{\tau_1}=\{[2,3],[2,4],[3,4]\}$,
$\CB_{\tau_2}=\{[1,2],[1,4],[2,3],[3,4]\}$, and $\CB_{\tau_3}=\{[1,2],[1,3],[2,3]\}$.

The index $1$ does
not belong to the union of the sets $B$ when $B$ varies in $\CB_{\tau_1}$, thus the condition $-x_1\leq b_1$ is redundant
for the polytope $\p(b)$ when $b\in \tau_1$,
as seen on Figure \ref{fig:parametric-dim2}.

Similarly,  the condition $-x_1+x_2\leq b_4$ is redundant for the polytope $\p(b)$ when $b\in \tau_3$.

For $b\in \tau_2$, the four  equations $\la\alpha_j,x\ra=b_j$ define facets of the quadrilateral $\p(b)$.
\end{example}
If $b$ lies in the boundary of an admissible chamber  $\tau$, then all vertices of $\p(b)$ are of the form
$s_B(b)$, for $B\in \CB_\tau$, but several $B\in \CB_\tau$  may give the same vertex $s_B(b)$.

\begin{remark}[Minkowski sums of polytopes]
  Remark finally the following relation between parametric polytopes and
  Minkowski sums $t_1\p_1+t_2\p_2+\cdots+t_q\p_q$ of polytopes.  Let $\tau$ be
  an admissible chamber.  Let $b_1,b_2,\ldots, b_q$ in $\overline \tau$ and
  $t_1\geq 0,\ldots, t_q\geq 0$, then $t_1b_1+t_2b_2+\cdots+t_q b_q$ is in
  $\overline \tau$, and
\begin{equation}\label{eq:Minkowski}
  t_1\p(b_1)+t_2\p(b_2)+\cdots+t_q \p(b_q)=
  \p(t_1b_1+\cdots+t_q b_q).
\end{equation}

Indeed, using the linearity of the map $b\mapsto s_B(b)$, we see immediately
that any point in the convex hull of the elements $s_B(t_1b_1+\cdots+t_q b_q)$
with $B\in \CB_\tau$ is a sum of points in $\p(t_1 b_1),\ldots, \p(t_q b_q)$.

Conversely (see section \ref{subsubMinkowski}), it can be shown that any
Minkowski linear sum $t_1\p_1+t_2\p_2+\cdots+t_q\p_q$ can be embedded in a
parametric family of polytopes.
\end{remark}
\begin{remark}[Wall crossing]
  One of the interest of parametric polytopes is that we can also observe the
  variation of $\p(b)$ (see Figure \ref{fig:parametric-dim2}), when the
  parameter $b$ crosses a wall of a chamber $\tau$.  This corresponds to flips
  of the corresponding toric varieties.  The variation of the set $\p(b)$ has
  been studied in detail in \cite{BV-2011}. In this article we will only be
  concerned with the behavior of the function $b\mapsto S^L(\p(b),h)$ when $b$
  runs in the closure of a fixed admissible chamber.
\end{remark}
\subsubsection{Brion's theorem on supporting cones at vertices}
The basis of the present article is  the following theorem, 
which follows from  the Brianchon--Gram decomposition of a
 polytope $\p(b)$.

\begin{proposition}
    The indicator function of a polytope is equal to the sum of the indicator functions of its
    supporting cones at vertices, modulo linear combinations of indicator functions of
    affine cones with lines.
\end{proposition}
\begin{figure}
\begin{center}
  \vspace*{-2cm}
  \begin{minipage}[t]{4cm}
    \mbox{}\vspace*{-6cm}\par
    \includegraphics[width=4cm]{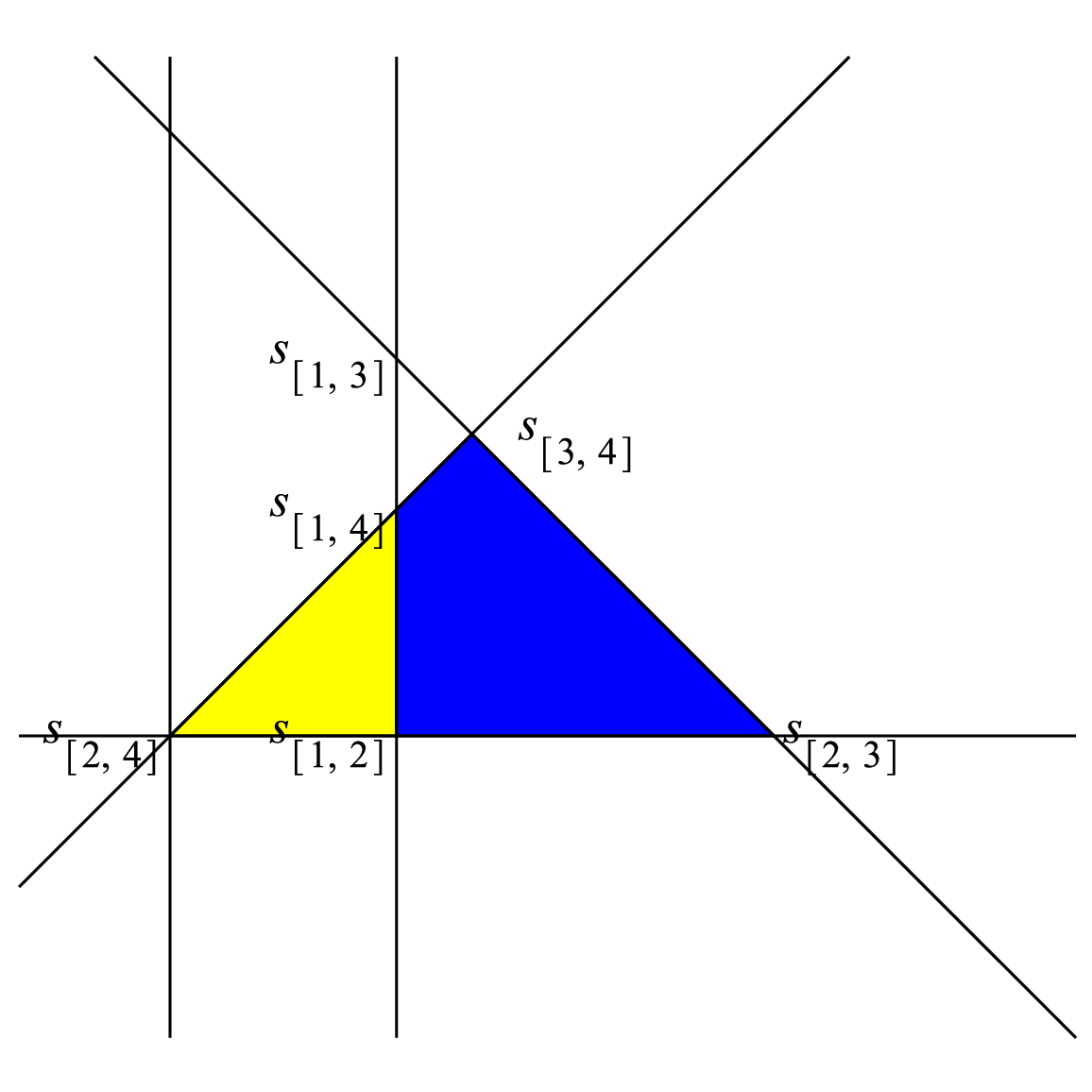}
  \end{minipage}
  \begin{minipage}[t]{8cm}
    \includegraphics[width=8cm]{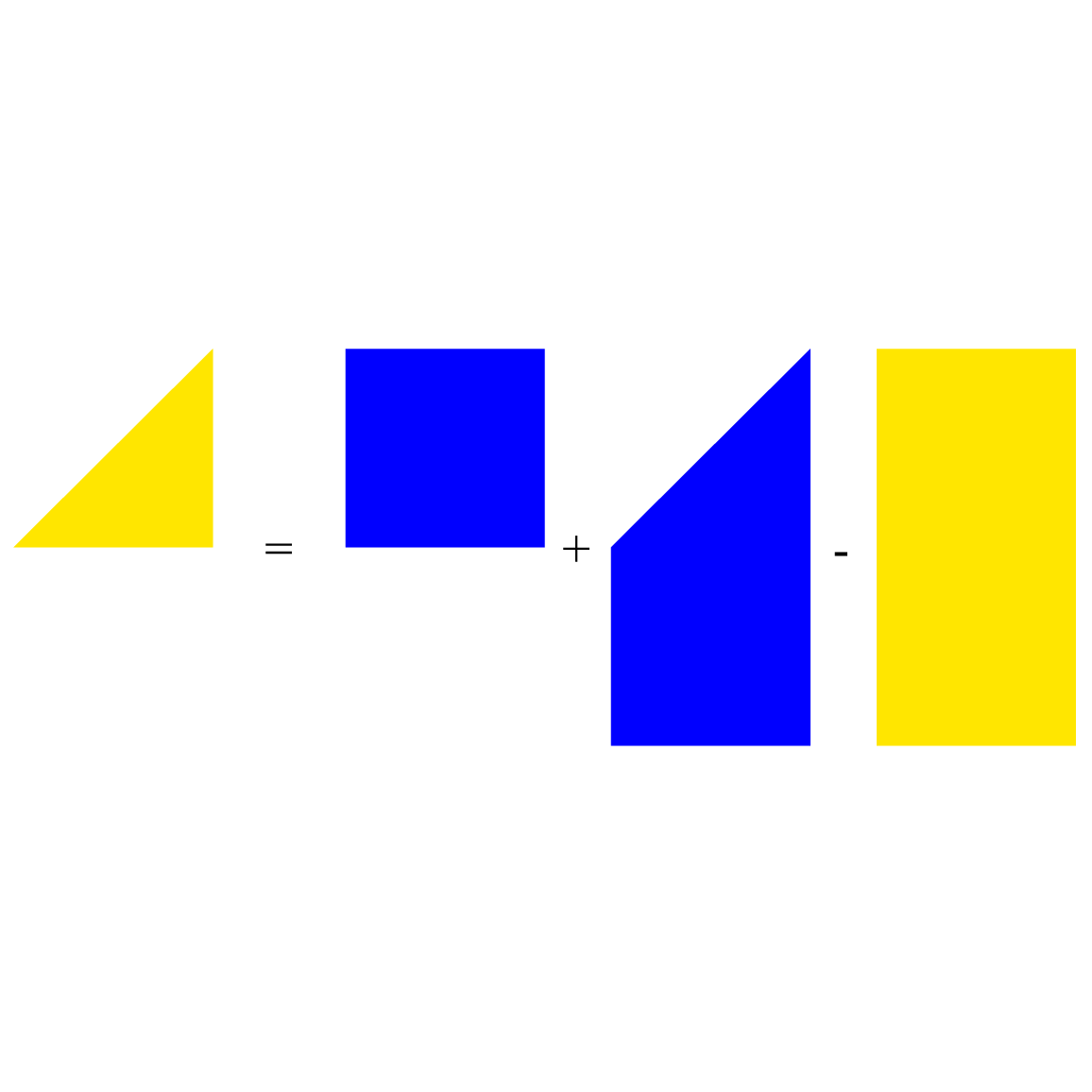}
  \end{minipage}
  \vspace*{-1.5cm}
 \caption{\emph{Left,} in the parametric polytope of Examples \ref{ex:parametric-dim2} and
   \ref{ex:parametric-dim2-chambers}, as the face $-x_1=b_1$ moves to the
   left, for $b \in \overline\tau_2$, the vertex $s_{[2,4]}$ merges with vertices
   $s_{[1,2]}$ and $s_{[1,4]}$, when $b$ reaches the boundary of the chamber
   $\tau_2$ and the quadrilateral degenerates to a triangle. 
   \emph{Right,} the indicator function of the supporting cone of $\p(b)$ at
   vertex $s_{[2,4]}(b) = s_{[1,2]}(b) = s_{[1,4]}$ (\emph{yellow})
   is the sum of the indicator functions of the cones $s_{[1,2]}(b)+\c_{[1,2]}$ and
   $s_{[1,4]}(b)+\c_{[1,4]}$ (\emph{blue}), modulo the indicator function of an affine cone with a 
   line (\emph{yellow}).}\label{fig:BrianchonGramContinu}
 \end{center}
\end{figure}
For parametric polytopes,  combined with the above description of cones at vertices,
this gives a decomposition of the indicator function  $[\p(b)]$ for $b\in \tau$.
If $b$ lies in the boundary of an admissible chamber~$\tau$, then several $B\in \CB_\tau$ may give the same vertex $s_B(b)$.
Nevertheless,  when $b$ lies in the closure $\overline{\tau}$, then  modulo
indicator functions of affine cones with lines,
the indicator function of the supporting cone of $\p(b)$ at vertex $s_B(b)$ is the sum
of the indicator functions of the cones $s_B(b)+\c_B$ for all the $B\in \CB_\tau$ which give this vertex (Figure~\ref{fig:BrianchonGramContinu}).
 One way to prove it is to use the continuity of the Brianchon--Gram decomposition of $[\p(b)]$, proven in \cite{BV-2011}.
\begin{proposition}\label{th:brion-for-parametric-polytopes}
  Let $\tau$ be an admissible  $\alpha$-chamber with closure
  $\overline{\tau}$. For $b\in \overline{\tau}$, the indicator function of $\p(b)$ is given by
  $$
  [\p(b)]\equiv\sum_{B\in \CB_\tau}[s_B(b)+\c_B] \mbox{  mod indicator functions of
    cones with lines}.
  $$
\end{proposition}

From   Proposition~\ref{th:brion-for-parametric-polytopes} and the valuation
property of  intermediate generating functions, we obtain  Brion's formula (cf.~\cite{barvinokzurichbook}),
which expresses the holomorphic function $S^L(\p(b))(\xi)$ as a sum of
meromorphic functions indexed by the vertices of the polytope $\p(b)$.

\begin{equation}\label{th:brion}
  S^L(\p(b))(\xi)=\sum_{B\in \CB_\tau} S^L(s_B(b)+\c_B)(\xi),  \mbox{ for }  b\in \overline{\tau}.
 \end{equation}

\subsection{Step-polynomials and (semi-)quasi-polynomials of the
  multi-parame\-ter $b$}\label{sub:quasipoly}

\tgreen{Rewrote this section. --Matthias}

In this section, we consider a fixed rational subspace $L\subseteq V$. There corresponds an intermediate   generating function
$S^L(\p(b))(\xi)$ and  intermediate weighted sums $S^L(\p(b),h)$  on a
parametric polytope $\p(b)$.

Recall from the introduction the algebras of step-polynomials and quasi-polyno\-mi\-als on $\R^N$, in
terms of which we will describe these intermediate sums as   functions of the parameter $b\in \R^N$.
We give more general definitions now.

\begin{definition}\label{def:step-poly-V}
  Let $\Psi \subseteq \R^N$.  
  \begin{enumerate}[\rm(i)]
  \item $\polypp^\Psi(\R^N)$ is the algebra of functions on $\R^N $ generated by
    the functions $b \mapsto \{\langle\eta,b \rangle\}$, where $\eta\in \Psi$.
    An element of $\polypp^{\Psi}(\R^N)$ is called a \emph{step-polynomial} on
    $\R^N$.
  \item For $\Psi=\Q^N$, we obtain the \emph{algebra of rational
      step-polynomials}, which we abbreviate as $\polypp(\R^N)$. 
  \item For $\Psi=\R^N$, we obtain the \emph{algebra of step-polynomials} $\polypp^{\R^N}(\R^N)$.
  \end{enumerate}
\end{definition}
Note that a step-polynomial function (whether rational or not) is a bounded function on~$\R^N$.

The algebra  $\polypp^{\Psi}(\R^N)$  has a natural filtration,
 where    $\polypp^{\Psi}_{[\leq k]}(\R^N)$ is  the subspace spanned  by products of at most $k$ functions
 $\{\langle \eta,b \rangle\}$.
 Again this is a filtration, not a grading, because several step-polynomials with
 different (step) degrees  may represent the same function.
 \begin{example} For every   $t\in \R$,
$$
1-\{t\}-\{-t\}-(1-\{2t\}-\{-2t\})(1-\{3t\}-\{-3t\})=0.
$$
\end{example}

\begin{definition}\label{def:semi-quasi-poly}
  Again let $\Psi \subseteq \R^N$.  
  \begin{enumerate}[\rm(i)]
\item The tensor product
    $\polypp\CP^\Psi(\R^N)  = \polypp^\Psi(\R^N) \otimes \CP(\R^N)$ is the algebra of functions on $\R^N$ generated by
    step-polynomials in $\polypp^\Psi(\R^N)$ and ordinary polynomials on~$\R^N$. An element of
    $\polypp\CP^\Psi(\R^N)$ is called a \emph{semi-quasi-polynomial} on $\R^N$.
  \item For $\Psi=\Q^N$, we obtain the \emph{algebra of quasi-polynomials} on
    $\R^N$, which we abbreviate as $\polypp\CP(\R^N)$.
  \item For $\Psi=\R^N$, we obtain the \emph{algebra of
      semi-quasi-polynomials} on~$\R^N$, denoted by $\polypp\CP^{\R^N}(\R^N)$.
  \end{enumerate}
\end{definition}

A (semi-)quasi-polynomial $f(b)$ is a piecewise polynomial.
  More precisely, let $\Psi$ be a \emph{finite} set of
  $\eta\in \R^N$ such that $f \in \polypp^\Psi(\R^N)$, i.e.,  $f(b)$ can be expressed as a polynomial in the functions
  $b\mapsto\{\langle \eta,b\rangle\}$, with $\eta\in \Psi$. The open
  ``pieces'' on which $f(b)$ is a polynomial function of~$b$ are the
  $\Psi$-alcoves defined as follows.

\begin{definition}\label{def:Psi-alcove}
 Let $\Psi$ be a finite subset of $\R^N$.
 We consider  the  hyperplanes in $\R^N$ defined by the equations
 $$
 \langle \eta ,b\rangle= n, \mbox{  for  } \eta\in \Psi  \mbox{ and   }   n\in \Z.
 $$
 A connected component of the complement of the union of these hyperplanes is called a \emph{$\Psi$-alcove}.
 \end{definition}

 On the tensor product  $\polypp\CP^{\Psi}(\R^N)$,
  we will consider the grading inherited from  the usual degree on  $\CP(\R^N)$.
 We will call the corresponding  degree the \emph{polynomial degree}.

We consider also the  degree arising from the tensor product filtration, which
we call the \emph{local degree}.
With these notations, if $f(b)\in \polypp\CP^{\Psi}_{[\leq k]}(\R^N)$,  then $f(b)$ restricts as a polynomial function of degree $\leq k$ on any $\Psi$-alcove.
\begin{example}
The quasi-polynomial $f(b)=b^3\{ b\}$ on $\R$  has polynomial degree~$3$ and
local degree~$4$. It is equal to $b^4- n b^3$, a polynomial of degree~$4$, for  $n\leq b < n+1$.
\end{example}
\medskip
%

In the remainder of this section, we will have $\Psi \subset \Q^N$, and hence
work with rational step-polynomials and quasi-polynomials; however,
in \autoref{sub:weighted-ehrhart-semi-quasi}, we will use more general~$\Psi$.

We now explain how to construct a finite set~$\Psi$ that is suitable for our
multi-parameter Ehrhart quasi-polynomials.

In \cite{SLII2014}, for any rational cone $\c\subset V$, and a rational subspace $L$,
we constructed a finite subset $\Psi_{\c}^L \subset\lattice^*\cap L^\perp$ of integral  linear forms on $V$
\cite[Definition~\ref{sl2:def:PsiL}]{SLII2014}.
We briefly recall the steps of this construction; the details can be found in
\cite[proof of Lemma 2.10]{SLII2014}.
We start by decomposing $\c$ in cones $\u$  (modulo cones with lines) with a face parallel to $L$,
 then we decompose the projections of $\u$ on $V/L$ as a signed decomposition of unimodular cones with respect to the projected lattice
 $\lattice_{V/L}$ with dual  lattice $\Lambda^*\cap L^{\perp}$.
 In the set $\Psi_{\c}^L$, we collect all the generators of the dual cones used in this decomposition.
(The decomposition is not unique, but we do not record the dependence of
$\Psi_{\c}^L$ on the decomposition in the notation.)
As $\Psi^L_\c\subset \Lambda^*\cap L^{\perp}$,
the step-polynomials in the algebra $\polypp^{ \Psi^L_\c}(V)$ (see \autoref{subsection:polyhedra}) are functions on $V/(\Lambda+L)$.

\begin{definition}\label{def:polyppB}
  \begin{enumerate}[\rm(i)]
  \item Let $B\in \CB$.  Define $\Psi_B^L(\alpha)$ as the set of rational linear forms
    on $\R^N$ defined by $ \langle \eta,b\rangle = \langle
    \gamma,s_B(b)\rangle$, for $\gamma\in \Psi^L_{\c_B}$.
  \item
    If $\tau$ is an $\alpha$-chamber, $\Psi^L_\tau(\alpha)$ is the union of the sets
    $\Psi^L_B(\alpha)$ when $B$ runs in $\CB_\tau$.  
  \end{enumerate}
\end{definition}

By definition, if $f\in  \polypp\CP^{\Psi^L_{\c_B}}(V)$ (see \autoref{subsection:polyhedra}), then the function
$b\mapsto f(s_B(b))$ is in $\polypp\CP^{\Psi^L_{B}(\alpha)}(\R^N)$.

Let us describe a little more precisely the sets $\Psi_{\c_B}^L $ and the corresponding step-polynomials on $\R^N$ when $L=V$ or $L=\{0\}$.

When $L=V$, the sets $\Psi^L_{\c_B}$ are empty, thus  the  step-polynomials in $\polypp^{ \Psi^L_\tau(\alpha)}(\R^N)$ are
just the constants.

When $L=\{0\}$, let $q\in \N$ be such that the lattice  $q\lattice^*$
is contained in the lattice generated by
the elements $\alpha_j$, for $j\in B$. Then the multiple  $q s_B(b)$ of the vertex $s_B(b)$  belongs to $\lattice$  if $b\in \Z^N$,
therefore any step-polynomial $f(b)\in \polypp^{ \Psi^{\{0\}}_B(\alpha)}(\R^N)$
is $q\Z^N$-periodic.
Thus a quasi-polynomial $f$ on $\R^N$ gives by restriction to $\Z^N$ a periodic function of period $q$,
and we recover the usual notion of quasi-periodic function on a lattice.

If $\c_B$ is the simplicial cone described by inequalities
$\la\alpha_j,x\ra\leq 0$, with $j\in B$, and if the $\alpha_j, j\in B$ form a basis of $\Lambda^*$,
then  the set $\Psi_{\c_B}^{\{0\}} \subset\lattice^*$ is just equal to $\{\,\alpha_j: j\in B\,\}$.
Thus if the sequence $\alpha$ is \emph{unimodular},
that is, if $\{\, \alpha_j : j\in B\,\}$ is a basis of $\lattice^*$ for any
$B\in \CB$, any step-polynomial $f(b)\in \polypp^{ \Psi^L_\tau(\alpha)}(\R^N)$ is
 $\Z^N$-periodic, in particular $f(b)$  is constant on $ \Z^N$.

\subsection{Weighted Ehrhart (semi-)quasi-polynomials}\label{sub:weighted-ehrhart-semi-quasi}
\subsubsection{Case of a parametric polytope}
We can now state the first important result of this article.  We summarized
it in the introduction, in a less technical form, as Theorem~\ref{th:ehrhart-chamber-intro-summary}.
\begin{theorem} \label{th:ehrhart-chamber}
Let $V$ be a rational vector space with lattice $\lattice$. Let $L\subseteq V$ be a rational subspace.
Let $\alpha=(\alpha_1,\dots,\alpha_N)$ be a list of  elements of $\lattice^*$ which generate $V^*$ as a cone.
For $b\in \R^N$, let
$\p(b)\subset V$ be the polytope defined by
$$\p(b)=\{\,x\in V: \la \alpha_j,x\ra\leq b_j,\, j=1,\ldots,N\,\}.$$
Let $\tau\subset \R^N$ be an admissible $\alpha$-chamber.
Let $h$ be a polynomial function on $V$ of degree $m$.

\begin{enumerate}[\rm(i)]
\item There exists a quasi-polynomial of local degree equal to $m+d$,
\begin{equation}\label{eq:ehrhart-chamber}
E^L(\alpha,h,\tau) \in
\polypp\CP_{[\leq {m+d}]}^{ \Psi^L_\tau(\alpha)}(\R^N)
\end{equation}
such that
\begin{equation}\label{eq:Ehrhart-quasi-poly}
S^L(\p(b),h)=E^L(\alpha,h,\tau)(b),
\end{equation}
for every $b\in \overline{\tau}$ (the closure of $\tau$ in $\R^N$).
\item\label{thpart:ehrhart-chamber:highest-degree}
  If  $h(x)$ is homogeneous of degree $m$, then the terms
  of $E^L(\alpha,h,\tau)(b)$  of polynomial degree $m+d$ form a homogeneous
  polynomial of degree $m+d$ that is equal to the integral $\int_{\p(b)}
  h(x)\, \mathrm dx$
  for $b \in \overline{\tau}$.
\item
More precisely, if $h(x)=\frac{\langle \ell,x \rangle ^m}{m!} $ for some $\ell\in V^*$, we have
\begin{equation}\label{eq:top-parametric-Ehrhart}
 E^L(\alpha,h,\tau)(b)=
\sum_{r=0}^{m+d} E^L_{[r]}(\alpha,h,\tau)(b)
\end{equation}
for $b \in \overline{\tau}$, where
for each $r$, the function of $b\in\R^N$ given by
$$ E^L_{[r]}(\alpha,h,\tau)(b)=\left(\sum_{B\in \CB_\tau}
\retroS^L(s_B(b),\c_B)_{[m-r]}(\xi)
 \frac{\langle\xi,s_B(b)\rangle^r}{r!}\right)\bigg|_{\xi=\ell}$$
is an element of $\polypp_{[\leq {m+d-r}]}^{ \Psi^L_\tau(\alpha)}(\R^N)\otimes
\CP_{[r]}(\R^N)$, i.e., 
of polynomial degree $r$ and local degree at most $m+d$.
\end{enumerate}
\end{theorem}
In fact, for a single $B$,
$\retroS^L(s_B(b),\c_B)_{[m-r]}$ is a rational function of $\xi$ and may be singular at $\ell$, so that the value
$\retroS^L(s_B(b),\c_B)_{[m-r]}(\ell)$ may not be well defined.
 However, as we will see in the proof, for each $r$, the sum over the $B\in \CB_\tau$
of the rational functions  $\retroS^L(s_B(b),\c_B)_{[m-r]}(\xi)
 \frac{\langle\xi,s_B(b)\rangle^r}{r!}$
is a polynomial function of $\xi$,  so it  can be evaluated at $\xi=\ell$.
So Formula \eqref{eq:top-parametric-Ehrhart} is well defined.

\begin{proof}
The proof of Theorem  \ref{th:ehrhart-chamber} rests on  Brion's formula (\ref {th:brion}).
 We observe that it is enough to  prove the theorem in the case
where the weight  $h$ is a power of a linear form,
$$
h(x)= \frac{\langle\ell,x\rangle^m}{m!},
$$
for some $\ell\in V^*$, as any weight can be written as a linear combination of those.
In this case, $ S^L(\p(b),h)$ is the value at $\xi=\ell$
of the homogeneous term of degree $m$ of the holomorphic function
$ S^L(\p(b))(\xi)$.
So we write (using the fundamental Equation \eqref{eq:MversusS})
\begin{align*}
  S^L(\p(b),h)_{[m]}(\ell) &= \left(\sum_{B\in \CB_\tau} S^L(s_B(b)+\c_B)_{[m]}(\xi)\right)\bigg|_{\xi=\ell}\\
&= \sum_{r=0}^{m+d}\left(\sum_{B\in \CB_\tau}
\retroS^L(s_B(b),\c_B)_{[m-r]}(\xi) \frac{\langle\xi,s_B(b)\rangle^r}{r!}\right)\bigg|_{\xi=\ell}   .
\end{align*}
For an individual  $B$, the term
  $S^L(s_B(b)+\c_B)_{[m]}(\xi)$ may be singular at $\xi=\ell$.
However the sum over the set of vertices $\CB_\tau$ is  a polynomial function of $\xi$.


In \cite{SLII2014} we studied the bidegree structure of
$\retroS^L(s,\c)(\xi)$, i.e., the interaction of the local degree in $s$ and
the homogeneous degree in~$\xi$, which allows us to
extract the refined asymptotics. 
For each $r$ and  $B$, we consider the function of $s\in V$, $\xi\in V^*$,
given by $(s,\xi)\mapsto \retroS^L(s,\c_B)_{[m-r]}(\xi)$.
By Theorem \ref{sl2:prop:homogeneous-ML} of \cite{SLII2014},
this function belongs to the space
$$\polypp_{[\leq m+d-r]}^{ \Psi^L_{\c_B}}(V)\otimes \CR_{[m-r]}(V^*).$$

Compose with the linear map $b\mapsto s_B(b)$.
  We obtain that,
 for each $r$ and $B$, the function of $b\in \R^N$, $\xi\in V^*$
given by
$$(b,\xi)\mapsto
\retroS^L(s_B(b),\c_B)_{[m-r]}(\xi) \frac{\langle\xi,s_B(b)\rangle^r}{r!}
$$
 belongs to
$$
\polypp_{[\leq {m+d-r}]}^{ \Psi_B^L(\alpha)}(\R^N)\otimes \CP_{[r]}(\R^N)
\otimes \CR_{[m]}(V^*).
$$
Therefore  the sum over $B\in \CB_\tau$ of these terms, for a fixed $r$,
belongs to
$$
\polypp_{[\leq {m+d-r}]}^{ \Psi^L_\tau(\alpha)}(\R^N)\otimes \CP_{[r]}(\R^N)
\otimes \CR_{[m]}(V^*).
$$
Now the sum over $B\in \CB_\tau$ is a quasi-polynomial function of $b$
with values in the space of {polynomials} in $\xi$, not just rational functions of~$\xi$.
It follows  that for each $r$, the term of polynomial degree $r$ in $b$ of the full sum
depends also polynomially on $\xi$.
This term is
$$
\sum_{B\in \CB_\tau} \retroS^L(s_B(b),\c_B)_{[m-r]}(\xi)  \frac{\langle\xi,s_B(b)\rangle^r}{r!}.
$$

   When we evaluate it at $\xi=\ell$,  we obtain
   \eqref{eq:top-parametric-Ehrhart}, and all statements but
   part~\eqref{thpart:ehrhart-chamber:highest-degree}, which we prove now.

Let us compute the term of polynomial degree $r=m+d$ with respect to $b$, in  \eqref{eq:top-parametric-Ehrhart}.
 From Equation \eqref{eq:lowest}, we know that
$\retroS^L(s_B(b),\c_B)_{[-d]}(\xi)=I(\c_B)(\xi)$, thus
 the term  of index $r=m+d$  in  \eqref{eq:top-parametric-Ehrhart} is equal to
$$
\left(\sum_{B\in \CB_\tau} \frac{\langle\xi,s_B(b)\rangle^{m+d}}{(m+d)!}I(\c_B)(\xi)\right)\bigg|_{\xi=\ell}=
\left(\sum_{B\in \CB_\tau} I(s_B(b)+\c_B)_{[m]}(\xi)\right)\bigg|_{\xi=\ell}.
$$
By Proposition \ref{th:brion-for-parametric-polytopes}, this last sum is equal to
$I(\p(b))_{[m]}(\ell)$, which is precisely the integral  $\int_{\p(b)}\frac{\langle \ell,x \rangle ^m}{m!}  \, \mathrm dx$.
\end{proof}

\begin{definition}\label{def:ehrhart-chamber}
   The function $E^L(\alpha,h,\tau)(b)$ of Theorem \ref{th:ehrhart-chamber}
    is called the \emph{weighted intermediate Ehrhart quasi-polynomial} of the parametric polytope $\p(b)$
   (with respect to the weight $h$, the subspace $L$ and  the chamber $\tau$).
\end{definition}
\begin{example}[Example~\ref{ex:parametric-dim1}, continued]\label{ex:generalizedEhrh-interval}
  Let $\alpha = (x, -x)$. Then for $b = (b_1, b_2) \in \R^2$,  $\p(b)$ is
   the interval $\{\, x: -b_2\leq x\leq b_1\,\}$. There are two chambers, $\{\, b :
   -b_2<b_1\,\}$ and $\{\, b : -b_2>b_1\,\}$.
   For the first chamber the number of integers in $\p(b)$ is
   $\lfloor b_1\rfloor - \lceil -b_2\rceil+1= b_1+ b_2 -\{b_1\}-\{b_2\}+1$.
   For the other chamber, it is of course~$0$.
 \end{example}
\begin{example}[continuation of Examples \ref{ex:parametric-dim2},
  \ref{ex:parametric-dim2-vertices}, and \ref{ex:parametric-dim2-chambers}]
  \let\frac=\tfrac
We compute  the quasi-polynomial function  $ E^{L}(\alpha,h,\tau_2)(b)$
($\p(b)$ is a quadrilateral for $b\in \tau_2$),
 first for $L=\{0\}$, then for the case when $L$ is the vertical line $L=\R(0,1)$. The weight is $h(x)=1$.

\begin{enumerate}[\rm(i)]
\item
  For $L={\{0\}}$ and $h=1$, i.e., we count the integer points  in $\p(b)$, we write
   $$E^{\{0\}}(\alpha,1,\tau_2)(b) = E_{[2]}(b)+E_{[1]}(b)+E_{[0]}(b),$$ where $E_{[r]}(b)$
   collects the terms of polynomial degree $r$ with respect to~$b$.
   $E_{[2]}(b)$ is the volume  of the quadrilateral. It is a polynomial function,
   easy to compute. The other two functions were computed with our Maple
   program. 
\begin{align*}
E_{[2]}(b)&= 
-\frac{b_1^2}{2}+\frac{b_2^2}{2}+\frac{b_3^2}{4}-\frac{b_4^2}{4}+b_1b_2+b_1b_4+b_2b_3+\frac{b_3b_4}{2}.\\
E_{[1]}(b)&=\bigl(\frac{1}{2}+\{b_1\}-\{b_2\}-\{b_4\}\bigr)b_1\\
&\quad +\bigl(\frac{3}{2}-\{b_1\}-\{b_2\}-\{b_3\}\bigr)b_2\\
&\quad +\bigl(1-\{b_2\}-\frac{1}{2}\{b_3\}-\frac{1}{2}\{b_4\}\bigr)b_3\\
&\quad +\bigl(\frac{1}{2}-\{b_1\}-\frac{\{b_3\}}{2}+\frac{\{b_4\}}{2}\bigr)b_4.\\
E_{[0]}(b)&=1-\frac12{\{b_1\}}-\frac32{\{b_2\}}-\{b_3\}-\frac12{\{b_4\}}\\
&\quad-\frac{1}{2}\{b_1\}^2+\frac{1}{2}\{b_2\}^2-\frac{1}{2}\{b_4\}^2
-\{\frac{b_4+b_3}{2}\}^2\\
&\quad+\{b_1\}\{b_2\}+\{b_1\}\{b_4\}+\{b_2\}\{b_3\}\\
&\quad+\{b_3\}\{\frac{b_4+b_3}{2}\}+\{b_4\}\{\frac{b_4+b_3}{2}\}.
\end{align*}
We see that $E_{[2]}(b)$  is a linear combination of products of two linear forms,  $E_{[1]}(b)$
is a linear combination of products of
linear forms with step-linear forms,  while  $E_{[0]}(b)$  is a linear combination
of products of at most two step-linear forms.
Thus each of the $E_{[r]}(b)$ is of local degree~$2$.

\item We compute the intermediate quasi-polynomial $$E^L(\alpha,1,\tau_2)(b)=E^L_{[2]}(b)+E^L_{[1]}(b)+E^L_{[0]}(b)$$ for the same chamber $\tau_2$,
 when $L$ is the vertical line $L=\R(0,1)$, and again $h=1$. Thus we compute the sum $S^L(\p(b),1)$
 of the lengths of  vertical segments in the quadrilateral  $\p(b)$.
 Then $E^L_{[2]}(b)= E_{[2]}(b)=\vol(\p(b))$ is again the volume of $\p(b)$, and we compute
 \begin{align*}
   E^L_{[1]}(b)&=-\frac{1}{2}b_1+\frac{1}{2}b_2+\frac{1}{2}b_4+ \{b_1\}b_1-\{b_1\}b_2
   -\{b_1\}b_4,\\
   E^L_{[0]}(b)&=\frac{1}{2}\{b_1\} +\frac{1}{2}\{b_2+b_3\}
   -\{\frac{b_3-b_4}{2}\}\\
   &\quad -\frac{1}{2}\{b_1\}^2-\frac{1}{2}\{b_2+b_3\}^2+\{\frac{b_3-b_4}{2}\}^2.
\end{align*}
Again, we observe that the local degree of $E^L_{[r]}(b)$ is indeed $2$ for $r=0,1,2$.
\end{enumerate}
\end{example}

 \begin{remark}
In this theorem, the parameter $b$ varies in $\R^N$.
In particular, the results of \cite{HenkLinke} on ``vector dilated
 polytopes'' follow easily from this theorem.\footnote{Note that \cite{HenkLinke} states
   and proves results for \emph{rational} vector dilations only. }
 The article \cite{HenkLinke} was part of our motivation to consider the case of a
 multidimensional real-valued parameter and not just one parameter dilations. 
\end{remark}

\begin{remark}
   When $L=V$, we are computing an integral over $\p(b)$.
   It is clearly a polynomial function of $b$ on any chamber.
   This is consistent with the fact that  $\polypp^{ \Psi^V_\tau(\alpha)}(\R^N) $ is just the constants.

 Classically,   in particular when computing the number of lattice points (case $L=\{0\}$, $h(x)=1$),
the parameter  $b$ was restricted to $\Z^N$. As we already  observed,
if $q\lattice^*$ is contained in the lattice generated by $(\alpha_j,j\in B)$
for any $B\in \CB_\tau$, then the coefficients  of the Ehrhart quasi-polynomial \eqref{eq:ehrhart-chamber} are $q\Z^N$-periodic functions on~$\R^N$,
 therefore the Ehrhart quasi-polynomial restricts to any coset $\{\,
 b_0+q n : n\in\Z^N \,\}$
 as a true polynomial function of $n\in \Z^N$.
\end{remark}

\begin{remark}[Case of partition polytopes]\label{rem:parametric-versus-partition}
The paper~\cite{BV-2011} deals with  Ehrhart quasi-polynomials for weighted sums and integrals over  a partition polytope.
Their variation is computed, when the parameter crosses a wall between two chambers.
 Let us recall how a parametric polytope $\p(\alpha,b)$ is associated to a partition polytope $\pp(\Phi,\lambda)$.
 Let $F$ be a vector space of dimension $N-d$ and let $\Phi=(\phi_1,\ldots,\phi_N)$ be a sequence of elements of $F$.
 We assume that $\Phi$ generates a full-dimensional pointed cone in $F$.
 For $\lambda\in F$, let
$$\pp(\Phi,\lambda)=\Bigl\{\,y=(y_j)\in \R^N : y_j\geq 0,\, \textstyle\sum_{j=1}^N y_j \phi_j=\lambda \,\Bigr\}.
 $$
  This 
  is a polytope contained in the affine subspace
  $\bigl\{\, y=(y_j)\in \R^N : \sum_{j=1}^N y_j \phi_j=\lambda \,\bigr\}$.
 Define  $V:=\bigl\{\,y=(y_j)\in \R^N : \sum_{j=1}^N y_j \phi_j=0\,\bigr\}$.
Let $\alpha_j$ be the linear form on $V$ defined by $\langle\alpha_j,x\rangle=-x_j$.
 For $b\in \R^N$, let  $\lambda=\sum_{j=1}^N b_j \phi_j$.
Then $x\mapsto x+b$ is a bijection between $\p(\alpha,b)$ and $ \pp(\Phi, \lambda)$.

\end{remark}
%
%

\begin{remark}[Wall crossing]
Finally, it would be interesting to study
the variation of the quasi-polynomials $S^L(\p(b),h)$ when $b$ crosses the wall of a chamber $\tau$.
The method of\/ \cite{BV-2011} could probably  be adapted to the more general case
of intermediate weighted sums  of a parametric polytope.
\end{remark}

\subsubsection{Specialization to other parameter domains}\label{s:other-parameter-domains}
\tgreen{Completed this section. --Matthias}

From the study of a general parametric polytope, it is not difficult to derive
results when the multi-parameter $b \in \R^N$ is itself a function of another
parameter $t \in \R^q$, $b = b(t_1,\dots,t_q)$.  
We restrict ourselves to the setting where $b$ is a (homogeneous) linear function
of~$t$, which we write as $b(t) = Tt$, where $T \in \R^{N\times q}$ is a matrix. 
This is sufficient for two popular settings, which we explain in the
following sections.  In
section~\ref{subsubdilated}, we will consider the case of a fixed
semi-rational polytope $\p$ dilated by a one-dimensional real parameter $t\geq
0$.  In section~\ref{subsubMinkowski}, we will consider the more general case
of a Minkowski linear system $t_1\p_1+\dots + t_q\p_q$.

To describe how the specialization yields the function $t \mapsto E^L(\alpha, h, \tau)(Tt)$,
let us first describe the alcoves.  Let $T^* \in \R^{q\times N}$ be the
adjoint (transpose) matrix. The linear
forms $\eta \in \Psi^L_\tau(\alpha) \subset \Q^N$ defining the alcoves of~$\R^N$
(see \autoref{th:ehrhart-chamber}) give rise to 
linear forms on~$\R^q$,
\begin{equation}
  \langle T^* \eta, t \rangle
  = \langle \eta, Tt \rangle \quad\text{for $t \in \R^q$}.
\end{equation}
Thus we consider the alcoves of~$\R^q$ defined by the finite set
$T^*(\Psi^L_\tau(\alpha)) \subset \R^q$.  Note that, when $T$ is not rational,
$T^*(\Psi^L_\tau(\alpha))$ will no longer be rational, in contrast to the development
in \autoref{sub:quasipoly}.  The function  $t \mapsto E^L(\alpha, h, \tau)(Tt)$ 
will therefore belong to the subalgebra $\polypp\CP^{T^*(\Psi^L_\tau(\alpha))}(\R^q)$
of semi-quasi-polynomials.  (When $T$ is rational, this is a subalgebra of quasi-polynomials.)

Using this notation, we can formulate a theorem analogous to
\autoref{th:ehrhart-chamber}.  We omit the statement.

In contrast to \autoref{th:ehrhart-chamber}, we no longer know the precise
local degree of the semi-quasi-polynomial $t \mapsto E^L(\alpha, h, \tau)(Tt)$. 
The ``expected'' degree is $m+d$, but there may be cancellations of terms, as
illustrated by the example  $\p =[-1,1]$, $h(x) = x$ given in the introduction.  

\subsubsection{Case of a dilated polytope}\label{subsubdilated}

A first example appeared in the introduction as
Example~\ref{ex:irrational-rectangle}, which already illustrated that in the
case of semi-rational polytopes~$\p$ which are not rational, we may not get
quasi-polynomials of the dilation factor~$t$ but merely semi-quasi-polynomials.
Let us give a few more examples for the rational case.
\begin{figure}
\begin{center}
 \includegraphics[width=6cm]{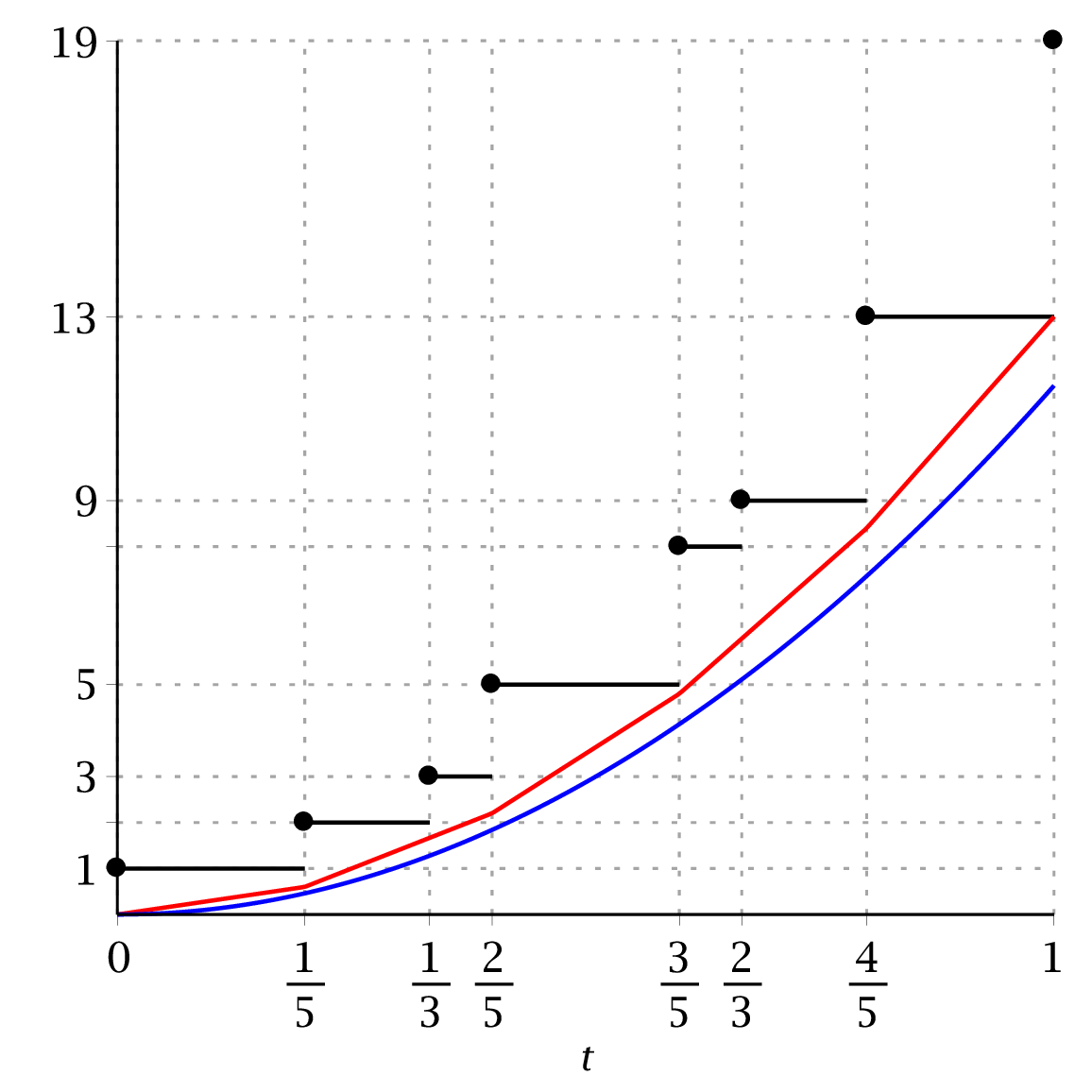}\includegraphics[width=6cm]{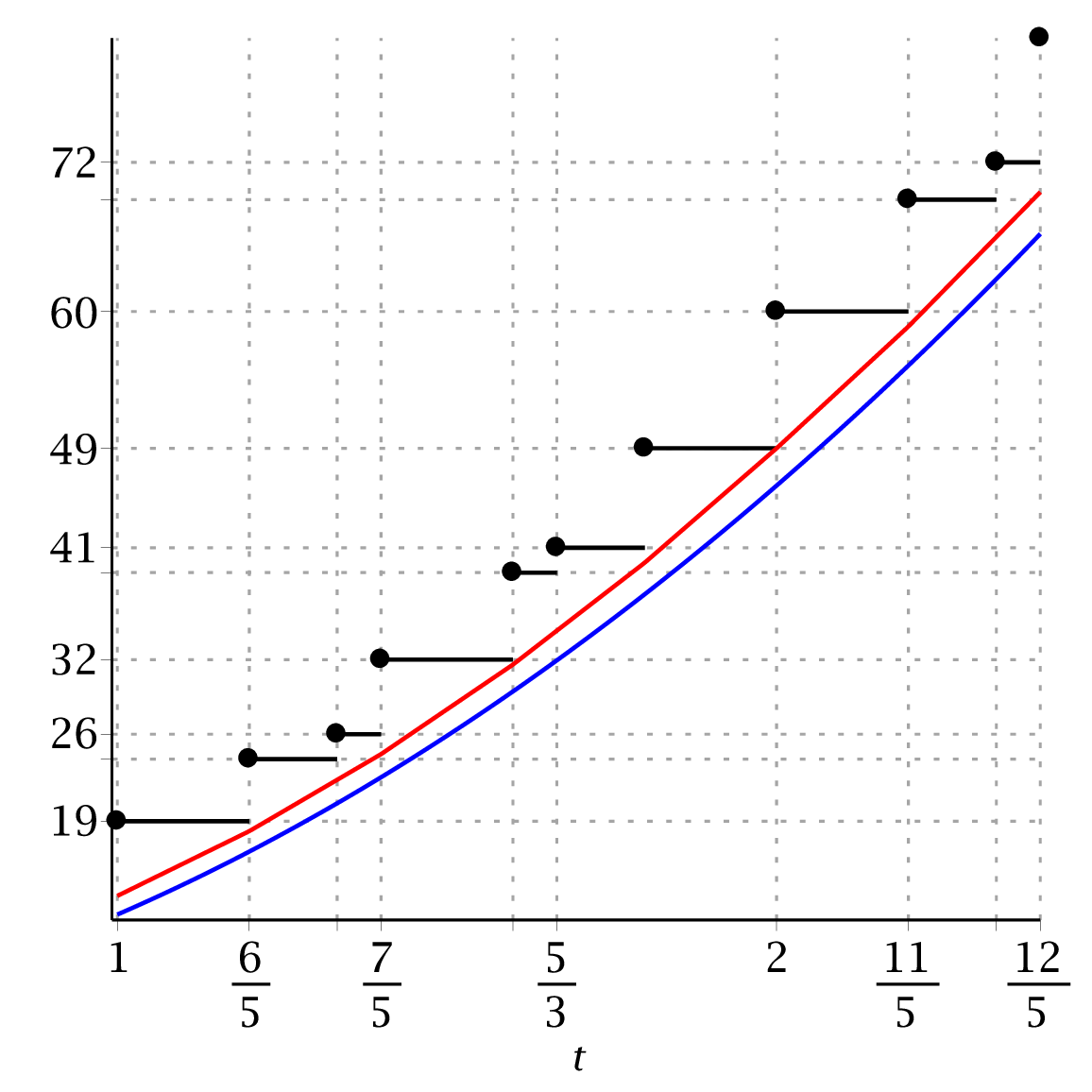}
\caption{$S^L(t\p,1)$ for the quadrilateral  of  Example
  \ref{ex:parametric-dim2-rational-Ehrhart}, for $L=\{0\}$ (\emph{black}), $L$
  vertical (\emph{red}), $L=\R^2$  (\emph{blue}). On the left, $ t$ varies from $0$ to $1$,
  new lattice points occur for $t=0, \frac15, \frac13, \frac{2}{5}, \frac35, \frac23, \frac45,1$. On the right,  $ t$ varies from $1$ to $2.4$, new lattice points occur for $t= \frac65, \frac43, \frac75, \frac85,\frac53,2,\frac{11}5,\frac73,\frac{12}5 $. For $L$ vertical, $S^L(t\p,1)$ is continuous, but its derivative has discontinuities. }
\label{fig:dilatedtetra-rational}
 \end{center}
\end{figure}
\begin{example}[Continuation of Examples \ref{ex:parametric-dim2},
  \ref{ex:parametric-dim2-vertices} and
  \ref{ex:parametric-dim2-chambers}]\label{ex:parametric-dim2-rational-Ehrhart}
\let\frac=\tfrac
Fix $b_0=(0,0,5,3)$, so that $\p=\p(b_0)$ is the quadrilateral of Figure \ref{fig:parametric-dim2}
with vertices $[0, 0], [0, 3], [1, 4], [5, 0]$.
We specialize the formula for   $E^{L}(\alpha,1,\tau_2)(b)$
to the line $b=tb_0$. We consider the cases $L=\{0\}$, $L$ the vertical line, and $L=V$.

\begin{enumerate}[\rm(a)]
\item
First, with $L=V$, we compute the volume.
For $t\geq 0$,   $S^V(t\p)=\frac{23}{2}t^2$. It is a polynomial function of $t$ with rational coefficients.

\item
Next, with $L=\{0\}$, we count the lattice points of $t\p$, for $t\geq 0$.
\begin{equation*}
  \begin{aligned}
    S^{\{0\}}(t\p,1)&=\frac{23}{2}t^2+(\frac{13}{2}-\{3t\} -4\{5t\})t \\
    &\qquad-\frac{1}{2}\{3t\}^2-\{4t\}^2+\{4t\}\{3t\}+\{5t\}\{4t\}\\&\qquad -\{5t\}-\frac{1}{2}\{3t\}+1.
  \end{aligned}
\end{equation*}
It takes only integral values, and is locally constant over some rational intervals.
These facts are more apparent on the graph (Figure~\ref{fig:dilatedtetra-rational}) than on the formula.
When $t$ is in $\Z$, all terms $\{qt\}$, for $q\in \Z$, are equal to $0$,
and we obtain the usual Ehrhart polynomial of $\p$  over $\Z$ (it is a polynomial as $\p$ has integral vertices)
$$\frac{23}{2}t^2+\frac{13}{2} t+1.$$
The value at $t=1$  is $19$, the number of integral points in $\p$.

\item
When  $L$ is the vertical line, we add the lengths of the vertical segments in $t\p$, for $t\geq 0$.
$$
S^{L}(t\p,1)
=\frac{23}{2}t^2+\frac{3}{2}t+\frac{1}{2}\{5t\}^2+\{t\}^2-\{4t\}^2+\frac{1}{2}\{5t\}-\{t\}.
$$
This is a continuous function of $t$. Its value at $t=1$ is $13$.
\end{enumerate}
\end{example}

\tgreen{Reworked the following paragraphs and updated the notation in the theorem. --Matthias}

Now we describe how this specialization works in general.  Let  $L$ be a rational
subspace of $V$.  We take a
(semi-)rational polytope $\p = \p(b_0)$ associated to a fixed real multi-parameter~$b_0$, 
and specialize the formula for $E^L(\alpha,h,\tau)(b)$, where $b_0 \in \bar\tau$, 
when $b = tb_0$ for $t \in \R$, $t>0$.  Using the notation from
\autoref{s:other-parameter-domains}, we have the matrix $T = (b_0) \in \R^{N\times 1}$. 
We then compute the finite set $\Psi := T^*(\Psi^L_\tau(\alpha)) \subset \R$.  
It can be described in a simpler way as follows.  
Denote by $\Psi^L_\p$ the union of the sets
$\Psi_\c^L \subset \Lambda^*\cap L^{\perp}$, described in
\autoref{sub:quasipoly}, where $\c$ varies over the cones  of feasible
directions at the vertices of $\p$ (they are rational polyhedral cones). 

Then $\Psi$ is the finite set of real numbers $\langle \gamma, s\rangle$, where
$\gamma$ runs in $\Psi^L_\p$, and $s$ runs over the vertices of~$\p$.  It
describes the alcoves of~$\R$.  Thus the function $t\mapsto S^L(t\p,h)$ is a
(semi-)quasi-polynomial, which coincides with a polynomial function of
$t$ on intervals with possibly irrational ends, as in Example~\ref{ex:irrational-rectangle}.
Its coefficient functions are bounded functions of the variable $t\in \R$.
If $\p$ is rational, then $\Psi$ is rational, and thus the coefficient
functions are periodic functions.  

We summarize this discussion in the following result.

\begin{theorem}\label{th:ray}
  Let $\p$ be a (semi-)rational polytope and $h$ a polynomial function of degree $m$ on $V$.
  Let $\Psi \subset \R$ be the set of real numbers $\langle \gamma, s\rangle$, where
  $\gamma$ runs in $\Psi^L_\p$ and $s$ runs over the vertices of~$\p$.  
  \begin{enumerate}[\rm(i)]
  \item There exists a (semi-)quasi-polynomial
    $$
    E^L(\p,h)(t)=\sum_{r=0}^{d+m} E^L_r(\p,h)(t) \, t^r,
    $$
    such that $S^L(t\p,h)= E^L(\p,h)(t)$ for all $t\in \R$ with $t\geq 0$.
    It belongs to $\polypp\CP^\Psi_{[\leq m+d]}(\R)$.
  \item The coefficient functions $E^L_r(\p,h)(t)$ are step-polynomials; they
    belong to the space $\polypp_{[\leq m+d-r]}^\Psi(\R)$. 
  \item Let $\p$ be rational and $q\in \N$ is such that $q\p$ has lattice vertices.
    Then the coefficient functions $E^L_r(\p,h)(t)$ are rational
    step-polynomials on~$\R$; they are periodic functions with period~$q$.
    Thus $E^L(\p,h)(t)$ is 
    a quasi-polynomial.
  \end{enumerate}
\end{theorem}


One can also prove the  theorem directly with a proof similar to that of
Theorem \ref{th:ehrhart-chamber}, based on Brion's theorem for $\p$,  without embedding~$\p$ in a parametric family.
This was done in \cite{so-called-paper-2} under the assumption that $\p$ is rational.

\begin{figure}
\begin{center}
 \includegraphics[width=6cm]{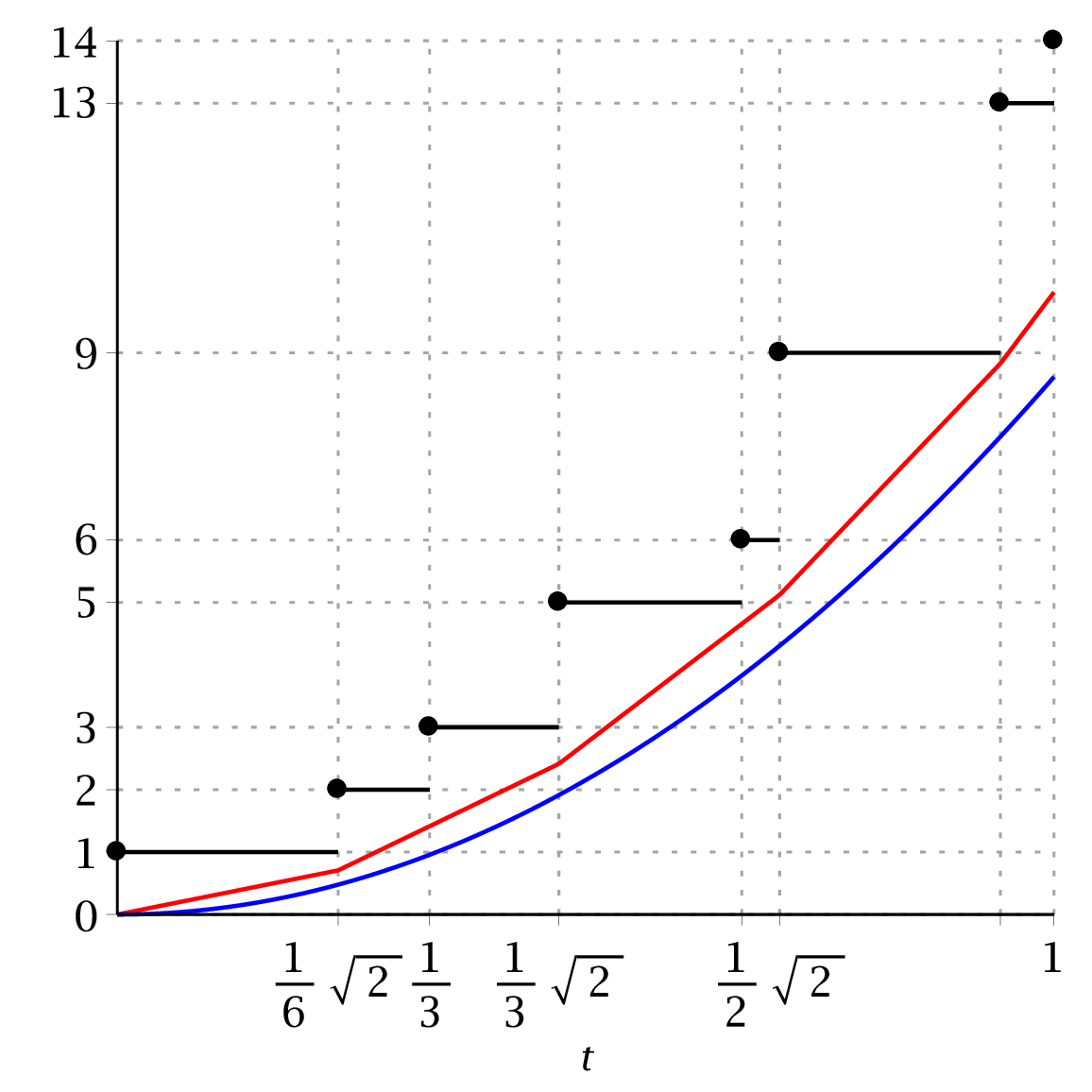}\includegraphics[width=6cm]{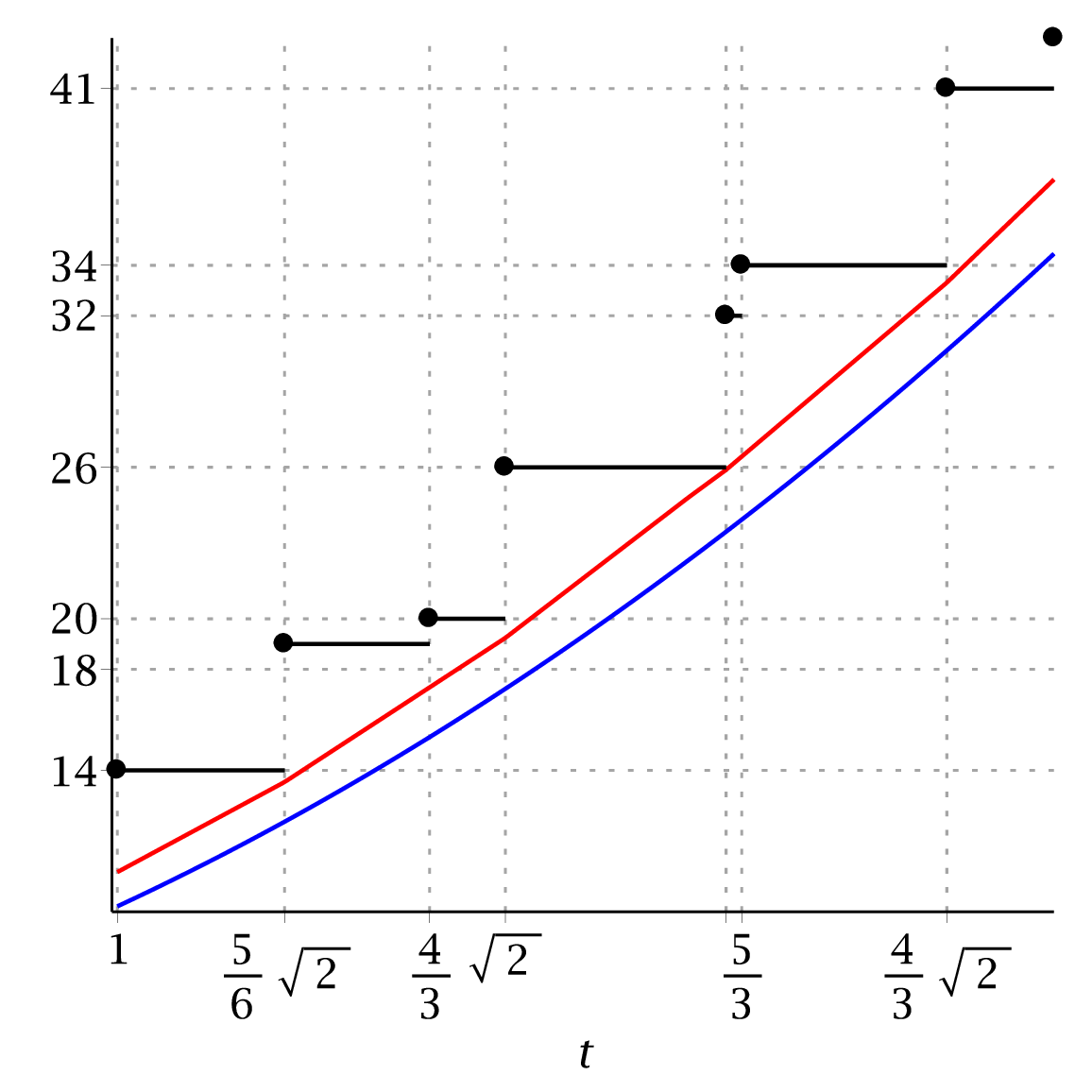}
\caption{$S^L(t\p_I,1)$ where $\p_I$ is the  quadrilateral  with irrational vertices of  Example \ref{ex:parametric-dim2-irrational-Ehrhart}, for $L=\{0\}$ (\emph{black}), $L$ vertical (\emph{red}), $L=\R^2$  (\emph{blue}). On the left, $ t$ varies from $0$ to $1$. On the right,  $ t$ varies from $1$ to $2$. }
\label{fig:dilatedtetra-irrational}
 \end{center}
\end{figure}
\begin{example}[continuation of Examples \ref{ex:parametric-dim2}, \ref{ex:parametric-dim2-chambers}]\label{ex:parametric-dim2-irrational-Ehrhart}
\let\frac=\tfrac
Consider now  the  quadrilateral $\p_I=\p(b_I)$  with
$b_I=[0,0,3\sqrt{2},3]$. Its four vertices are
$[0,0]$, $[0,3]$, $[-\frac{3}{2}+\frac{3}{2}\,\sqrt
{2},\frac{3}{2}+\frac{3}{2}\,\sqrt {2}]$, and $[3\,\sqrt {2},0]$.

We specialize the formula which gives    $E^{L}(\alpha,1,\tau_2)(b)$
for $b=tb_I$, when $L=V$, $L$ is the vertical line, $L=\{0\}$, respectively.
\begin{enumerate}[\rm(a)]
\item First, for $L=V$, $E^{L}(\alpha,1,\tau_2)(tb_I)$ is the volume given by
$$
V(t\p_I)= \frac{9}{4} \bigl( 1+2\,\sqrt {2} \bigr) {t}^{2}.
$$
This is a polynomial function of $t$ with real coefficients.

\item
When $L$ is the vertical line, we add the lengths of the vertical segments in
$t\p_I$, for $t\geq 0$.  $S^{L}(t\p_I,1)=E_2^L(t)\, t^2 +E_1^L(t)\,
t+E_0^L(t)$ with coefficient functions
\begin{align*}
  E_2^L(t)&= \frac{9}{4} \bigl( 1+2\,\sqrt {2} \bigr),
  \qquad E_1^L(t)=\frac{3}{2},\\
  E_0^L(t)&= -1/2 \bigl\{ 3\sqrt {2}t \bigr\} ^{2}+\frac{1}{2}\bigl\{ 3\sqrt {2}t \bigr\} -\bigl\{
  -\frac{3}{2}t+ \frac{3}{2}\sqrt {2}t \bigr\} \\ &\qquad + \bigl\{ -\frac{3}{2}t+\frac{3}{2}\sqrt{2}t
  \bigr\}^{2}.
\end{align*}
This is a semi-quasi-polynomial, but not a quasi-polynomial, because
the coefficient function $E^L_0(t)$ is not periodic but merely bounded.

\item
Finally, we compute the number of integral points $S^{\{0\}}(t\p_I,1)$ in
$t\p_I$, for $t\geq 0$. We have
$
S^{\{0\}}(t\p_I,1)=E_2(t)\,t^2+E_1(t)\,t+E_0(t),
$
where the coefficient functions $E_r(t)$ are step-polynomial functions of $t$, 
\begin{align*}
  E_2(t)&= \frac{9}{4} \bigl( 1+2\,\sqrt {2} \bigr), \\
  E_1(t)&= \frac{3}{2}+\frac{3}{2}\{ 3\,t \} -\frac{3\sqrt {2}}{2}\{
    3\sqrt {2}\,t \} -\frac{3}{2}\{ 3\sqrt {2}\,t
    \} -\frac{3\sqrt {2}}{2}\{ 3t \} +3\sqrt {2} ,\\
  E_0(t)&= 1-1/2\,\{ 3\,t \} -\{ 3\,\sqrt {2}\,t\} -\frac{1}{2}\,  \{ 3\,t \}^{2}- \{ \frac{3}{2}\,t+\frac{3}{2}\,\sqrt {2}\,t \} ^{2}\\
  & \qquad +\{ \frac{3}{2}\,t+ \frac{3}{2}\,\sqrt {2}\,t \} \{ 3\,t \} +\{
  3\,\sqrt {2}\,t \} \{ \frac{3}{2}\,t+\frac{3}{2}\,\sqrt {2}\,t \}.
\end{align*}
%
%
%
%
%
%
Thus again $S^{\{0\}}(t\p_I,1)$ is a semi-quasi-polynomial, but not a quasi-polynomial.  
It takes only integral values, and  is constant over some intervals with  end points of the form $\frac{n}{3}, \frac{n}{3 \sqrt{ 2}}, \frac{n}{3 (1+\sqrt{ 2})}$ with $n$ an integer.
\end{enumerate}
The graphs of these three semi-quasi-polynomials are displayed in Figure~\ref{fig:dilatedtetra-irrational}.
\end{example}

\bigskip
Let us discuss some qualitative properties of the (semi-)quasi-polyno\-mial function
$E^L(\p,h)(t)=\sum_{r=0}^{d+m} E^L_r(\p,h)(t)\, t^r,$ defined in Theorem \ref{th:ray}.
The coefficient functions
$E^L_r(\p,h)(t)$ are given by polynomial formulae (with rational coefficients)
of  functions $\{r_j t\}$, where some $r_j$ may be irrational.  In particular,
$E^L_r(\p,h)(t)$ is a bounded 
function on $\R$. (If the polytope~$\p$ is rational, the coefficients $r_j$ are
rationals, and $E^L_r(\p,h)(t)$ is a periodic function on~$\R$ and thus
$E^L(\p,h)(t)$ is an ordinary quasi-polynomial function on~$\R$.)
The individual function $\{r_jt\}$ is right-continuous if
$r_j>0 $, and coincides with the affine linear function $r_jt- n$ on the semi-open interval $[n/r_j, (n+1)/r_j[$.
If $r_j<0$, $\{r_jt\}$ is left-continuous.
It follows that there is a sequence $0\leq d_1<d_2<\cdots<d_i<d_{i+1}<\cdots$ such that
$E^L_r(\p,h)(t)$ is given by a polynomial formula on
the open interval  $]d_i,d_{i+1}[$. At the points $d_i$, the function $E^L_r(\p,h)(t)$ may be left-continuous, right-continuous, or $d_i$ can be a point of discontinuity.
In section \ref{sub:dilatedsimplex}, examples are given of  functions $E^L(\p,h)(t)$  having all types of discontinuity on a discrete set of points (left continuity, right continuity, or discontinuous).

Note, however, that if $0\in \p$ (as in Examples \ref{ex:parametric-dim2-rational-Ehrhart} and  \ref{ex:parametric-dim2-irrational-Ehrhart}), then the function $E^L(\p,h)(t)$ is right-continuous.

\subsubsection{Case of Minkowski linear systems}\label{subsubMinkowski}

Let $\p_1,\p_2,\ldots, \p_q$ be  semi-rational polytopes in $V$.
Then it is well known that there exists a parametric polytope $\p(b)$ defined by
$\alpha_j(x)\leq b_j$ for $1\leq j\leq N$,
a chamber  $\tau\subset \R^N$ and values   $b^1,b^2,\ldots, b^q\in \overline{\tau}$ of the multi-parameter $b$,
such that, for all $t_i\geq0$, the  Minkowski sum $t_1\p_1+t_2\p_2+\cdots +t_q\p_q$ is the polytope
$\p(t_1b^1+ t_2 b^2+\cdots+t_q b^q)$, see for instance \cite{MR0326574}.

Let us recall how to determine this parametric family. For simplicity, we take the case  of  two polytopes $\p_1,\p_2$. Then
 the outer normals $\alpha_j$  and the values $b^1_j,b^2_j$ are determined in the following way.
 Without loss of generality, we can assume that $\p_1+\p_2$ is full-dimensional.
 Let $(\alpha_1,\dots , \alpha_N) $ be the list of outer normal vectors to the facets of $\p_1+\p_2$.
 Then, for each index $j$, we have  $b^1_j:=\max (\langle\alpha_j,x\rangle, x\in \p_1)$ and
 $b^2_j:=\max (\langle\alpha_j,x\rangle, x\in \p_2)$. In other words, the facet $\f(\p_1+\p_2,\alpha_j)$
  where $\langle\alpha_j,x\rangle$ reaches its maximum on $\p_1+\p_2$ is the Minkowski sum of the face $\f(\p_1,\alpha_j)$ of $\p_1$ where $\langle\alpha_j,x\rangle$ reaches its maximum on $\p_1$ and the analogous  face $\f(\p_2,\alpha_j)$ of $\p_2$.

Therefore we obtain the following corollary of Theorem \ref{th:ehrhart-chamber}:

\begin{theorem}\label{th:Ehrhart-for-Minkowski-sum}
 Let $\p_1,\p_2,\ldots, \p_q$ be (semi-)rational polytopes in $V$, $L\subseteq V$ a rational subspace  and
 $h(x)$ a weight on $V$.
 There exists a (semi-)quasi-polynomial function $E(t_1,t_2,\ldots,t_q)$ on $\R^q$
 such that for $t_1\geq 0,\ldots, t_q\geq 0$,
 the  intermediate weighted sum on the polytope $t_1\p_1+t_2\p_2+\cdots +t_q \p_q$ is given by
 $$
  S^L(t_1\p_1+t_2\p_2+\cdots +t_q \p_q, h)=E(t_1,t_2,\ldots, t_q).
  $$
\end{theorem}

An example has already appeared in the introduction as Example~\ref{ex:minkowski}
(\autoref{fig:minkowski}).

\begin{remark}
The results in \cite{HenkLinke} were obtained by writing any  ``vector dilated
polytope'' as a Minkowski linear system.\footnote{Again we note that their
  results are for the rational case only.}
In the present article,  we take the opposite route,  starting with Brion's theorem.
\end{remark}

\section{Patched sums and highest polynomial degree terms of weighted Ehrhart quasi-polynomials}\label{sect:fullBarvinok}

Let $\p(b)$ be our parametric polytope.
Our next concern is to study the highest polynomial degree terms of the weighted Ehrhart quasi-polynomial
$$E(\alpha,h,\tau)(b)=E^{\{0\}}(\alpha,h,\tau)(b)=\sum_{x\in \p(b)\cap \Lambda}h(x)$$
where $b$ varies in the chamber $\tau$.

Following Barvinok \cite{barvinok-2006-ehrhart-quasipolynomial},
we introduce some particular linear combinations of intermediate weighted sums on polytopes and the analogous
linear combinations of intermediate generating functions of a polyhedron.

\subsection{Patched generating function associated with a family of slicing subspaces}\label{sub:patch}

 Let $\CL$ be a finite family
of linear subspaces  $L\subseteq V$ which is closed under sum.
Consider the subset $\bigcup_{L\in \mathcal L}L^{\perp}$ of $V^*$.
 As the family $\{\,L^{\perp}: L\in \mathcal L\,\}$ is stable under intersection,
 there exists a unique function $\rho$ on $\mathcal L$ such that
$$
\Bigl[\bigcup_{L\in \mathcal L} L^{\perp}\Bigr]=\sum_{L\in \mathcal L}\rho(L)[L^{\perp}].
$$
We call $\rho$  the \emph{patching function} on $\mathcal L$.
Let us recall its relationship with the M\"obius function of the poset $\CL$ (\cite{barvinok-2006-ehrhart-quasipolynomial}, \cite[Lemma~4.1]{SLII2014}).
Let $ \hat{\mathcal L}$ be the poset obtained by adding a smallest element $ \hat{0}$ to $\CL$.
 Denote by $\mu$ its M\"obius function.
\begin{lemma}\label{lemma:patching-Mobius}
 The {patching function} $\rho$ on $\mathcal L$  satisfies
$$
\rho(L)=-\mu( \hat{0},L).
$$
\end{lemma}

We consider the following linear combination of intermediate generating functions.
\begin{definition}\label{def:Barvinok-sum}
The \emph{Barvinok patched generating function} of a semi-rational polyhedron
$\p\subseteq V$ (with respect to the family $\CL$) is
$$
S^{\CL}(\p)(\xi)=\sum_{L\in\CL}\rho(L) S^L(\p)(\xi).
$$
\end{definition}
Technically, we will use the shifted version,  similar to Definition \ref{def:Barvinok-sum}:
\begin{definition}\label{def:retroBarvinok-sum}
  $$
   M^\CL (s,\c )(\xi)=\e^{-\langle\xi,s\rangle}S^{\CL}(s+\c)(\xi)=\sum_{L\in\CL}\rho(L) M^L(s,\c)(\xi).
   $$
\end{definition}

We now define a subspace of the space $\CM_{\ell}(V^*)$ from Definition~\ref{def:Mell}.
\begin{definition}
  We introduce the notation $\CM_{[\geq q]}(V^*)$ for the space of functions
  $\phi$ in $\CM_{\ell}(V^*)$ such that $\phi_{[m]}(\xi)=0$ if $m<q$.
\end{definition}

We now state the approximation  theorem of \cite{SLII2014} for generating
functions of cones. 

\begin{theorem}[Theorem~\ref{sl2:th:better-than-Barvinok} in \cite{SLII2014}]\label{th:better-than-Barvinok}
Let $\c$ be a rational cone. Fix $k$, $0\leq k\leq d$. Let
${\CL_k}$ be a family of subspaces of $V$, closed under sum, such that
$\lin(\f)\in{\CL_k}$ for every face $\f$ of codimension $\leq k$ of $\c$.
Let $\rho(L),L\in {\CL_k}$, be the patching coefficients of  ${\CL_k}$, let $ S^{\CL_k} (s+\c )(\xi)$
be the Barvinok patched generating function of Definition \ref{def:Barvinok-sum}
and let $M^{\CL_k} (s,\c )(\xi)$ be as in Definition \ref{def:retroBarvinok-sum}. Then, for any $s\in V$,
\begin{equation}\label{eq:approx-M}
M(s,\c )(\xi)-M^{\CL_k} (s,\c )(\xi)\in  \CM_{[\geq -d+k+1]}(V^*),
\end{equation}
\begin{equation}\label{eq:better-than-Barvinok-cone}
S(s+\c )(\xi)-S^{\CL_k} (s+\c )(\xi)\in  \CM_{[\geq -d+k+1]}(V^*).
\end{equation}
\end{theorem}
We show in the next section that these approximations of  generating functions of cones
lead to  computations of the highest polynomial degree terms of  weighted Ehrhart quasi-polynomials of parametric polytopes.

\subsection{Highest polynomial degree terms}
Recall the notations of Theorem \ref{th:ehrhart-chamber}.
Our key technical result is the following.

\begin{theorem}\label{prop:computation-Ehrhart-terms}
 Let $\p(b)\subset V$ be a parametric
polytope.
Fix $k$, $0\leq k \leq d$. Let $h(x)=\frac{\langle \ell,x\rangle^m}{m!}$. Let $\tau$ be an admissible chamber for $\p(b)$.

For each $B\in \CB_\tau$, let $\CL_{k,B}$ be a family of subspaces which contains the faces of $\c_B$ of codimension $\leq k$
and is closed under sum.
Then for $r\geq d+m-k$, the terms of polynomial degree $r$ of the Ehrhart quasi-polynomial $E(\alpha,h,\tau)(b)$ are given by
\begin{equation}\label{eq:computation-Ehrhart-terms}
E_{[r]}(\alpha,h,\tau)(b)=\left(\sum_{B\in \CB_\tau} \retroS^{\CL_{k,B}}(s_B(b),\c_B)_{[m-r]}(\xi) \frac{\langle\xi,s_B(b)\rangle^r}{r!}
\right)\bigg|_{\xi=\ell}.
\end{equation}
\end{theorem}

\begin{proof}
By Formula \eqref{eq:top-parametric-Ehrhart}, the terms of polynomial degree~$r$ of the
Ehrhart quasi-polynomial $E(\alpha,h,\tau)(b)$ are given by 
  $$
   E_{[r]}(\alpha,h,\tau)(b)=
 \left(\sum_{B\in \CB_\tau}
 \retroS(s_B(b),\c_B)_{[m-r]}(\xi) \frac{\langle\xi,s_B(b)\rangle^r}{r!}\right)\bigg|_{\xi=\ell}.
  $$
  By Theorem \ref{th:better-than-Barvinok}, for $r\geq d+m-k$, i.e., $m-r\geq -d+k$,
  we have the equality of the homogeneous components of $\xi$-degree $m-r$,
  $\retroS(s_B(b),\c_B)_{[m-r]}(\xi)= \retroS^{\CL_{k,B}}(s_B(b),\c_B)_{[m-r]}(\xi)$.
\end{proof}

We have an analogous result for the dilation of a single  polytope.
\begin{proposition}\label{prop:computation-dilated-Ehrhart-terms}
Let $\p\subset V$ be a (semi-)rational polytope.  Fix $k$, $0\leq k \leq d$. Let $h(x)=\frac{\langle \ell,x\rangle^m}{m!}$. For each vertex $s$ of $\p$, let $\c_s$ be the cone of feasible directions of $\p$ at $s$.
For each $s$, let $\CL_{k,s}$ be a family of subspaces which contains the faces of $\c_s$ of codimension $\leq k$
and is closed under sum.
Then for $r\geq d+m-k$, the term of polynomial degree $r$ of the Ehrhart (semi-)quasi-polynomial $E(\p,h)(t)$ is given by
\begin{equation}\label{eq:computation-dilated-Ehrhart-terms}
E_{[r]}(\p,h)(t)=\left(\sum_{s}  \retroS^{\CL_{k,s}}(t s,\c_s)_{[m-r]}(\xi) \frac{\langle\xi,s\rangle^r}{r!}\right)\bigg|_{\xi=\ell} \;t^r.
\end{equation}
\end{proposition}

In the case of a single polytope $\p$, there are two canonical families $\CL_{k,s}$ associated with a vertex $s$ of $\p$.
The first one, $\CL_{k,s}^{\ConeByCone}$,
is the smallest family which contains the subspaces parallel to the faces of~$\p$ through~$s$,
and which is closed under sum (this condition is automatic if $\p$ is simple).
 This family depends only on the cone $\c_s$.
 The second family, $\CL_k^{\Barvi}$,  is the smallest family which contains the subspaces parallel to all the faces of $\p$ and which is closed under sum. The second family is the one which was originally used by Barvinok in \cite{barvinok-2006-ehrhart-quasipolynomial}. It depends on the polytope $\p$ and is the same family at each vertex $s$ of $\p$.
We will return to these two choices in section~\ref{sect:three}
 and associate to each of these choices a canonical quasi-polynomial function of the multi-parameter $b$
 obtaining our three quasi-polynomials, as promised  by the title of our article.

We first give more details on patching functions in two simple cases.

\subsection{The patching function of a simplicial cone}

Let $\c$ be a simplicial cone with edge generators $v_1,\dots, v_d$.
Let us denote here by $\CL_k(\c)$  the family of subspaces $\lin(\f)$ for faces
of codimension $\leq  k$. This family is  closed under sum.
We computed the patching function of $\CL_k(\c)$ in  \cite{so-called-paper-1}. Let us recall the result.
\begin{definition}
 Let $\CJ^d_{\geq (d-k)}$ be the set  of subsets  $I\subseteq
\{1,\dots,d\}$ with $|I|\geq d-k$.
\end{definition}

Recall that we denote  by $L_I$ the  linear space spanned by  the vectors $v_i$, $i\in I$.
Let $d=\dim V$. Then $I\mapsto L_I$ is an isomorphism of posets between $\CJ^d_{\geq (d-k)}$ and $\CL_k(\c)$.
 We denote by $\rho_{d,k}$ the patching function on $\CL_k(\c)$.\footnote{In
   \cite{so-called-paper-1}, it appears under the name
   $\lambda_{\Moebius}(I)$.}
\begin{proposition}[{\cite[Proposition 29]{so-called-paper-1}}]\label{prop:cone-patch}
\begin{equation}\label{eq:cone-patch}
\rho_{d,k}(L_I)= (-1)^{|I|-d+k}\binomial(|I|-1,d-k-1)
\end{equation}
where $\binomial(a,b)=\frac{a!}{b! (a-b)!}$ is the binomial coefficient.
\end{proposition}

\subsection{The  patching function of  a simplex}
\label{s:Lk-Barvinok}

Let $\p$ be a $d $-dimensional simplex with vertices $s_1,\ldots,
s_{d+1}$. We fix $0\leq k \leq d$. The faces of $\p$ of dimension $\geq d-k$ are labeled by
subsets $I$ of $\{1,\dots,d+1\}$ of cardinality $|I|\geq d-k+1$, the face $\f_I$ being the affine
span of the vertices $s_i$, for $i\in I$.
Let us denote the
corresponding linear subspace $\lin(\f_I)$ by $L_I$.
Thus $L_I$ is the linear span of vectors $s_i-s_j$ with $i,j\in I$.
We consider only the case $k\leq d-1$.
The family $L_I$ is not closed under sum, in general. If $I_1\cap I_2 \neq \emptyset$,  we clearly have
$$
L_{I_1}+ L_{I_2} = L_{I_1\cup I_2}.
$$
On the contrary, if  $I_1\cap I_2= \emptyset$, then the sum $L_{I_1} \oplus
 L_{I_2}$ is direct and is not of the form $\lin(\f)$. Therefore,
 to describe the family $\CL_k^{\Barvi}$ for the simplex, we
 need to
 consider ``subpartitions'' $I=\{I_1,\ldots, I_m\}$ of $\{1,\dots,d+1\}$,
 meaning that   $I_1,\ldots, I_m$ are pairwise disjoint subsets of
$\{1,\dots,d+1\}$.
 The
 corresponding subspace is
\begin{eqnarray}\label{eq:labels}
    L_I =  L_{I_1} \oplus \cdots \oplus L_{I_m}\subseteq V.
\end{eqnarray}
\begin{definition}
  We denote by $\BjoernerLovasz(N,n)$ the poset of all the subpartitions
 $I=\{I_1,\ldots, I_m\}$  of $\{1,\dots,N\}$, with $ m\geq 1$ and  $|I_j|\geq
 n$, ordered by refinement and set inclusion.\footnote{The poset can also be
   identified with the subposet of those partitions of 
   $\{1,\dots,N\}$ that have no block sizes in $2, 3, \dots, n-1$.  (The
   singleton blocks represent the elements of $\{1,\dots,N\}$ not in any of
   the $I_j$.) In this form, the poset appears under the notation $\Pi_{N,n}$ in a paper by A.~Bj\"orner
   and L.~Lov\'asz \cite{MR1243770}; we adopt the same notation.} 
 (Its least element is the empty subpartition, $\hat0 = \emptyset$.)
\end{definition}

Let  $N=d+1$, $n=d-k+1$. Then the
 map $I\mapsto L(I)$  is a poset isomorphism
of $\BjoernerLovasz(d+1,d-k+1)$ with  the  poset
 $\CL_k^{\Barvi}$  associated with the faces of codimension $\leq k$ of the $d$-dimensional simplex $\p$.

 We denote by $\sigma_{d,k}(I)$ the patching function on $\CL_k^{\Barvi}$.
 By Lemma \ref{lemma:patching-Mobius}, it is the opposite of  the M\"obius function $\mu(\hat{0},I)$ of the
 poset $\BjoernerLovasz(d+1,d-k+1)$.
\begin{example}
The  patching function on $\CL_2^{\Barvi}\simeq \BjoernerLovasz(4,2)$,
associated with the $3$-dimensional simplex (tetrahedron) and its faces of codimension $\leq 2$.
The six  edges correspond to  $ \{\{1,2\}\},\{\{1,3\}\},\ldots $, the
four facets to $ \{\{1,2,3\}\},\{\{1,2,4\}\},\ldots $, the simplex itself to
$\{\{1,2,3,4\}\}$, and the three planes spanned by the directions of two
opposite edges to
$\{\{1,2\},\{3,4\}\}$, $\{\{1,3\},\{2,4\}\}$, $\{\{2,3\},\{1,4\}\}$.
The  M\"obius function of $\BjoernerLovasz(4,2)$ is easy to compute directly. We obtain
\begin{align*}
\sigma_{3,2}(\{\{i,j\}\}) &=  1  \\
\sigma_{3,2}(\{\{i,j,k\}\}) &=  -2 \\
\sigma_{3,2}\{\{1,2,3,4\}\}) &=  6\\
\sigma_{3,2}(\{\{i,j\},\{k,l\}\}) &= -1
\end{align*}
\end{example}
It turns out that the M\"obius function $\mu(\hat{0},I)$  of the poset $\BjoernerLovasz(N,n)$
has been computed by A.~Bj\"orner and L.~Lov\'asz in a different context
\cite[section 4]{MR1243770}.
Their result is the following.
\begin{proposition}\label{prop:mobius}
Denote by $\mu_{N,n}(I)$ the M\"obius function $\mu(\hat{0},I)$ of the poset $\BjoernerLovasz(N,n)$. Then,

\begin{enumerate}[\rm(i)]
\item $ \mu_{N,n}(I_1,\ldots, I_r)$ depends only on the
  block-sizes $n_1=|I_1|,\ldots, n_r=|I_r| $.

  Let us write $\mu_{N,n}(n_1,\ldots, n_r)$ for $ \mu_{N,n}(I_1,\ldots, I_r)$.

\item $\mu_{N,m}(n_1,\ldots, n_r)= \mu_{N,m}(n_1) \dots
  \mu_{N,m}(n_r)$.

\item $\mu_{N,n}(m)= \mu_{m,n}(m)$, for $n\leq m\leq N$.

\item Let $\mu_N(n)= \mu_{N,n}(N)$ if $N\geq n$, $\mu_1(n)=1$ for
  every $n\geq 1$ and $\mu_N(n)=0$ for $2\leq N\leq n-1$.  Consider the
  generating series
$$
F_n(z)=\sum_{N=1}^\infty \mu_N(n)\frac{z^N}{N!}.
$$
Then
$$
\e^{ F_n(z)}= \sum_{N=0}^{n-1}\frac{z^N}{N!}.
$$
\end{enumerate}
\end{proposition}
Applying Lemma \ref{lemma:patching-Mobius},  we deduce  the following computation rules for
  the patching function $\sigma_{d,k}(I)$.
\begin{proposition}\label{prop:patching-simplex}~
  \begin{enumerate}[\rm(i)]
  \item $ \sigma_{d,k}(I_1,\ldots, I_r)$ depends only on the block-sizes
    $n_1=|I_1|,\ldots, n_r=|I_r| $.

    Let us write $\sigma_{d,k}(n_1,\ldots, n_r)$ for $
    \sigma_{d,k}(I_1,\ldots, I_r)$.

  \item $\sigma_{d,k}(n_1,\ldots, n_r)= (-1)^{r-1} \sigma_{d,k}(n_1)
    \dots \sigma_{d,k}(n_r)$.

  \item  For $d-k\leq m\leq d+1$, $\frac{\sigma_{d,k}(m)}{m!}$ is
    the coefficient of degree $m$ of the power series
$$
-\ln \sum_{p=0}^{d-k-1}\frac{z^p}{p!}.
$$
\end{enumerate}
\end{proposition}

\section{The cone-by-cone patched generating function}\label{sect:intermediate-todd}

Let $\p\subset V$ be a  polytope. Fix $k$, $0\leq k\leq d$, where $d$ again is
the dimension of~$\p$. 

If ${\CL_k}$ is a family of subspaces which is closed under sum,  then it follows from Brion's theorem that
$$
S^{\CL_k}(\p)(\xi)= \sum_{s\in \CV(\p)} S^{{\CL_k}}(s+\c_s )(\xi).
$$
In particular, the sum $\sum_{s\in \CV(\p)} S^{{\CL_k}}(s+\c_s )(\xi)$ is analytic, although each $s$-term is singular at $\xi=0$. The singularities cancel out when we sum over the vertices.

In contrast, if  we take a different family   $\CL_{k, s}$  for each vertex $s$, the
sum over vertices $\sum_{s\in \CV(\p)} S^{\CL_{k, s}}(s+\c_s )(\xi)$ need not be
analytic. 

However, when $\p$ is a simple polytope (i.e., its cones at vertices $\c_s$ are
simplicial), and $\CL_{k, s}$ is chosen as the  family of subspaces parallel to the faces
of $\c_s$ which are of codimension $\leq k$, we will show that this sum
is actually holomorphic.
We will deduce this fact from the computation in \cite{SLII2014} of   the residues of the generating function of a shifted cone.

Let us give a name and a notation for this sum.
\begin{definition}\label{def:cone-by-cone-generating-function} Let $\p$ be a simple polytope.
The \emph{cone-by-cone patched generating function} of $\p$ is
$$
 A^{k}(\p)(\xi)= \sum_{s\in \CV(\p)} S^{\CL_{k,s}}(s+\c_s )(\xi)
 $$
 where, for each vertex $s$ of $\p$,  $\CL_{k,s}$ is  the  family of subspaces
 parallel to the faces of $\c_s$ which are of codimension $\leq
 k$.\footnote{This function has appeared in \cite{so-called-paper-1}, using
   the notation $A_{\geq
     d-k}(\p)(\xi)$, where $d$ is the dimension of~$\p$.} 
\end{definition}
\tgreen{Changed it from $A_k$ to $A^{k}$ to match $S^{k,....}$ and
  $E^{k,....}$ notation.  (The notation is different from the one in
  \cite{so-called-paper-1} anyway.)  The only place where $k$ remains a subscript is in
  $\CL_k$ -- which I don't want to change because $S^{\CL^k}$ would look bad.}

\begin{proposition}\label{prop:cone-by-cone-is-analytic}
Let $\p\subset V$ be a  simple polytope.
Then  $A^{k}(\p)(\xi)$  is analytic near $\xi=0$.
\end{proposition}
\begin{proof}
  Let $(v_a)$ be a set of  pairwise distinct generators of all the edges of $\p$.
  By Proposition \ref{sl2:prop:polesSL} in \cite{SLII2014},  $S^{L}(s+\c_s )$ has simple hyperplane poles near $\xi=0$, given by  the edges $v_a$  of the cone at vertex $s$.
It follows that the product
  $\prod_a \langle \xi,v_a\rangle A^{k}(\p)(\xi)$ is analytic.
  Therefore it is enough to show that for each edge $v_a$, the corresponding residue vanishes.
 In other words, we want to show that
 \begin{equation}\label{eq:zero-residue-1}
  \bigl(\langle \xi,v_a\rangle A^{k}(\p)(\xi) \bigr)\big|_{v_a^\perp}=0
  \end{equation}

 Thus,  we need only to compute the residues for those vertices $s$ where $v=v_a$ is an edge of $c_s$.
 Let $p$ be the projection $V\to V/\R  v$.
Let $s$ and $s'$ be adjacent vertices of $\p$ such that $s'-s\in \R v$.
If we take $v$ to be the edge generator for the cone $\c_s$,
then the edge generator for the other cone
$\c_{s'}$ is $-v$.
Moreover, the  projected cones $p(s+\c_s)$ and
$p(s'+\c_{s'})$ are both equal to the tangent cone at the vertex $p(s)$ of the projected polytope.

Recall that we denote by  $\Jposet{d}{d_0}$ the set of subsets $I\subseteq
\{1,\dots,d\}$ of cardinality $|I|\geq d_0$.
For a given vertex $s$ of $\p$, let $v_1,\dots, v_d$  be primitive  edge generators at the vertex $s$.
For $I\in  \Jposet{d}{d_0}$,  the subspace generated by $v_i$ for $ i\in I$ is denoted by $L_I^s$.
Then the Barvinok patched generating function  for the cone $s+\c_s$, with respect to the codimension
$k=d-d_0$,   is
$$
S^{\CL_{k,s}}(s+\c_s )(\xi)= \sum_{p\geq d_0} (-1)^{p-d_0}\binomial(p-1,d_0-1)\sum_{I,|I|=p}S^{L_I^s}(s+\c_s)(\xi).
$$
We can label the edges at $s$ and $s'$ in such a way that the projections $p(L_I^s)$ and $p(L_I^{s'})$ coincide for all $I$.
Then by Proposition \ref{sl2:prop:polesSL} in \cite{SLII2014},  we conclude that the
residues of $S^{\CL_{k,s}}(s+\c_s )(\xi)$ and $S^{\CL_{k,s'}}(s'+\c_{s'} )(\xi)$
at $v^\perp$ cancel out.
\end{proof}

\begin{remark}[intermediate Todd classes]\label{remark:intermediate-todd}
Let us  outline an explanation
of the analyticity of $A^{k}(\p)(\xi)$ based on intermediate Todd classes.
We will not use this remark in the rest of this article.

When $\p$ is a  simple lattice polytope,
we can relate the function  $A^{k}(\p)(\xi)$ to some equivariant cohomology classes of the associated toric variety.
For simplicity, let us consider a Delzant polytope $\p$.
Let $G$ be the torus with Lie algebra
$\mathfrak g=V^*$ and weight lattice $\lattice\subset \mathfrak g^*=V$.

A smooth toric variety  $M$ and $G$-equivariant line bundle $\CL$ on $M$ are associated with $\p$.
Let $c_\CL(\xi)$ be the $G$-equivariant Chern character of $\CL$
and let ${\rm Todd}^M(\xi)$ be the $G$-equivariant Todd class of $M$.
The  Riemann--Roch theorem relates the set of lattice points of $\p$
with the integral over $M$ of the product of these two classes:

\begin{equation}\label{eq:RR}
S(\p)(\xi)=\int_M \e^{c_\CL(\xi)}{\rm Todd}^M(\xi).
\end{equation}

Brion's formula can be understood  as  the localization formula applied to this integral.
For a generic element $\xi \in {\mathfrak g}$, the vertices of $\p$ correspond to the fixed points in $M$ under the one-parameter group
$\exp(t\xi)$ and
the contribution of the fixed point corresponding to the vertex $s$
is precisely $S(s+\c_s)(\xi)$.

Now, let $c_1,\dots, c_d$ be the equivariant Chern classes of the tangent bundle of $M$, so that
${\rm Todd}^M=\prod_{j=1}^d \frac{c_j}{1-\e^{-c_j}}$.
Fix $0\leq k\leq d$ and let $d_0=d-k$.

Let us introduce the equivariant \emph{intermediate Todd class}
\begin{equation}\label{eq:intermediate-todd-class}
{\rm Todd}_{k}^M=\sum_{p\geq d_0}(-1)^{p-d_0}\binomial(p-1,d_0-1)
\sum_{I,|I|=p}\prod_{j\in I}\frac{c_j}{1-\e^{-c_j}}.
\end{equation}
We recover the cone-by-cone patched generating function when we replace ${\rm Todd}$
with ${\rm Todd}_{k}$ in \eqref{eq:RR} and apply the localization formula:
\begin{multline}\label{eq:intermediate-RR}
\int_M \e^{c_\CL(\xi)}{\rm Todd}_{k}^M(\xi)=\\
\sum_{s\in \CV(\p)}
\sum_{p\geq d_0}(-1)^{p-d_0}\binomial(p-1,d_0-1)\sum_{I,|I|=p}S^{L_I}(s+\c_s)(\xi).
\end{multline}
Now the integral on the left-hand side of \eqref{eq:intermediate-RR} depends on $\xi$ analytically.

However,  we do not have any geometric interpretation
of the quantity computed by the left hand  side of \eqref{eq:intermediate-RR}.
\end{remark}

\section{Three Ehrhart quasi-polynomials}\label{sect:three}
\subsection{Case of a parametric polytope}
 Let $\p(b)\subset V $ be a parametric polytope, defined by  $\alpha\colon V\to \R^N$,  and let $\tau\subset \R^N$ be an admissible  $\alpha$-chamber. Let $h(x)$ be a weight of degree $m$ on $V$.  We fix a codimension~$k$. Associated with these data, we have three \textbf{canonical} weighted Ehrhart quasi-polynomials on $\R^N$  which have the same terms of  polynomial degree $\geq d+m-k$.

 The first
 one, $E^{\{0\}}(\alpha,h,\tau)(b)=E(\alpha,h,\tau)(b)$, does not depend on $k$. It is equal, for $b\in \overline{\tau}$,  to  the ordinary weighted sum $S(\p(b),h)$. This is the quantity that we want to
 study.

 The second one is associated with the full family of slicing subspaces $\CL_k^{\Barvi}$.
 \begin{definition}
 $\Barvinok(\alpha,h,\tau)(b)$ is the quasi-polynomial such that
 \begin{equation} \label{eq:fullBarvinok}
\Barvinok(\alpha,h,\tau)(b)=S^{\CL_k^{\Barvi}}(\p(b),h)  \mbox{  for } b\in \overline{\tau}.
\end{equation}
\end{definition}
\tgreen{Index $k$ now is a superscript in notation $\Barvinok$, to match
  notation from introduction, and to remove notation clash with coefficient
  functions $E_k$.}

Our third quasi-polynomial owes its existence to the analyticity of the
cone-by-cone patched generating function  $A^{k}(\p(b))(\xi)$ of Definition
\ref{def:cone-by-cone-generating-function} for simple polytopes, 
established in Proposition~\ref{prop:cone-by-cone-is-analytic}.  Recall that
$\p(b)$ is simple when $b$ lies in an open chamber.%

\begin{proposition}\label{prop:third-parametric-case}
  Let $\ell\in V^*$, $m\geq0$, and $h(x) =
  \frac{\langle\ell,x\rangle^m}{m!}$. Fix a codimension~$k$ and a
  chamber~$\tau$. 
  Then there exists a unique quasi-polynomial $E^{k,\ConeByCone}(\alpha,h,\tau)(b)$ on
  $\R^N$ such that
  \begin{equation}\label{eq:third-quasi-polynomial}
    E^{k,\ConeByCone}\bigl(\alpha,h,\tau\bigr)(b)= A^{k}(\p(b))_{[m]}(\ell) \quad\mbox{for } b\in {\overline{\tau}},
  \end{equation}
  where $A^{k}(\p)_{[m]}(\xi)$ denotes the homogeneous component of degree $m$
  with respect to~$\xi$. 
\end{proposition}
We extend the definition of $E^{k,\ConeByCone}\bigl(\alpha,h,\tau\bigr)(b)$ to
arbitrary polynomial functions $h(x)$ on~$V$ by 
decomposition as a sum of powers of linear forms.
\tgreen{Replaced use of differential operators by sums of power of linear forms.}
The piecewise quasi-polynomial $S^{k,\ConeByCone}(\p(b),h)$ from the
introduction is given on the closure~$\overline{\tau}$ of each chamber by
$E^{k,\ConeByCone}(\alpha,h,\tau)(b)$. \tgreen{This makes the connection to the
  intro notation missing so far.}

\begin{proof}
We need only prove that  $b\mapsto  A^{k}(\p(b))_{[m]}(\xi) $ is given by a
quasi-polynomial function of $b$ (with values in the space of polynomials
in~$\xi$), when $b$ varies in $\tau$. The proof is similar to that of
Theorem~\ref{th:ehrhart-chamber} and again 
relies on \cite[{Theorem~\ref{sl2:prop:homogeneous-ML}}]{SLII2014}.
\end{proof}

The next theorem  follows immediately from Proposition \ref{prop:computation-dilated-Ehrhart-terms}.
\begin{theorem}\label{th:cone-by-cone-parametric-Ehrhart}
Let $\p(b)\subset V $ be a parametric polytope, defined by  $\alpha\colon V\to \R^N$,  and let $\tau\subset \R^N$ be an $\alpha$-chamber. Fix a codimension~$k$.  Let $h(x)$ be a polynomial function on $V$ of degree $m$.
Then the three quasi-polynomials  $E(\alpha,h,\tau)(b)$, $ \Barvinok(\alpha,h,\tau)(b)$ and $ E^{k,\ConeByCone}(\alpha,h,\tau)(b)$
 have the same  terms of polynomial degree $\geq d+m-k$.
\end{theorem}


\subsection{Case of a dilated polytope. Polynomial time algorithms. Examples.}\label{sub:dilatedsimplex}
In this section, we consider just one polytope $\p\subset V$ and the dilated
polytope $t\p$ for $t$ real $>0$. We fix a weight $h(x)$ on $V$ and a
codimension $k$.  If $\p$ is simple, then we have again three canonical
quasi-polynomials of the parameter $t$ associated with $\p$ and $h(x)$, all three of polynomial degree $d+m$.  If $\p$ is not simple,
only the first two quasi-polynomials are defined.\footnote{
  Of course, we could define $E^{k, \ConeByCone}(\p, h)$ as well for
  non-simple polytopes by taking limits.  However, it would no longer be canonical, 
  as it would depend on the path.  It is an open question whether a canonical
  definition is possible, which also should have a toric interpretation
  (Remark~\ref{remark:intermediate-todd}).  \tgreen{Moved this footnote to the
  right place -- was wrong before.}} 
The first quasi-polynomial is defined by
$$
E(\p,h)(t) = E^{\{0\}}(\p,h)(t)=S(t\p,h) \quad\mbox{  for } t> 0.
$$
The second quasi-polynomial is defined by
$$
\Barvinok(\p,h)(t)=S^{\CL_k^{\Barvi}}(t\p,h)  \quad\mbox{  for } t> 0.
$$
If $\p$ is simple, the third quasi-polynomial is defined as follows.  If $h(x) =
\frac{\langle\ell,x\rangle^m}{m!}$, then 
$$
E^{k,\ConeByCone}(\p,h)(t)= 
A^{k}(t\p)_{[m]}(\ell)
\quad\mbox{  for } t>0.
$$
This definition is again extended to
arbitrary polynomial functions $h(x)$ on~$V$ by 
decomposition as a sum of powers of linear forms.
\tgreen{Gave a definition by sums of power of linear forms.}

Those three quasi-polynomials have the same terms of polynomial degree $\geq d+m-k$.

Furthermore, if   $\p$ is a rational simple polytope, and if $h$ is a power of
a rational linear form,  $\Barvinok(\p,h)(t)$ and $E^{k,\ConeByCone}(\p,h)(t)$
can be computed in polynomial time when the codimension $k$ is fixed.  
(We suppress a detailed statement of the algorithm and its complexity.)
Similar results on polynomial complexity hold for a more general weight $ h(x)$.  One can assume that
the weight is given as a polynomial in a fixed number $R$ of linear
forms,
$h(x)=f(\langle\ell_1,x\rangle,\dots,\langle\ell_R,x\rangle)$, or
has a fixed degree $D$.

For $E^{k,\ConeByCone}(\p,\langle\ell,x\rangle^M)(t)$, this result is obtained in  \cite{so-called-paper-1}.
In this article,
we used the step-functions $n\mapsto \{\zeta n\}_q$ which are defined as
$\zeta n \bmod q$
for $\zeta,q,n \in \Z$. However, as noted in \cite{so-called-paper-2}, the
same proof\/\footnote{Of course, we can no longer reduce $\zeta$ modulo $q$ in
  the case of real parameters.}  gives the result for real
parameters using $t\mapsto q \{\frac{\zeta}{q} t\}$.

For the case of $\Barvinok(\p,\langle\ell,x\rangle^M)(t)$, we apply  directly Theorem  28 in  \cite{so-called-paper-2} where we considered just one intermediate sum $S^L(t\p,h)$. The crucial point is that the codimension of $L$ is bounded by $k$, which is fixed.
Here, we need also to compute the patching function.
This can be done by computing recursively  the M\"obius function of the poset $\CL_k^{\Barvi}$.

The algorithm for computing
$E^{k,\ConeByCone}(\p,\langle\ell,x\rangle^M)(t)$ is simpler than the original
algorithm given by Barvinok in \cite{barvinok-2006-ehrhart-quasipolynomial}, where the  subspaces~$L$ in the Barvinok family
do not necessarily correspond to faces of $\p$.
\bigbreak

We have implemented these algorithms in the case where $\p$ is a simplex,  in
Maple.  The Maple programs are distributed as part of \emph{LattE integrale},
version 1.7.2 \cite{latteintegrale}, and also separately via the LattE
website.\footnote{The most current versions are available at
  \url{https://www.math.ucdavis.edu/~latte/software/packages/maple/}.} 
We give below some examples computed with our Maple programs.

In the case of  a lattice simplex,   when restricted to  $t\in \N$,  the three quasi-polynomials are usual polynomials in $t$.
Here is an example.
\begin{example}\label{ex:4-simplex}
  Let $\p$ be the $4$-dimensional simplex with vertices
  $$[4,6,4,3],[5,7,9,1],[5,7,3,7],[6,8,3,9],[2,1,8,0].$$
  We use the weight function $h=1$.  Table~\ref{tab:4-simplex} shows the quasi-polynomials
  $E^{k,\Barvi}\allowbreak(\p,1)(t)$ and $E^{k,\ConeByCone}(\p,1)(t)$, for maximal
  codimension $k$, $0\leq k\leq 4$.  For $k=0$, both give the volume of the
  dilated simplex, for $k=4$, both give the exact number of points.
  \begin{table}[t]
  \centering
  \caption{The two polynomials $E^{k,\Barvi}(\p,1)(t)$ and
    $E^{k,\ConeByCone}(\p,1)(t)$ for integer dilations of the lattice simplex
    of Example~\ref{ex:4-simplex}}
  \label{tab:4-simplex}
  $\def\arraystretch{1.3}
  \begin{array}{ccc}
    \toprule
    k & E^{k,\ConeByCone}(\p,1)(t) & E^{k,\Barvi}(\p,1)(t) \\
    \midrule
    0
    & \frac{3}{4} t^4
    & \frac{3}{4} t^4
    \\
    1
    & \frac{3}{4}{t}^{4}+2\,{t}^{3}+{\frac {7}{24}}\,{t}^{2}-{\frac {5}{5184}}
    & \frac{3}{4} t^4+2 t^3+\frac{7}{24} t^2
    \\
    2
    & \frac{3}{4}{t}^{4}+2\,{t}^{3}+{\frac {15}{4}}\,{t}^{2}+{\frac {15}{8}}\,t+{\frac {67}{432}}
    & \frac{3}{4} t^4+2 t^3+\frac {15}{4} t^2+\frac {15}{8} t
    \\
    3
    & \frac{3}{4}{t}^{4}+2\,{t}^{3}+{\frac {15}{4}}\,{t}^{2}+\frac{7}{2}\,t+{\frac {389}{ 432}}
    & \frac{3}{4} t^4+2 t^3+\frac {15}{4} t^2+\frac{7}{2} t
    \\
    4
    & \frac{3}{4}{t}^{4}+2\,{t}^{3}+{\frac {15}{4}}\,{t}^{2}+\frac{7}{2}\,t+1
    & \frac{3}{4}{t}^{4}+2\,{t}^{3}+{\frac {15}{4}}\,{t}^{2}+\frac{7}{2}\,t+1
    \\
    \bottomrule
  \end{array}$
\end{table}
\end{example}

Next, an example of a rational triangle dilated by a real parameter~$t$.
\begin{example}\label{ex:3poly-dim3}
  \let\frac=\tfrac
Let $\p$ be the triangle with vertices $ [1,1], [1, 2], [2, 2]$.
We use the weight function $h(x)=1$, so we approximate the number of lattice
points in the triangle dilated by a real number~$t$.  We list  below  the
cone-by-cone and the full-Barvinok  quasi-polynomials,
for   codimension $k$, $0\leq k\leq 2$.
For $k=0$, both $E^{k,\ConeByCone}(\p,1)(t)$ and $\Barvinok(\p,1)(t)$
give  the area of the dilated triangle, which is $\frac{t^2}{2}$.

\noindent
For $k=1$,
\begin{align*}
  E^{k,\ConeByCone}(\p,1)(t)
  &= \frac{t^2}{2} + \left( \frac{3}{2}- \{-t\}-\{2t\}\right)t\\
  &\qquad +\frac14-\frac{\{-t\}}{2}-\frac{\{2t\}}{2}+\frac{\{-t\}^2}{2}+\frac{\{2t\}^2}{2}, \\
  E^{k,\Barvi}(\p,1)(t)
  &= \frac{t^2}{2} + \left( \frac{3}{2} -\{-t\}
    -\{2t\} \right)t -\frac{\{t\}^2}{2}+\frac{\{t\}}{2}.
\end{align*}
For $k=2$, both give the exact number of points,
\begin{multline}\label{eq:exemple-triangle}
 \frac{t^2}{2}+ \bigl( \frac{3}{2}-\{-t\}  -\{2t\}\bigr)t \\
 +\frac{1}{2}\{2t\}^2+\frac{1}{2}\{-t\}^2+\{2t\}\{-t\}
-\frac{3}{2}\{-t\}
-\frac{3}{2}\{2 t\}+1.
\end{multline}
For instance, for $t=\frac12$, $t=\frac16\pi=0.52359877
\dots$, or $t=\frac12 {\sqrt[3]{17/10}}=0.59674159
\dots$, the last expression gives $ 1$,
which is indeed the number of lattice points in the triangle with vertices  $
[\frac12,\frac12]$, $[\frac12, 1]$, and $[1, 1]$.

On the other hand, for $t=1$, Formula~\eqref{eq:exemple-triangle} gives~$3$, which is
indeed the number of lattice points in the triangle with vertices
$ [1,1]$, $[1, 2]$, and $[2, 2]$, while for $t=1\pm \epsilon$,  with any  small
$\epsilon$, Formula~\eqref{eq:exemple-triangle} gives~$1$.  We leave it as an
exercise to the reader to understand the mystery; Figure~\ref{fig:epsilon}
may help. Figure \ref{fig:exemple-triangle}  displays the graphs
of the above quasi-polynomials.
\end{example}
\begin{figure}
\begin{center}
  \includegraphics[width=6cm]{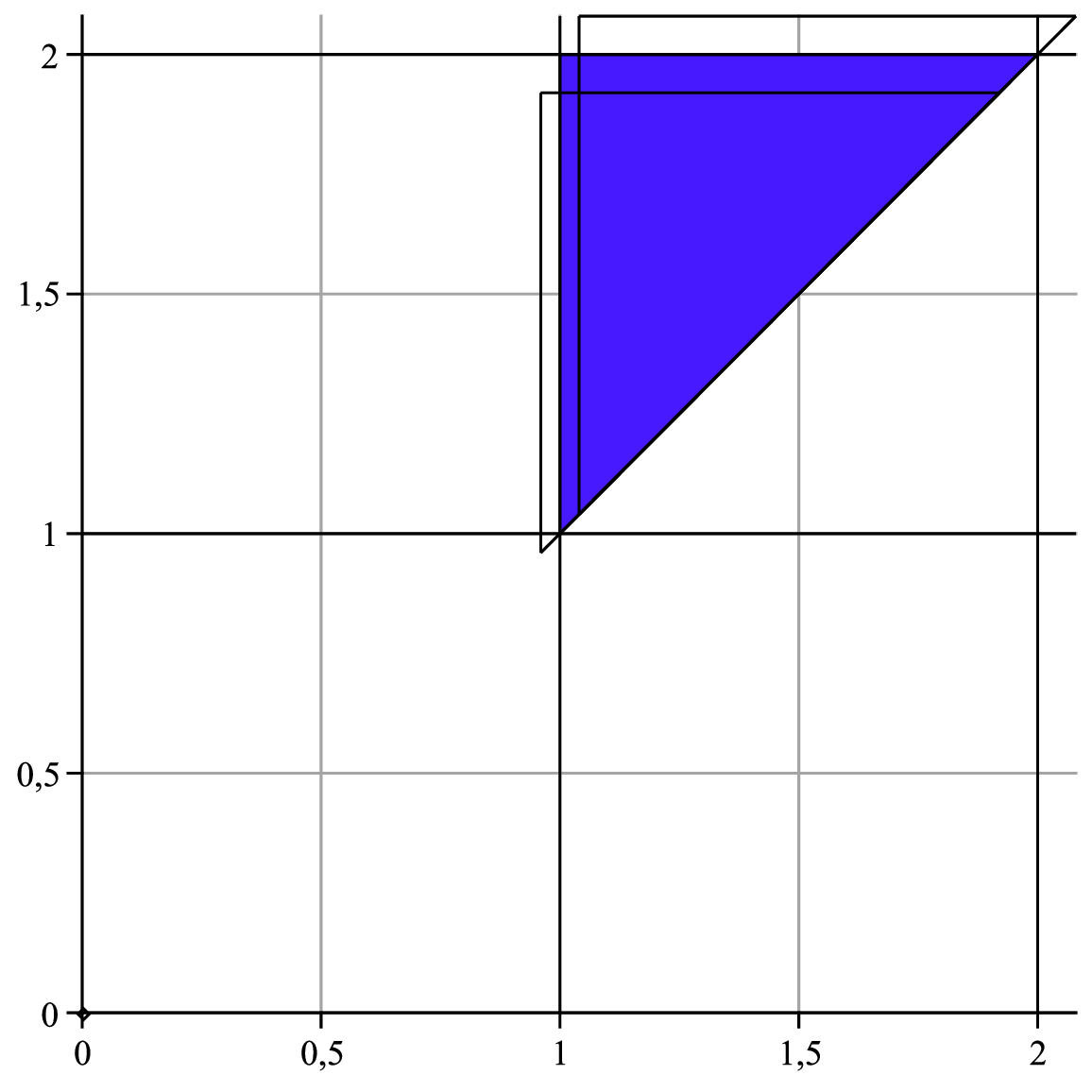} \includegraphics[width=6cm]{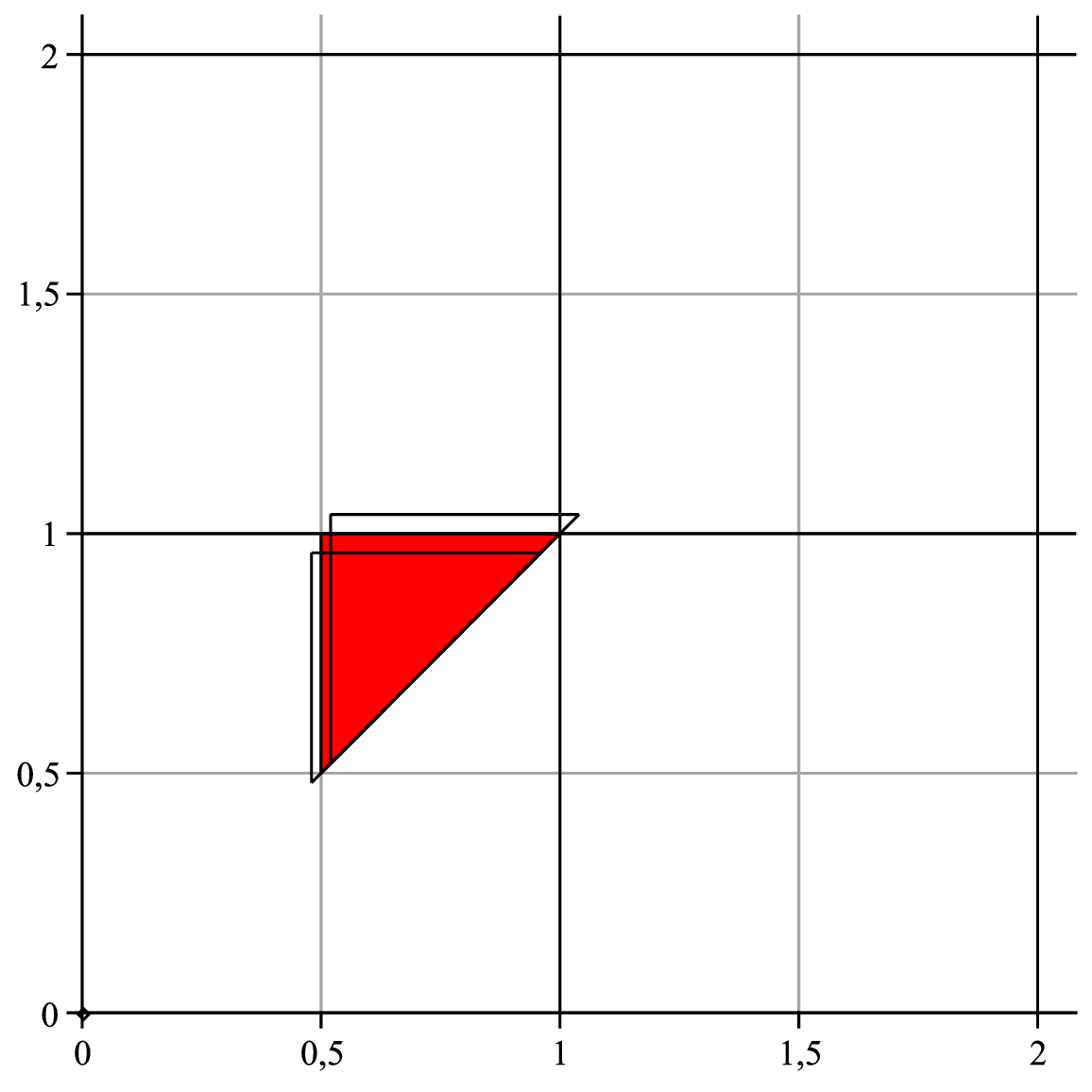}\\
 \caption{ In blue, the triangle with vertices $[1,1]$, $[1, 2]$, and $[2, 2]$, dilated by $t=1+\epsilon$ and $t=1-\epsilon$. In red, the same triangle, dilated by $t=\tfrac12$ and $t=\tfrac12\pm\epsilon$. }\label{fig:epsilon}
 \end{center}
\end{figure}
\begin{figure}
\begin{center}
  \includegraphics[width=6cm]{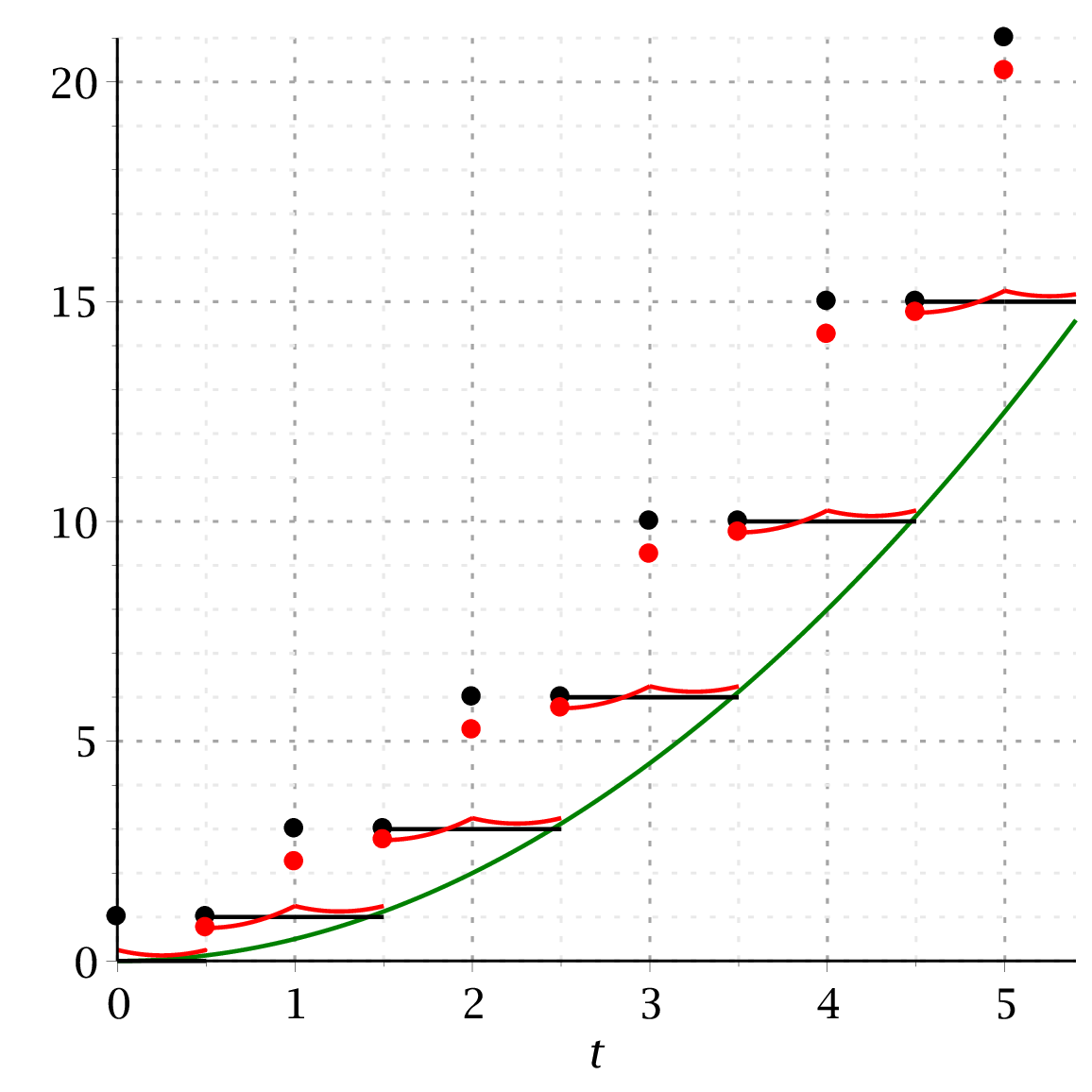}\includegraphics[width=6cm]{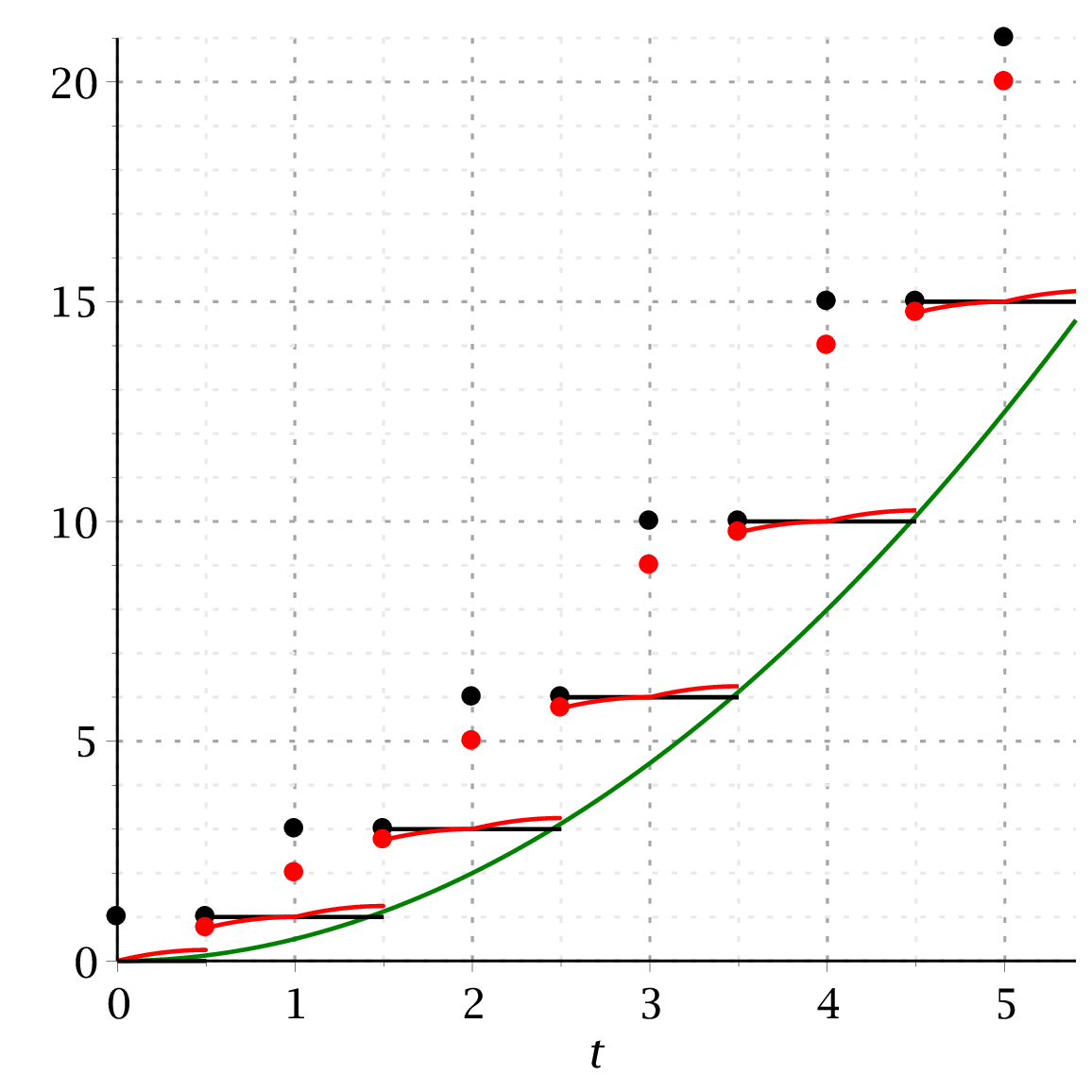}\\
 \caption{ Graphs of the quasi-polynomials
   $E^{k,\ConeByCone}(\p,1)(t)$  (\emph{left}) and  $E^{k,\Barvi}(\p,1)(t)$ (\emph{right}) 
   for the triangle~$\p$ with vertices $ [1,1], [1, 2], [2,
   2]$
   and $k = 0$ (\emph{green}), $k=1$ (\emph{red}), and $k=2$ (\emph{black}).}\label{fig:exemple-triangle}
 \end{center}
\end{figure}

In higher dimensions, the quasi-polynomials of a real variable $t$ which
  arise are too long to display.  We will only show some
  graphs
.
\begin{example}\label{ex:3-simplex}
  Figures \ref{fig:simplexDim3-cone-by-cone} and
  \ref{fig:simplexDim3-Barvinok} display the graphs of
  the quasi-polynomials
  $E^{k,\ConeByCone}(\p,1)(t)$ and $E^{k,\Barvi}(\p,1)(t)$ for the
  $3$-dimen\-sional simplex with vertices $$[0,1,1],[4,2,1],[1,1,2],[1,2,4],$$
  for $t\in [1.3,3.9]$. For $k=0$, both quasi-polynomials give the volume
  $t^3$ of the dilated simplex; for $k=3$, they are both equal to $S(t\p,1)$
  which gives the number of integral points. In this example, we see that the
  function $t\mapsto S(t\p,1)$ has discontinuities on any side (left, right or
  both).
\end{example}
\begin{figure}[p]
\begin{center}
  \includegraphics[width=8cm]{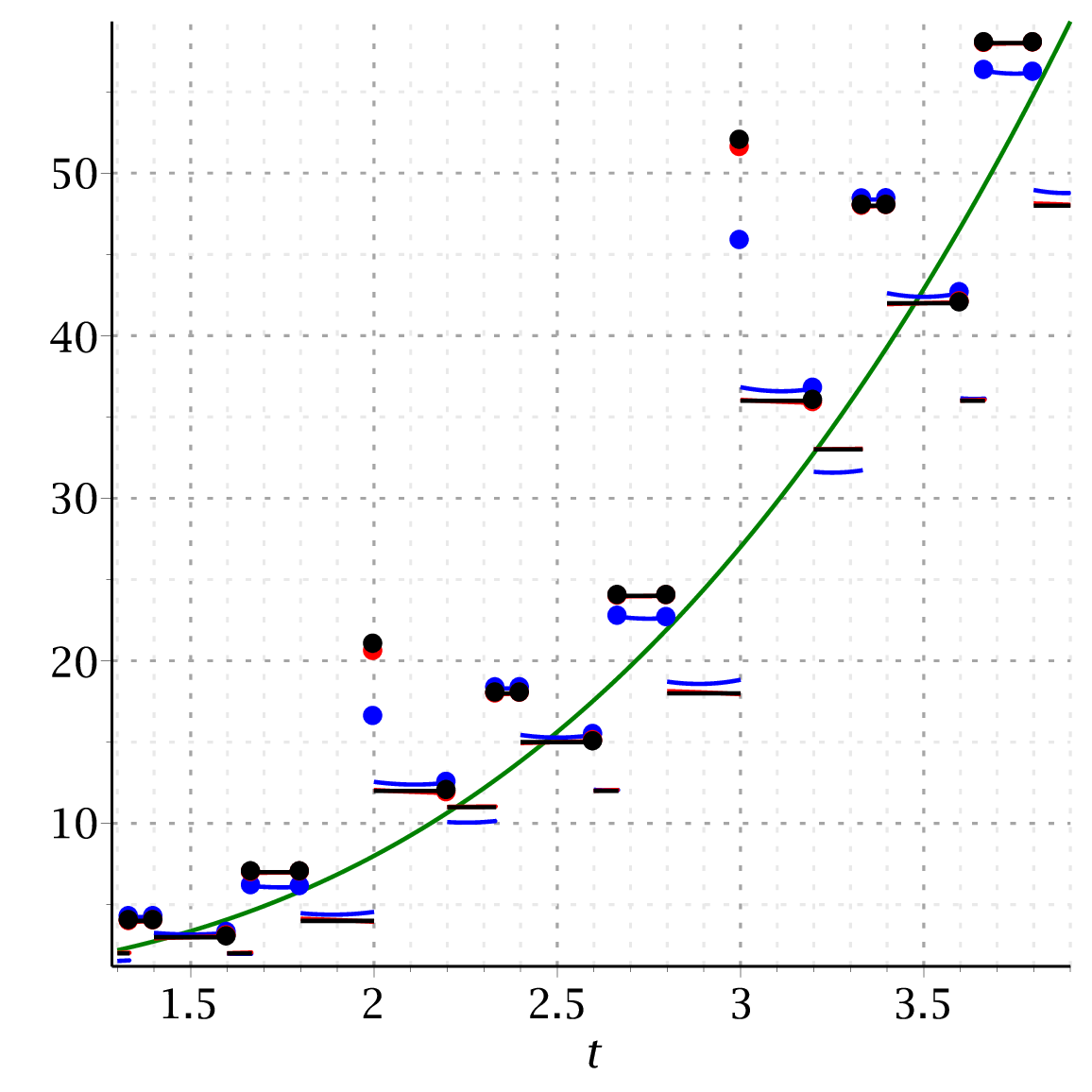}
 \caption{$E^{k,\ConeByCone}(\p,1)(t)$  for the 3-dimensional simplex with
   vertices $[0,1,1]$, $[4,2,1]$, $[1,1,2]$, and $[1,2,4]$ from Example~\ref{ex:3-simplex}, for $t\in [1.3,3.9]$
   and $k = 0$ (\emph{green}), $k=1$ (\emph{blue}), $k=2$ (\emph{red}), $k=3$
   (\emph{black}).}
 \label{fig:simplexDim3-cone-by-cone}
 \end{center}
\end{figure}
\begin{figure}[p]
\begin{center}
  \includegraphics[width=8cm]{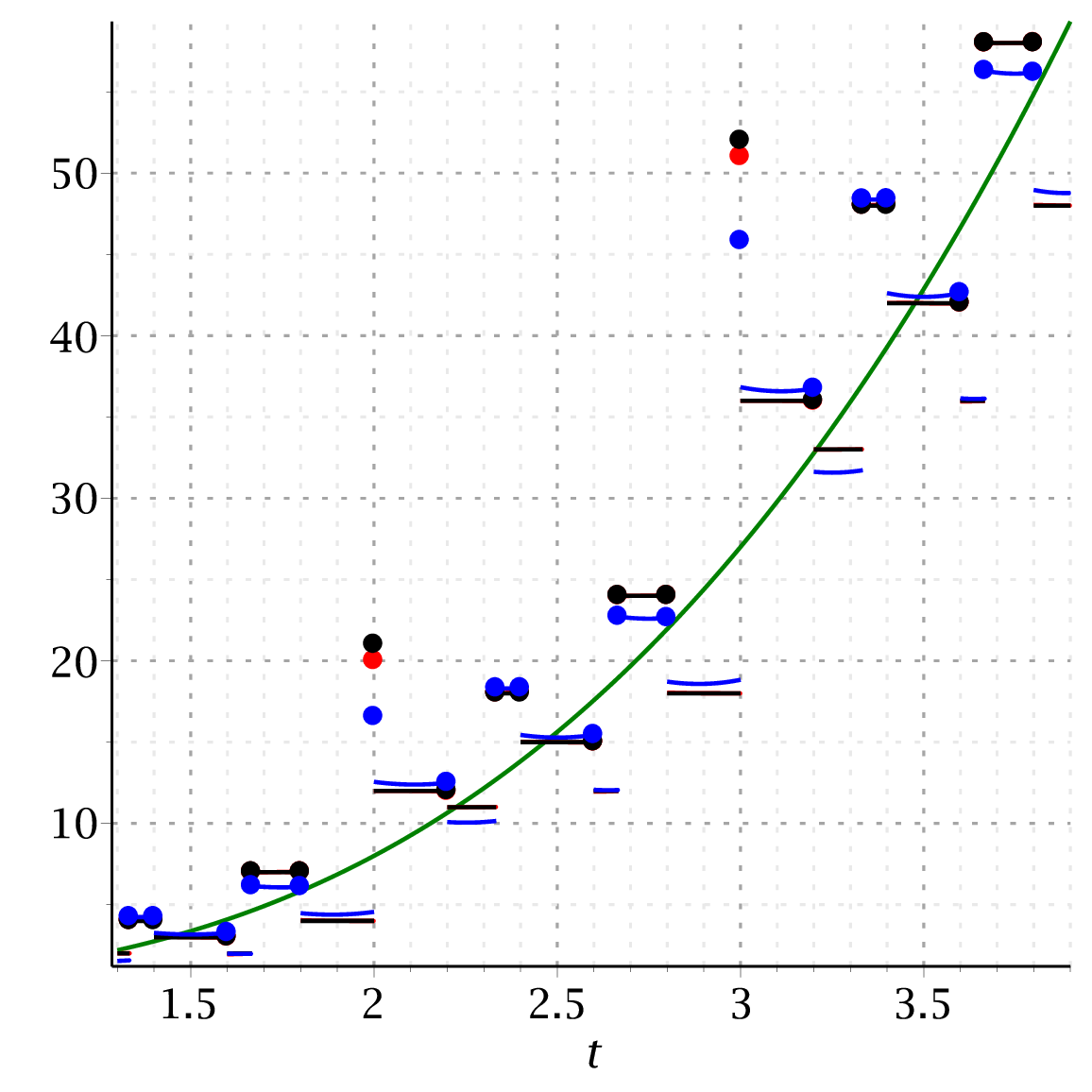}
 \caption{$E^{k,\Barvi}(\p,1)(t)$ for the same simplex as in  Figure \ref{fig:simplexDim3-cone-by-cone}}
 \label{fig:simplexDim3-Barvinok}
 \end{center}
\end{figure}

\clearpage

\section*{Acknowledgments}

 This article is part of a research project which was made possible by
 several meetings of the authors,
at the Centro di Ricerca Matematica Ennio De Giorgi of the Scuola
Normale Superiore, Pisa in 2009,  in a SQuaRE program at the
American Institute of Mathematics, Palo Alto, in July 2009,
September 2010, and February 2012,
 in the Research in Pairs program at
Mathematisches Forschungsinstitut Oberwolfach in March/April 2010,
and at the Institute for Mathematical Sciences (IMS) of the National
University of Singapore in November/December 2013.
The support of all four institutions is gratefully acknowledged.
V.~Baldoni was partially supported by a PRIN2009 grant.
J. De Loera was partially supported by grant DMS-0914107 of the National Science Foundation.
M.~K\"oppe was partially supported by grant DMS-0914873 of the National
Science Foundation.

\bibliographystyle{../amsabbrvurl}
\bibliography{../biblio}

\providecommand{\bysame}{\leavevmode\hbox to3em{\hrulefill}\thinspace}
\providecommand\ISBN{ISBN }
\providecommand{\MR}{\relax\ifhmode\unskip\space\fi MR }
\providecommand{\MRhref}[2]{%
  \href{http://www.ams.org/mathscinet-getitem?mr=#1}{#2}
}
\providecommand{\href}[2]{#2}
\begin{thebibliography}{10}

\bibitem{latteintegrale}
V.~Baldoni, N.~Berline, J.~De~Loera, B.~Dutra, M.~K{\"o}ppe, S.~Moreinis,
  G.~Pinto, M.~Vergne, and J.~Wu, \emph{A user's guide for {LattE} integrale
  v1.7.3}, Available from URL {\url{http://www.math.ucdavis.edu/~latte/}},
  2015.

\bibitem{baldoni-berline-deloera-koeppe-vergne:integration}
V.~Baldoni, N.~Berline, J.~A. De~Loera, M.~K{\"o}ppe, and M.~Vergne, \emph{How
  to integrate a polynomial over a simplex}, Mathematics of Computation
  \textbf{80} (2011), no.~273, 297--325, \href
  {https://doi.org/10.1090/S0025-5718-2010-02378-6}
  {\path{https://doi.org/10.1090/S0025-5718-2010-02378-6}}.

\bibitem{so-called-paper-1}
\bysame, \emph{Computation of the highest coefficients of weighted {E}hrhart
  quasi-polynomials of rational polyhedra}, Foundations of Computational
  Mathematics \textbf{12} (2012), 435--469, \href
  {https://doi.org/10.1007/s10208-011-9106-4}
  {\path{https://doi.org/10.1007/s10208-011-9106-4}}.

\bibitem{SLII2014}
\bysame, \emph{Intermediate sums on polyhedra {II}: Bidegree and {P}oisson
  formula}, Mathematika \textbf{62} (2016), 653--684, \href
  {https://doi.org/10.1112/S0025579315000418}
  {\path{https://doi.org/10.1112/S0025579315000418}}.

\bibitem{so-called-paper-2}
V.~Baldoni, N.~Berline, M.~K{\"o}ppe, and M.~Vergne, \emph{Intermediate sums on
  polyhedra: Computation and real {E}hrhart theory}, Mathematika \textbf{59}
  (2013), no.~1, 1--22, \href {https://doi.org/10.1112/S0025579312000101}
  {\path{https://doi.org/10.1112/S0025579312000101}}.

\bibitem{bar}
A.~I. Barvinok, \emph{Polynomial time algorithm for counting integral points in
  polyhedra when the dimension is fixed}, Mathematics of Operations Research
  \textbf{19} (1994), 769--779, \href {https://doi.org/10.1287/moor.19.4.769}
  {\path{https://doi.org/10.1287/moor.19.4.769}}.

\bibitem{barvinok-2006-ehrhart-quasipolynomial}
\bysame, \emph{Computing the {E}hrhart quasi-polynomial of a rational simplex},
  Math. Comp. \textbf{75} (2006), no.~255, 1449--1466.

\bibitem{barvinokzurichbook}
\bysame, \emph{Integer points in polyhedra}, Z\"urich Lectures in Advanced
  Mathematics, European Mathematical Society (EMS), Z\"urich, Switzerland,
  2008.

\bibitem{beck:multidimensional-reciprocity}
M.~Beck, \emph{Multidimensional {E}hrhart reciprocity}, Journal of
  Combinatorial Theory, Series A \textbf{97} (2002), no.~1, 187--194, \href
  {https://doi.org/10.1006/jcta.2001.3220}
  {\path{https://doi.org/10.1006/jcta.2001.3220}}.

\bibitem{BV-2011}
N.~Berline and M.~Vergne, \emph{Analytic continuation of a parametric polytope
  and wall-crossing}, Configuration Spaces: Geometry, Combinatorics and
  Topology (A.~Bj{\"o}rner, F.~Cohen, C.~De~Concini, C.~Procesi, and
  M.~Salvetti, eds.), CRM Series, Edizioni della Normale, Pisa, 2012, \href
  {https://doi.org/10.1007/978-88-7642-431-1_6}
  {\path{https://doi.org/10.1007/978-88-7642-431-1_6}},
  \ISBN{978-88-7642-430-4}.

\bibitem{MR1243770}
A.~Bj{\"o}rner and L.~Lov{\'a}sz, \emph{Linear decision trees, subspace
  arrangements and {M}\"obius functions}, J. Amer. Math. Soc. \textbf{7}
  (1994), no.~3, 677--706, \href {https://doi.org/10.2307/2152788}
  {\path{https://doi.org/10.2307/2152788}}, \MR{1243770 (95e:52024)}.

\bibitem{Brion1997residue}
M.~Brion and M.~Vergne, \emph{Residue formulae, vector partition functions and
  lattice points in rational polytopes}, J. Amer. Math. Soc. \textbf{10}
  (1997), no.~4, 797--833, \href
  {https://doi.org/10.1090/S0894-0347-97-00242-7}
  {\path{https://doi.org/10.1090/S0894-0347-97-00242-7}}, \MR{1446364
  (98e:52008)}.

\bibitem{Clauss1998parametric}
P.~Clauss and V.~Loechner, \emph{Parametric analysis of polyhedral iteration
  spaces}, Journal of {VLSI} Signal Processing \textbf{19} (1998), no.~2,
  179--194, \href {https://doi.org/10.1023/A:1008069920230}
  {\path{https://doi.org/10.1023/A:1008069920230}}.

\bibitem{HenkLinke}
M.~Henk and E.~Linke, \emph{Lattice points in vector-dilated polytopes},
  e-print arXiv:1204.6142 [math.MG], 2012.

\bibitem{MR1671451}
A.~Knutson and T.~Tao, \emph{The honeycomb model of {${\rm GL}_n({\bf C})$}
  tensor products. {I}. {P}roof of the saturation conjecture}, J. Amer. Math.
  Soc. \textbf{12} (1999), no.~4, 1055--1090, \href
  {https://doi.org/10.1090/S0894-0347-99-00299-4}
  {\path{https://doi.org/10.1090/S0894-0347-99-00299-4}}, \MR{1671451}.

\bibitem{koeppe-verdoolaege:parametric}
M.~K\"oppe and S.~Verdoolaege, \emph{Computing parametric rational generating
  functions with a primal {B}arvinok algorithm}, The Electronic Journal of
  Combinatorics \textbf{15} (2008), 1--19, \#R16.

\bibitem{linke:rational-ehrhart}
E.~Linke, \emph{Rational {E}hrhart quasi-polynomials}, Journal of Combinatorial
  Theory, Series A \textbf{118} (2011), no.~7, 1966--1978, \href
  {https://doi.org/10.1016/j.jcta.2011.03.007}
  {\path{https://doi.org/10.1016/j.jcta.2011.03.007}}.

\bibitem{MR0326574}
P.~McMullen, \emph{Representations of polytopes and polyhedral sets},
  Geometriae Dedicata \textbf{2} (1973), 83--99, \href
  {https://doi.org/10.1007/BF00149284}
  {\path{https://doi.org/10.1007/BF00149284}}, \MR{0326574 (48 \#4917)}.

\bibitem{Pemantle:2013:ACS:2505450}
R.~Pemantle and M.~C. Wilson, \emph{Analytic combinatorics in several
  variables}, Cambridge University Press, New York, NY, USA, 2013,
  \ISBN{1107031575, 9781107031579}.

\bibitem{Verdoolaege2005PhD}
S.~Verdoolaege, \emph{Incremental loop transformations and enumeration of
  parametric sets}, Ph.D. thesis, Department of Computer Science, K.U. Leuven,
  Leuven, Belgium, April 2005.

\bibitem{Verdoolaege2007parametric}
S.~Verdoolaege, R.~Seghir, K.~Beyls, V.~Loechner, and M.~Bruynooghe,
  \emph{Counting integer points in parametric polytopes using {Barvinok}'s
  rational functions}, Algorithmica \textbf{48} (2007), no.~1, 37--66, \href
  {https://doi.org/10.1007/s00453-006-1231-0}
  {\path{https://doi.org/10.1007/s00453-006-1231-0}}.

\end{thebibliography}
\end{document}